\documentclass[a4paper,10pt
 ,final
]{article}%

\usepackage[dvipsnames]{xcolor}
\usepackage{ifdraft}
\ifdraft{
	\color[RGB]{63,63,63}
	\usepackage[notref]{showkeys}
	\usepackage{todonotes}
}

\usepackage{xr-hyper}
\usepackage[pagebackref, colorlinks, citecolor=PineGreen, linkcolor=PineGreen]{hyperref}
\hypersetup{
	final,
	pdftitle={On the homotopy theory of equivariant colored operads},
	pdfauthor={Bonventre, P. and Pereira, L. A.},
	linktoc=page
}

\usepackage{amsmath, amsthm}

\usepackage[latin1]{inputenc}%
\usepackage{MnSymbol}

\usepackage[normalem]{ulem}
\usepackage{bbm}

\DeclareMathAlphabet\mathbb{U}{msb}{m}{n}
\usepackage{upgreek}
\usepackage{mathrsfs}

\usepackage[normalem]{ulem}

\usepackage[inline,shortlabels]{enumitem}
\setenumerate{label=(\roman*)}

\usepackage[final]{microtype}
\usepackage{relsize}
\usepackage{geometry}

\usepackage{tikz}%
\usetikzlibrary{matrix,arrows,decorations.pathmorphing,
cd,patterns,calc}
\tikzset{%
  treenode/.style = {shape=rectangle, rounded corners,%
                     draw, align=center,%
                     top color=white, bottom color=blue!20},%
  root/.style     = {treenode, font=\Large, bottom color=red!30},%
  env/.style      = {treenode, font=\ttfamily\normalsize},%
  dummy/.style    = {circle,draw,inner sep=0pt,minimum size=2mm}%
}%

\usetikzlibrary[decorations.pathreplacing]

\makeatletter

\def\@testdef #1#2#3{%
  \def\reserved@a{#3}\expandafter \ifx \csname #1@#2\endcsname
  \reserved@a  \else
  \typeout{^^Jlabel #2 changed:^^J%
    \meaning\reserved@a^^J%
    \expandafter\meaning\csname #1@#2\endcsname^^J}%
  \@tempswatrue \fi}

\makeatother

\numberwithin{equation}{section} 
\numberwithin{figure}{section}

\usepackage{mathtools}
\mathtoolsset{showonlyrefs,showmanualtags} 

\newtheorem{theorem}[equation]{Theorem}%
\newtheorem*{theorem*}{Theorem}%
\newtheorem{lemma}[equation]{Lemma}%
\newtheorem{proposition}[equation]{Proposition}%
\newtheorem{corollary}[equation]{Corollary}%
\newtheorem{conjecture}[equation]{Conjecture}%
\newtheorem*{conjecture*}{Conjecture}%
\newtheorem{innercustomgeneric}{\customgenericname}
\providecommand{\customgenericname}{}
\newcommand{\newcustomtheorem}[2]{%
  \newenvironment{#1}[1]
  {%
   \renewcommand\customgenericname{#2}%
   \renewcommand\theinnercustomgeneric{##1}%
   \innercustomgeneric
  }
  {\endinnercustomgeneric}
}

\newcustomtheorem{customthm}{Theorem}
\newcustomtheorem{customcor}{Corollary}

\theoremstyle{definition} 
\newtheorem{definition}[equation]{Definition}%
\newtheorem*{definition*}{Definition}%
\newtheorem{example}[equation]{Example}%
\newtheorem{remark}[equation]{Remark}%
\newtheorem{notation}[equation]{Notation}%
\newcommand{\set}[1]{\left\{#1\right\}}%
\newcommand{\sets}[2]{\left\{ #1 \;|\; #2\right\}}%
\newcommand{\longto}{\longrightarrow}%

\usepackage{harpoon}
\newcommand{\vect}[1]{\text{\overrightharp{\ensuremath{#1}}}}

\newcommand{\Sym}{\ensuremath{\mathsf{Sym}}}%

\newcommand{\Top}{\ensuremath{\mathsf{Top}}}
\newcommand{\sSet}{\ensuremath{\mathsf{sSet}}}%
\newcommand{\Cat}{\mathsf{Cat}}

\newcommand{\Op}{\mathsf{Op}}%
\newcommand{\sOp}{\ensuremath{\mathsf{sOp}}}%
\DeclareMathOperator{\colim}{colim}%
\DeclareMathOperator{\Ho}{Ho}
\newcommand{\F}{\ensuremath{\mathcal F}}
\newcommand{\V}{\ensuremath{\mathcal V}}

\renewcommand{\O}{\ensuremath{\mathcal O}}
\renewcommand{\P}{\ensuremath{\mathcal P}}
\newcommand{\C}{\ensuremath{\mathcal C}}

\newcommand{\1}{\ensuremath{\mathbbm 1}}

\title{On the homotopy theory of equivariant colored operads}

\author{Peter Bonventre, Lu\'is A. Pereira}%

\date{\today}

\begin{document}

\maketitle

\begin{abstract}
	We build model structures 
	on the category of equivariant simplicial operads
	with weak equivalences determined by families of subgroups,
	in the context of operads with a varying set of colors
	(and building on the fixed color model structures in the prequel).
	In particular, by specifying to the family of graph subgroups 
	(or, more generally, one of the indexing systems of Blumberg-Hill),
	we obtain model structures on the category of equivariant simplicial operads
	whose weak equivalences are determined by norm map data.
\end{abstract}

\tableofcontents

\section{Introduction}

This paper is a direct sequel to \cite{BP_FCOP},
assembling the model structures on the categories
$\sOp_{\mathfrak C}^G$
of equivariant simplicial operads with a fixed 
$G$-set of colors $\mathfrak C$
\cite[Thm. \ref{OC-THMI}]{BP_FCOP}
to a Dwyer-Kan style model structure on the category
$\sOp_\bullet^G$
of equivariant simplicial operads with any set of colors.

More broadly, this paper follows
\cite{Per18,BP21,BP20,BP_FCOP}
as part of a larger project culminating in 
the sequel \cite{BP_TAS}
with the existence of a Quillen equivalence
\begin{equation}\label{BPMAINTHM_EQ}
\mathsf{dSet}^G \rightleftarrows \mathsf{sOp}_{\bullet}^G
\end{equation}
where $\mathsf{dSet}^G$ is the category
of equivariant dendroidal sets with the model structure
from \cite{Per18}
and 
$\mathsf{sOp}_{\bullet}^G$
has the model structure from 
Theorem \ref{THMA} herein,
thereby generalizing the 
analogous Cisinski-Moerdijk-Weiss project
\cite{MW09,CM11,CM13a,CM13b} to the equivariant context.

Much of the challenge in building the model structures 
in \eqref{BPMAINTHM_EQ} 
comes from the fact that they model the homotopy theory of
equivariant operads \emph{with norm maps},
which are an extra piece of data
not present non-equivariantly
(and recalled in \eqref{GRAPHEQ EQ},\eqref{ALGNORM EQ} below),
the importance of which was made clear by
Hill, Hopkins, Ravenel 
in their solution to the Kervaire invariant one problem \cite{HHR16}.
Notably, the presence of norm maps causes the existence
of the model structures in \eqref{BPMAINTHM_EQ}
not to be a formal consequence of the 
existence of their non-equivariant analogues
(this is further discussed at the end of this introduction).

We now briefly recall 
the notion of norm maps.
For simplicity,
fix a finite group $G$
and consider the category
$\mathsf{sOp}^{G}_{\**} = \mathsf{Op}_{\**}(\mathsf{sSet}^G)$
of single colored (symmetric) operads
on $G$-equivariant simplicial sets $\mathsf{sSet}^G$.
Note that, for $\O \in \mathsf{sOp}^{G}_{\**}$,
the $n$-th operadic level $\O(n)$ has both a $\Sigma_n$-action and a $G$-action, commuting with each other, or, equivalently, 
a $G \times \Sigma_n$-action.
A key upshot
of Blumberg and Hill's work \cite{BH15}
is then that 
the preferred notion of weak equivalence in $\mathsf{sOp}^{G}_{\**}$
is that of \emph{graph equivalence},
i.e. those maps 
$\O \to \mathcal{P}$
such that the fixed point maps
\begin{equation}\label{GRAPHEQ EQ}
\O(n)^{\Gamma} \xrightarrow{\sim} \mathcal{P}(n)^{\Gamma}
\qquad
\text{for }
\Gamma \leq G \times \Sigma_n
\text{ such that }
\Gamma \cap \Sigma_n = \**
\end{equation}
are Kan equivalences in $\mathsf{sSet}$.
Here, the term ``graph'' comes from a neat characterization of the $\Gamma$
as in \eqref{GRAPHEQ EQ}:
such a $\Gamma$ must be the graph of a partial homomorphism
$\phi \colon H \to \Sigma_n$ for some subgroup $H \leq G$,
i.e.
$\Gamma = \sets{(h,\phi(h))}{h \in H}$. 
Note that one then has a canonical isomorphism $\Gamma \simeq H$.
Briefly, the need to consider such \emph{graph subgroups} $\Gamma$ comes from the study of algebras.
Suppose $X \in \mathsf{sSet}^G$ is an algebra over
$\O$,
so that one has algebra multiplication maps
as on the left below
\begin{equation}\label{ALGNORM EQ}
\O(n) \times X^n \to X
\qquad \qquad
\O(n)^{\Gamma} \times N_{\Gamma}X \to X
\end{equation}
that need to be 
$G \times \Sigma_n$-equivariant
(the target $X$ is given the trivial $\Sigma_n$-action).
One then has induced 
$H$-equivariant maps 
on the right in \eqref{ALGNORM EQ},
where the \emph{norm object} $N_{\Gamma} X$
denotes $X^{n}$ with the $H$-action given by $\Gamma \simeq H$.
In particular, each point
$\rho \in \O(n)^{\Gamma}$
encodes a $H$-equivariant \emph{norm map}
$\rho \colon N_{\Gamma} X \to X$,  
and such maps are a key piece of data
for algebras. 
Thus, the advantage of the graph equivalences \eqref{GRAPHEQ EQ}
is that equivalent operads have
equivalent ``spaces of norm maps''.

In the single colored case,
the existence of a model structure 
on $\mathsf{sOp}^G_{\**}$
with weak equivalences the graph equivalences
in \eqref{GRAPHEQ EQ}
was established in both 
\cite[Thm. I]{BP21} and \cite[Thm. 3.1]{GW18}.

Our main result, Theorem \ref{THMA}, 
then extends the graph equivalence model structures of
\cite[Thm. I]{BP21}, \cite[Thm. 3.1]{GW18}
from the context of single colored operads 
to the context of operads with a varying set of colors
(also known as multicategories).
Alternatively, one may view Theorem \ref{THMA} as extending the well-known Dwyer-Kan model structures on all colored operads
in \cite[Thm 1.14]{CM13b}, \cite[Thm. 2]{Rob} (or more generally \cite{Cav}, \cite{Yau})
from the non-equivariant context to the 
equivariant context.

It is worth noting that 
Theorem \ref{THMA} is non-formal,
as discussed below
(a similar discussion for the 
single color and fixed color contexts
can be found in the introduction to \cite{BP_FCOP}).

First,
we note that our category of interest
is the category 
$\mathsf{sOp}^G_{\bullet}=
\left(\mathsf{sOp}_{\bullet}\right)^G$
of $G$-objects on the usual category
$\mathsf{sOp}_{\bullet}$
of simplicial (symmetric) operads studied in 
\cite{CM13b},\cite{Rob},\cite{Cav}.
In particular, the group $G$ is allowed to act non-trivially
on the objects of an equivariant operad
$\O \in \mathsf{sOp}^G_{\bullet}$.
By contrast, in the category
$\mathsf{Op}_{\bullet}(\mathsf{sSet}^G)$
of colored operads in $\mathsf{sSet}^G$
the group $G$ never acts on objects (or, equivalently, acts trivially).
As such, one only has a proper inclusion
$\mathsf{Op}_{\bullet}(\mathsf{sSet}^G)
\subsetneq \mathsf{sOp}_{\bullet}^G$,
so a model structure on 
$\mathsf{sOp}_{\bullet}^G$
can not be built using the enriched colored operad results of \cite{Cav}.

Alternatively, one could also try to build a model structure
on $\mathsf{sOp}^G_{\bullet}$
by applying the formalism in
\cite[Prop 2.6]{Ste16},
which builds model structures on $G$-objects,
to the model structure on $\mathsf{sOp}_\bullet$
from \cite{CM13b},\cite{Rob}.
However, this approach does not produce
the desired notion of weak equivalence suggested by graph subgroups (more precisely, this approach ignores
``non-trivial norm maps'', i.e.
it ignores any graph subgroups $\Gamma$ in \eqref{GRAPHEQ EQ}
associated to non-trivial homomorphisms $\phi\colon H \to \Sigma_n$).
It is worth noting that the issue with this latter approach is intrinsically operadic
and does not occur when working with categories.
Indeed, the inclusion
$\mathsf{sCat}^G_{\bullet} \subset 
\mathsf{sOp}^G_{\bullet}$
of colored $G$-categories into 
colored $G$-operads identifies
$\mathsf{sCat}^G_{\bullet} = 
\mathsf{sOp}^G_{\bullet} \downarrow \**$
as the overcategory over the terminal category $\**$.
Hence, our model structure on 
$\mathsf{sOp}^G_{\bullet}$
induces a model structure on 
$\mathsf{sCat}^G_{\bullet}$ which \emph{does}
in fact coincide with the model structure obtained by applying
\cite[Prop. 2.6]{Ste16}
to the usual model structure on $\mathsf{sCat}_{\bullet}$.

\subsection{Main Results}

Before stating our main result, Theorem \ref{THMA},
we require some preliminary setup.

First, as noted in \eqref{GRAPHEQ EQ},
our preferred notion of equivalence of equivariant operads is determined by the graph subgroups.
However, as in \cite{BP_FCOP}, we will work with general collections of subgroups.

\begin{definition}\label{FAM1ST DEF}
	A \emph{$(G,\Sigma)$-family} is a
	a collection
	$\mathcal{F} = \{\mathcal{F}_n\}_{n \leq 0}$,
	where each $\mathcal{F}_n$
	is a family of subgroups of $G \times \Sigma_n^{op}$.
\end{definition}
The use of $\Sigma_n^{op}$ rather than $\Sigma_n$
in Definition \ref{FAM1ST DEF} 
(and throughout) 
is motivated by regarding $\Sigma$
as the category of corollas (i.e. trees with a single node; 
see \eqref{OPSSYMS EQ},\eqref{CSYM EQ1},\eqref{CSYM EQ2}),
and the fact that the dendroidal nerve \cite[\S 1]{MW07} of an operad is contravariant on the category of trees.

In contrast to \cite[Thm. \ref{OC-THMI}]{BP_FCOP},
Theorem \ref{THMA} requires a minor restriction on the 
$(G,\Sigma)$-family $\mathcal{F}$.
Noting that $\F_1$ is a family of subgroups of 
$G \times \Sigma_1^{op} \simeq G$,
we regard 
$H \in \F_1$ as a subgroup $H \leq G$.

\begin{definition}\label{FAMRESUNI DEF}
	Write $G \times \Sigma_n^{op} \xrightarrow{\pi_n} 
	G$
	for the natural projection.
	
	We say that a $(G,\Sigma)$-family $\F$ 
	\emph{has enough units}	if,
	for all $H \in \F_n$, $n\geq 0$,
	it is $\pi_n(H) \in \F_1$.
\end{definition}
 The motivation for the condition
in Definition \ref{FAMRESUNI DEF} is discussed in Remark \ref{WHYEU REM}.

Next, recall that a colored operad $\O$	
with color set $\mathfrak{C}$ has levels 	
$	
\O(\vect{C})=	
\O(\mathfrak{c}_1,\cdots,\mathfrak{c}_n;\mathfrak{c}_0)$	
indexed by tuples	
$\vect{C} = (\mathfrak{c}_1,\cdots,\mathfrak{c}_n;\mathfrak{c}_0) = (\mathfrak c_i)_{0 \leq i \leq n}$	
of elements in $\mathfrak{C}$, called \emph{$\mathfrak{C}$-profiles}.	
If the operad is symmetric one has associative and unital isomorphisms	
$	
\O(\mathfrak{c}_1,\cdots,\mathfrak{c}_n;\mathfrak{c}_0) \to 	
\O(\mathfrak{c}_{\sigma(1)},\cdots,\mathfrak{c}_{\sigma(n)};\mathfrak{c}_0)	
$	
for each permutation $\sigma \in \Sigma_n$.	
On the other hand, if 	
$\O \in \mathsf{sOp}^G_{\mathfrak{C}}$	
is a $G$-equivariant operad, 	
the color set $\mathfrak{C}$ is itself a $G$-set,	
and one similarly has associative and unital isomorphisms	
$	
\O(\mathfrak{c}_1,\cdots,\mathfrak{c}_n;\mathfrak{c}_0) \to 	
\O(g\mathfrak{c}_{1},\cdots,g\mathfrak{c}_{n};g\mathfrak{c}_0)	
$ for $g \in G$.
All together, one thus has isomorphisms	
\begin{equation}\label{OPSSYMS EQ}	
\O(\mathfrak{c}_1,\cdots,\mathfrak{c}_n;\mathfrak{c}_0)	
\to 	
\O(g \mathfrak{c}_{\sigma(1)},\cdots,g \mathfrak{c}_{\sigma(n)};g\mathfrak{c}_0)	
\end{equation}	
for $(g,\sigma) \in G \times \Sigma_n^{op}$.	
Note that these isomorphisms 	
are associated with an action of 	
$G \times \Sigma_n^{op}$	
on the set $\mathfrak{C}^{n+1}$ of $n$-ary profiles via	
$(g,\sigma) (\mathfrak{c}_i)_{0\leq i \leq n}	
= (g \mathfrak{c}_{\sigma(i)})_{0\leq i \leq n}$,	
where we implicitly write $\sigma(0)=0$.

As such, we say that a subgroup 	
$\Lambda \leq G \times \Sigma_n^{op}$	
\emph{stabilizes} a profile $\vect{C}=(\mathfrak{c}_i)_{0 \leq i \leq n}$ if,	
for any $(g,\sigma) \in \Lambda$,	
it is 	
$\mathfrak{c}_i = g \mathfrak{c}_{\sigma(i)}$ for $0 \leq i \leq n$.	
Note that,	
for $\O \in \mathsf{sOp}^G_{\mathfrak{C}}$,	
the level $\O(\vect{C})$ has a $\Lambda$-action.

Lastly, we need a notion of 
\emph{essential surjectivity}.
For this purpose, we recall the following construction, 
which associates to a $\V$-category $\mathcal{C}$
a \emph{category of components} $\pi_0 \mathcal{C}$.

\begin{definition}
	Suppose $\V$ is as in Theorem \ref{THMA}
	(in particular, $\V$ has a cofibrant unit).
	Given $\mathcal C \in \Cat_{\mathfrak{C}}(\V)$, 
	define 
	$\pi_0 \mathcal C \in \Cat_{\mathfrak{C}} = \Cat_{\mathfrak{C}}(\mathsf{Set})$ 
	to be the ordinary category with the same objects and
        \[
                \pi_0(\mathcal{C})(c,d)=
                \Ho(\V)(1_\V, \mathcal C(c,c'))=
                [1_\V, \mathcal{C}_f(c,c')]
        \]
        where $[-,-]$ denotes homotopy equivalence classes of maps
        and $\mathcal{C}_f$ denotes a fibrant replacement of
        $\mathcal C$ in the \emph{canonical model structure} on $\Cat_{\mathfrak{C}}(\V)$
        \cite{BM13} (also, see Remarks 
        \ref{RESTTOCATS REM} and \ref{GTRIV REM}).
\end{definition}

Further writing
$j^{\**} \colon 
\mathsf{Op}^G_{\bullet}(\V) \to \mathsf{Cat}^G_{\bullet}(\V)$
for the ``underlying category''
functor which forgets the non-unary operations,
we can now state the main result.

\begin{customthm}{A}\label{THMA}
Fix a finite group $G$ 
and a $(G,\Sigma)$-family $\F = \set{\F_n}_{n \geq 0}$
that has enough units.

Then there exists a cofibrantly generated model structure on
$\mathsf{sOp}^G_{\bullet} = 
\mathsf{Op}^G_{\bullet}(\mathsf{sSet})$,
which we call the \emph{$\F$-model structure},
such that a map
$F\colon \mathcal{O} \to \mathcal{P}$
is a weak equivalence (resp. trivial fibration) if
\begin{itemize}
\item the maps
\begin{equation}\label{THMIII1ST EQ}
	\O(\vect{C})^{\Lambda} \to \mathcal{P}(F(\vect{C}))^{\Lambda}
\end{equation}
are Kan equivalences (trivial Kan fibrations)
in $\mathsf{sSet}$
for all $\mathfrak{C}$-profiles $\vect{C}$
and $\Lambda \in \F$ that stabilizes $\vect{C}$; and
\item 
the maps of unenriched categories
\begin{equation}\label{THMIII2ND EQ}
\pi_0 j^{\**} \O^H 
\to 
\pi_0 j^{\**} \mathcal{P}^H 
\end{equation}
are essentially surjective (surjective on objects)
for all $H \in \F_1$.
\end{itemize}
More generally, a $\F$-model structure on 
$\Op^G_{\bullet}(\V)$
with weak equivalences/trivial fibrations as in 
\eqref{THMIII1ST EQ},\eqref{THMIII2ND EQ}
exists provided $(\V,\otimes)$ satisfies:
\begin{enumerate}[label = (\roman*)]
	\item $\V$ is a cofibrantly generated model category
	such that the domains of the generating (trivial) cofibrations are small;
	\item for any finite group $G$, the $G$-object category $\V^G$ admits the genuine model structure 
	\cite[Def. \ref{OC-GENMOD DEF}]{BP_FCOP};
	\item $(\V, \otimes)$ is a closed symmetric monoidal model category with cofibrant unit;
	\item $(\V, \otimes)$ satisfies the global monoid axiom \cite[Def. \ref{OC-GLOBMONAX_DEF}]{BP_FCOP};
	\item $(\V, \otimes)$ has cofibrant symmetric pushout powers \cite[Def. \ref{OC-CSPP_DEF}]{BP_FCOP}.
	\item[(vi)] $\V$ is right proper;
	\item[(vii)]
	for any finite group $G$, fixed points
	$(-)^{G} \colon \V^G \to \V$
	send genuine trivial cofibrations 
	(cf. \cite[Def. \ref{OC-GENMOD DEF}]{BP_FCOP})
	to trivial cofibrations;
	\item[(viii)] $(\V, \otimes)$ has a generating set of intervals
	(Definition \ref{INTGENSET DEF}).
\end{enumerate}
\end{customthm}

Conditions (i) through (v) above  
are the conditions in \cite[Thm. \ref{OC-THMI}]{BP_FCOP},
which builds the model structures on
fixed color operads
$\Op^G_{\mathfrak{C}}(\V)$, 
one of the key ingredients in the proof
of Theorem \ref{THMA}.
In particular,
herein the technical conditions
(iv),(v) will only be needed to cite
results in \cite{BP_FCOP}.

\begin{remark}
Maps satisfying both of the weak equivalence conditions in 
\eqref{THMIII1ST EQ},\eqref{THMIII2ND EQ}
are called \emph{Dwyer-Kan equivalences},
while maps satisfying only \eqref{THMIII1ST EQ}
are called \emph{local equivalences}.
\end{remark}

\begin{remark}\label{WHYEU REM}
The ``enough units'' condition in Definition \ref{FAMRESUNI DEF}
ensures compatibility of the \emph{local equivalences}
in \eqref{THMIII1ST EQ}
with the essential surjectivity in	\eqref{THMIII2ND EQ}.
Informally, this guarantees that the spaces 
$\O(\vect{C})^{\Lambda}$
are homotopically well-behaved when replacing colors 
in $\vect{C}$
(for details, see \S \ref{HMTYEQ SEC}).
\end{remark}

\begin{remark}\label{WETRFCAN REM}
The requirement that the maps in 
\eqref{THMIII1ST EQ}
are weak equivalences implies that the maps in
\eqref{THMIII2ND EQ}
are fully faithful.
Therefore, the condition following \eqref{THMIII2ND EQ}
can be restated as saying that those maps are
equivalences of categories (resp. equivalences of categories that are surjective on objects) or, in other words,
that the maps in \eqref{THMIII2ND EQ}
are weak equivalences/trivial fibrations in the canonical model structure on the category $\mathsf{Cat}$ of unenriched categories
\cite{Rez}.
\end{remark}

\begin{remark}\label{FIBSALT REM}
	In light of Remark \ref{WETRFCAN REM}, 
	it is natural to ask if the fibrations in Theorem \ref{THMA}
	admit an analogous description.
	That is, we may ask if a map $F\colon \O \to \mathcal{P}$
	is a fibration in the sense of Theorem \ref{THMA}
	iff the maps in 
	\eqref{THMIII1ST EQ}
	are Kan fibrations in $\mathsf{sSet}$
	and the maps in
	\eqref{THMIII2ND EQ}
	are isofibrations (i.e. fibrations in the canonical model structure in $\mathsf{Cat}$).
	
	However, at our level of generality we can only guarantee the 
	``only if'' direction of this characterization.
	For the ``if'' direction to hold we need to either demand
	that $\P$ itself is fibrant or
	impose an extra condition on the unit of $\V$ (which happens to be satisfied by $\mathsf{sSet}$).
	See Propositions \ref{ISOFIBEASY PROP} and \ref{ISOFIBHARD PROP} for more details.
\end{remark}

\begin{remark}\label{RESTTOCATS REM}
	As noted at the end of the introduction, 
	there is an identification
	$\Cat_\bullet^G(\V) \simeq \Op_\bullet^G(\V) \downarrow \**$,
	where $\**$ denotes the terminal $\V$-category,
	so the $\F$-model structure on $\Op_\bullet^G(\V)$
	also induces a model structure on $\Cat_\bullet^G(\V)$.
	Since $\Cat_\bullet^G(\V)$ contains only unary operations,
	this latter model structure depends only on $\F_1$,
	which is identified with a family of subgroups of $G$ itself.
	In fact, the resulting model structure on 
	$\Cat_\bullet^G(\V)$ matches the model structure 
	obtained by applying \cite{Ste16}
	to the family $\F_1$ and the canonical model structure on
	$\Cat_\bullet(\V)$.

	Moreover, we note that the analogues for $\Cat_\bullet^G(\V)$
	of all three of \cite[Thms. \ref{OC-THMI} and \ref{OC-THMII}]{BP_FCOP} and Theorem \ref{THMA}
	follow from our proofs without using either
	the cofibrant pushout power condition (v)
	or (vii) in Theorem \ref{THMA},
	and without additional restrictions on $\F_1$
	(i.e. no analogues of the  pseudo indexing system
	(cf. \cite[Thm. \ref{OC-THMII}]{BP_FCOP})
	and ``enough units'' (cf. Definition \ref{FAMRESUNI DEF})
	conditions are needed).
	For details, see \cite[Rem. \ref{OC-CSPNTHI REM}]{BP_FCOP} and
	Remark \ref{ALBEETA_REM}.
\end{remark}

\begin{remark}\label{SEMI_REM}
	When working with operads, some authors (e.g. \cite{Spi,Whi17,WY18})
	discuss \emph{semi-model structures}.
	Briefly, these are a weakening of Quillen's original definition
	where those factorization and lifting axioms
	that involve trivial cofibrations
	are only required to hold if the trivial cofibration 
	has cofibrant source \cite[\S 2.2]{WY18}.
	We note that, in particular, semi-model structures suffice for 
	performing 
	bifibrant replacements.
        
	The semi-model structure analogues of 
	\cite[Thms. \ref{OC-THMI} and \ref{OC-THMII}]{BP_FCOP}
	and Theorem \ref{THMA}
	can be obtained by slight variants of our proofs
	without using the global monoid axiom (iv).
	For details, see \cite[Rem. \ref{OC-THMISM REM}]{BP_FCOP}
	and Remarks \ref{JCELLSM REM}, \ref{MONAXSUP REM}.
\end{remark}

\begin{remark}\label{GTRIV REM}
	It seems tempting to think that, for trivial $G=\**$,
	one can omit the existence of genuine model structures condition (ii) in Theorem \ref{THMA}.
	However, this is not so, since 
	our work uses 
	\cite[Thm. \ref{OC-THMI}]{BP_FCOP},
	whose proof needs
	the cofibrant pushout powers condition (v)
	\cite[Rem. \ref{OC-GTRIV REM}]{BP_FCOP}.

	If one further specifies to 
	$G=\**$ and the categorical case $\Cat_\bullet(\V)$,
	there is only one interesting choice of family $\F_1$,
	i.e. the non-empty family of subgroups of $\Sigma_1$,
	which recovers the \emph{canonical model structure} on 
	$\Cat_\bullet(\V)$ discussed in \cite{BM13}.
	In this case, an analysis of our proofs shows that
	one \emph{can}
	drop assumptions (iii),(v)(vii) of Theorem \ref{THMA},
	and replace the global monoid axiom in (iv) 
	with the usual Schwede-Shipley
	monoid axiom \cite{SS00}
	(see \cite[Rem. \ref{OC-MONAX_REM}]{BP_FCOP}),
	so that our assumptions are then a close variation on those in \cite{BM13}.
\end{remark}

\begin{remark}
While technical, the ``generating set of intervals''
condition (viii) is not overly restrictive since, 
by \cite[Lemma 1.12]{BM13},
this condition is satisfied by any 
combinatorial monoidal model category.
\end{remark}

\subsection{Examples}\label{EXAMPLES SEC}

The examples of 
model categories satisfying 
all of conditions (i) through (viii)
in Theorem \ref{THMA}
are fairly limited, 
mostly due to 
the cofibrant pushout powers axiom (v),
which is rather restrictive.
For a discussion of the role of this condition,
see \cite[Rems. \ref{OC-CPPWHY REM} and
\ref{OC-SPNONEX REM}]{BP_FCOP}.

Below we list those examples of categories satisfying all conditions that we are aware of.
\begin{enumerate}[label = (\alph*)]
\item $(\mathsf{sSet},\times)$ or $(\mathsf{sSet}_{\**},\wedge)$
with the Kan model structure.
\item $(\mathsf{Top},\times)$ or $(\mathsf{Top}_{\**},\wedge)$
with the usual Serre model structure.
\item $(\mathsf{Set},\times)$ the category of sets with its canonical model structure,
where weak equivalences are the bijections and all maps are both cofibrations and fibrations.
\item $(\Cat,\times)$ the category of usual categories
with the ``folk'' or canonical model structure (e.g. \cite{Rez})
where weak equivalences are the equivalences of categories,
cofibrations are the functors that are injective on objects,
and fibrations are the isofibrations.
\end{enumerate}
In all these cases, conditions (i) through (v)
were discussed in \cite[\S \ref{OC-EXAMPLES SEC}]{BP_FCOP},
(vi) is well-known,
and (viii) follows from either \cite[Lemma 1.12]{BM13} 
or \cite[Lemma 2.1]{BM13}.

The following is a noteworthy 
non-example.

\begin{remark}
	The category $(\mathsf{Sp}^{\Sigma}(\mathsf{sSet}),\wedge)$
	of symmetric spectra (on simplicial sets),
	with the positive $S$ model structure,
	satisfies most of axioms in Theorem
	\ref{THMA}, 
	with the exceptions being 
	the cofibrant unit requirement in (iii)
	and the cofibrant pushout powers axiom in (v).
	
	Nonetheless, $(\mathsf{Sp}^{\Sigma}(\mathsf{sSet}),\wedge)$
	does seem to satisfy variants of axioms 
	(iii) and (v). For further discussion, 
	see \cite[Rem. \ref{OC-SPNONEX REM}]{BP_FCOP}.
\end{remark}

\subsection{Outline}

\S \ref{SUM SEC} mostly recalls the notions and results from
\cite{BP_FCOP} that we will use,
recalling in particular the model structures
on fixed color operads 
$\mathsf{Op}_{\mathfrak{C}}(\V)$ from
\cite[Thm. \ref{OC-THMI}]{BP_FCOP},
which form the basis of the model structures on
$\mathsf{Op}_{\bullet}(\V)$
in Theorem \ref{THMA}.

\S \ref{MS_SEC} is dedicated to proving 
Theorem \ref{THMA}.
In \S \ref{MAPSOPG_SEC} we identify the relevant classes of maps
in $\mathsf{Op}^G_{\bullet}(\V)$.
Then, in \S \ref{GENCOF SEC} we identify
the necessary sets of generating (trivial) cofibrations,
and outline the overall proof of Theorem \ref{THMA},
with \S \ref{TRIVCOF_SEC},\ref{EQUIVOBJ_SEC},\ref{HMTYEQ SEC}
concluding the proof by addressing the hardest steps.
Lastly, \S \ref{ISOFIB_SEC} discusses an alternative description
of the fibrations in
$\mathsf{Op}^G_{\bullet}(\V)$,
elaborating on 
Remark \ref{FIBSALT REM}.

\section{Summary of previous work}
\label{SUM SEC}

This section is mostly expository, 
recalling the key definitions and results 
in \cite{BP_FCOP} that we need to 
prove Theorem \ref{THMA},
while converting some technical results therein 
to a more convenient format.

In \S \ref{COSYMSEQ SEC} and \S \ref{EQCOSYMSEQ SEC} 
we recall the definitions
of the categories
$\mathsf{Sym}^G_{\bullet}(\V)$
of equivariant symmetric sequences and
$\mathsf{Op}^G_{\bullet}(\V)$
of equivariant operads.
Of particular importance is the discussion on 
representable functors in 
$\mathsf{Sym}^G_{\bullet}
=
\mathsf{Sym}^G_{\bullet}(\mathsf{Set})$,
culminating in \eqref{REPALTDESC EQ},
which are needed in \S \ref{GENCOF SEC}
when describing the generating (trivial) cofibrations
in $\mathsf{Op}^G_{\bullet}(\V)$.

\S \ref{COLFIXMOD SEC} then recalls
\cite[Thm. \ref{OC-THMI}]{BP_FCOP}
as Theorem \ref{THMIREST},
which discusses the model structures on the categories
$\mathsf{Op}^G_{\mathfrak{C}}(\V)$
of fixed color operads
that are one of the main ingredients to building 
the model structure on 
$\mathsf{Op}^G_{\bullet}(\V)$
in Theorem \ref{THMA}.
The (rather technical) condition in
Remark \ref{GOTC_REM} is of particular importance,
as it plays a key role in proving Theorem \ref{THMA}
(more concretely, it is needed in the
proof of Proposition \ref{J_CELL_PROP},
which is one of the key claims needed for Theorem \ref{THMA}).

\subsection{Colored symmetric sequences and colored operads}
\label{COSYMSEQ SEC}

\subsubsection*{Colored symmetric sequences}

\begin{definition}\label{CSYM DEF}
	Let $\mathfrak {C} \in \mathsf{Set}$ be a fixed set of colors (or objects).
	A tuple
	$\vect C = (\mathfrak c_1, \dots, \mathfrak c_n; \mathfrak c_0) \in \mathfrak{C}^{\times n+1}$
	is called a \textit{$\mathfrak {C}$-profile} of \textit{arity} $n$.
	The \textit{$\mathfrak C$-symmetric category} $\Sigma_{\mathfrak C}$ is the category whose objects are the $\mathfrak{C}$-profiles and 
	whose morphisms are action maps
	\begin{equation}\label{CSYM EQ1}
	\vect{C} =
	(\mathfrak c_1, \dots, \mathfrak c_n; \mathfrak c_0) \xrightarrow{\sigma} (\mathfrak c_{\sigma^{-1}(1)}, \dots, \mathfrak c_{\sigma^{-1}(n)}; \mathfrak c_0)
	= \vect{C} \sigma^{-1}
	\end{equation}
	for each permutation $\sigma \in \Sigma_n$, with the natural notion of composition.
	
	Alternatively, one can visualize profiles as corollas (i.e. trees with a single node)
	with edges decorated by colors in $\mathfrak{C}$, as depicted below, so that the map labeled $\sigma$
	is the unique map of trees indicated such that the coloring of an edge equals the coloring of its image.
	\begin{equation}\label{CSYM EQ2}
	\begin{tikzpicture}
	[grow=up,auto,level distance=2.3em,every node/.style = {font=\footnotesize},dummy/.style={circle,draw,inner sep=0pt,minimum size=1.75mm}]
	
	\node at (0,0) [font=\normalsize]{$\vect{C}$}
	child{node [dummy] {}
		child{
			edge from parent node [swap,near end] {$\mathfrak c_n$} node [name=Kn] {}}
		child{
			edge from parent node [near end] {$\mathfrak c_1$}
			node [name=Kone,swap] {}}
		edge from parent node [swap] {$\mathfrak c_0$}
	};
	\draw [dotted,thick] (Kone) -- (Kn) ;
	\node at (5,0) [font=\normalsize] {$\vect{C} \sigma^{-1}
		$}
	child{node [dummy] {}
		child{
			edge from parent node [swap,near end] {$\mathfrak c_{\sigma^{-1}(n)}$} node [name=Kn] {}}
		child{
			edge from parent node [near end] {$\mathfrak c_{\sigma^{-1}(1)}$}
			node [name=Kone,swap] {}}
		edge from parent node [swap] {$\mathfrak c_0$}
	};
	\draw [dotted,thick] (Kone) -- (Kn) ;
	
	\draw[->] (1.5,0.8) -- node{$\sigma$} (3,0.8);
	\end{tikzpicture}
	\end{equation}
	Given any map of color sets $\varphi \colon \mathfrak{C} \to \mathfrak{D}$,
	there is a functor (abusively written)
	$\varphi \colon \Sigma_{\mathfrak{C}} \to \Sigma_{\mathfrak{D}}$,
	given by 
	$\varphi (\mathfrak c_1, \dots, \mathfrak c_n; \mathfrak c_0) = (\varphi(\mathfrak c_1),\cdots,\varphi(\mathfrak c_n);\varphi(\mathfrak c_0))$. 
\end{definition}

\begin{remark}\label{GLOBSIG REM}
	The notation $\vect{C} \sigma^{-1}$
	in \eqref{CSYM EQ1},\eqref{CSYM EQ2}
	reflects the fact that $\Sigma_n$
	acts on the right on $\mathfrak{C}$-profiles of arity $n$
	via 
	$\vect{C} \sigma = (\mathfrak{c}_i)\sigma = 
	(\mathfrak{c}_{\sigma(i)})$,
	where we make the convention that $\sigma(0)=0$.
\end{remark}

\begin{definition}\label{CSSYM DEF}
	Let $\mathcal{V}$ be a category.
	The category $\mathsf{Sym}_\bullet(\mathcal{V})$ of
	\textit{symmetric sequences on $\mathcal{V}$} 
	(on all colors) is the category with:
	\begin{itemize}
		\item objects given by pairs $(\mathfrak C, X)$ with
		$\mathfrak{C} \in \mathsf{Set}$ a set of colors and
		$X \colon \Sigma_{\mathfrak{C}}^{op} \to \mathcal{V}$ a functor;
		\item arrows $(\mathfrak C, X) \to (\mathfrak D, Y)$ given by a map 
		$\varphi \colon \mathfrak{C} \to \mathfrak{D}$ of colors and a natural transformation $X \Rightarrow Y \varphi$ as below.
	\begin{equation}\label{CSSYM EQ}
		\begin{tikzcd}[row sep = tiny, column sep = 45pt]
		\Sigma_{\mathfrak{C}}^{op} \arrow[dr, "X"{name=U}] 
		\arrow{dd}[swap]{\varphi}
	\\
		& \mathcal{V}
	\\
		|[alias=V]| \Sigma_{\mathfrak{D}}^{op} \arrow[ur, "Y"']
		\arrow[Leftarrow, from=V, to=U,shorten >=0.25cm,shorten <=0.25cm]
		\end{tikzcd}
	\end{equation}
	\end{itemize}
\end{definition}

\begin{notation}\label{SIGMABULL NOT}
	We write
	$\Sigma_{\bullet} \to \mathsf{Set}$
	for the Grothendieck construction 
	\cite[Not. \ref{OC-GROTHCONS NOT}]{BP_FCOP}
	of the functor
	$\mathsf{Set} \to \mathsf{Cat}$ given by
	$\mathfrak{C} \mapsto \Sigma_{\mathfrak{C}}$.
	Explicitly, 
	the objects of $\Sigma_{\bullet}$
	are the $\vect{C} \in \Sigma_{\mathfrak{C}}$ 
	for some set of colors $\mathfrak{C}$
	and an arrow from
	$\vect{C} \in \Sigma_{\mathfrak{C}}$ to
	$\vect{D} \in \Sigma_{\mathfrak{D}}$
	over $\varphi \colon \mathfrak{C} \to \mathfrak{D}$
	is an arrow
	$\varphi \vect{C} \to \vect{D}$ in $\Sigma_{\mathfrak{D}}$.
\end{notation}

\begin{remark}\label{COLCHADJ REM}
	We caution that 
	$\mathsf{Sym}_{\bullet}(\V)$
	is quite different from the presheaf category 
	$\mathsf{Fun}(\Sigma_{\bullet}^{op},\V)$.
	
	Instead,  
	$\mathsf{Sym}_{\bullet}(\V)$
	can be regarded as a category of ``fibered presheaves''.
	More precisely, 
	the \emph{color set functor}
	$\mathsf{Sym}_{\bullet}(\V) \to \mathsf{Set}$
	is both a Grothendieck fibration
	and opfibration
	(cf., e.g. \cite[\S \ref{OC-GROTFIB SEC}]{BP_FCOP}),
	with fibers the 
	presheaf categories
	$\Sym_{\mathfrak C}(\V)=
	\mathsf{Fun}(\Sigma_{\mathfrak{C}}^{op},\mathcal{V})$
	and cartesian (resp. cocartesian)
	arrows the diagrams \eqref{CSSYM EQ}
	that are natural isomorphisms (resp. left Kan extensions).

	In particular \cite[Rem. \ref{OC-ALSOOPADJ REM}]{BP_FCOP},
	for any map $\varphi \colon \mathfrak{C} \to \mathfrak{D}$
	one has adjunctions
	\begin{equation}\label{COLCHADJ EQ}
	\varphi_! \colon \Sym_{\mathfrak C}(\V) 
	\rightleftarrows 
	\Sym_{\mathfrak D}(\V) \colon \varphi^{\**}
	\end{equation}
	where $\varphi^{\**}$
	(resp. $\varphi_!$)
	is precomposition with (resp. left Kan extension along)
	$\varphi\colon 
	\Sigma^{op}_{\mathfrak{C}}
	\to 
	\Sigma^{op}_{\mathfrak{D}}$.
\end{remark}

\subsubsection*{Representable functors}

The description of the model structures 
on $\Sym_{\mathfrak C}(\V) $ in 
\S \ref{COLFIXMOD SEC}
will require us to identify certain
representable functors in 
$\mathsf{Sym}_{\bullet} = \mathsf{Sym}_{\bullet}(\mathsf{Set})$.
We start with the following.

\begin{notation}\label{FIBYON NOT}
Let $\mathfrak{C} \in \mathsf{Set}$, $\vect{C} \in \Sigma_{\mathfrak{C}}$.
We write
$\Sigma_{\mathfrak{C}}[\vect{C}] 
\in \mathsf{Sym}_{\mathfrak{C}} = \mathsf{Set}^{\Sigma_{\mathfrak{C}}^{op}}$ for the representable presheaf
\[\Sigma_{\mathfrak{C}}[\vect{C}](-)
= \Sigma_{\mathfrak{C}}(-,\vect{C}).\]
Moreover, this defines a 
\emph{fibered Yoneda functor}
$
\Sigma_{\bullet} \xrightarrow{\Sigma_{\bullet}[-]} \mathsf{Sym}_{\bullet}
$
by mapping an arrow
$\varphi \vect{C} \to \vect{D}$
over $\varphi \colon \mathfrak{C} \to \mathfrak{D}$
to the natural transformation 
$\Sigma_{\mathfrak{C}}[\vect{C}]
\Rightarrow
\varphi^{\**}
\Sigma_{\mathfrak{D}}[\vect{D}]
$ given by the composites
\[\Sigma_{\mathfrak{C}}[\vect{C}](-)
= \Sigma_{\mathfrak{C}}(-,\vect{C})
\to 
\Sigma_{\mathfrak{D}}(\varphi(-),\varphi\vect{C})
\to
\Sigma_{\mathfrak{D}}(\varphi(-),\vect{D})
=
\varphi^{\**} \Sigma_{\mathfrak{D}}[\vect{D}](-).
\]
\end{notation}

\begin{proposition}[{\cite[Prop. \ref{OC-FIBYONPUSH PROP}]{BP_FCOP}}]
        \label{FIBYONPUSH PROP}
        Let $\vect{C} \in \Sigma_{\mathfrak{C}}$,
        $\varphi \colon \mathfrak{C} \to \mathfrak{D}$
        be a map of colors.
        
        Then there is an identification
        $\varphi_{!} \Sigma_{\mathfrak{C}}[\vect{C}]
        \xrightarrow{\simeq}
        \Sigma_{\mathfrak{D}}[\varphi\vect{C}]$,
        adjoint to the canonical map
        $\Sigma_{\mathfrak{C}}[\vect{C}]
        \to
        \varphi^{\**}\Sigma_{\mathfrak{D}}[\varphi\vect{C}]$).
        In other words,
        $\Sigma_{\bullet}[-] \colon 
        \Sigma_{\bullet} \to \mathsf{Sym}_{\bullet}$
        preserves cocartesian arrows.
\end{proposition}

The fibered Yoneda functor
$
\Sigma_{\bullet} \xrightarrow{\Sigma_{\bullet}[-]} \mathsf{Sym}_{\bullet}
$
does not quite suffice for our purposes, 
due to the domain $\Sigma_{\bullet}$
lacking enough colimits.
To extend $\Sigma_{\bullet}[-]$,
we now discuss colored forests.
In the following, $\Phi$ denotes the category of forests
(i.e. formal coproducts of trees; see \cite[\S 5.1]{Per18}).

\begin{definition}\label{COLFOR DEF}
Let $\mathfrak{C}$ be a set of colors.
The category $\Phi_{\mathfrak{C}}$ of $\mathfrak{C}$-colored forests has
\begin{itemize}
\item objects pairs
$\vect{F} = (F,\mathfrak{c})$
where 
$F\in \Phi$ is a forest
and 
$\mathfrak{c}\colon \boldsymbol{E}(F) \to \mathfrak{C}$ 
is a coloring of its edges;
\item a map
$\vect{F}=(F,\mathfrak{c}) \to 
(F',\mathfrak{c}') = \vect{F'}$
is a map $\rho \colon F \to F'$ in $\Phi$
such that
$\mathfrak{c} = \mathfrak{c}' \rho$.
\end{itemize}
For a map of colors,
$\varphi\colon \mathfrak{C} \to \mathfrak{D}$
we again write
$\varphi \colon \Phi_{\mathfrak{C}} \to \Phi_{\mathfrak{D}}$
for the functor  
\begin{equation}
        \label{PHIF_EQ}
        \vect{F} = (F,\mathfrak{c})
        \mapsto (F,\varphi\mathfrak{c}) = \varphi\vect{F}.
\end{equation}
Adapting Notation \ref{SIGMABULL NOT},
we likewise write
$\Phi_{\bullet} \to \mathsf{Set}$
for the Grothendieck construction of the functor 
$\mathfrak{C} \mapsto \Phi_{\mathfrak{C}}$.
Explicitly, 
the objects of $\Phi_{\bullet}$
are the $\vect{F} \in \Phi_{\mathfrak{C}}$ 
for some set of colors $\mathfrak{C}$
and an arrow from
$\vect{F} \in \Phi_{\mathfrak{C}}$ to
$\vect{F'} \in \Phi_{\mathfrak{D}}$
over $\varphi \colon \mathfrak{C} \to \mathfrak{D}$
is an arrow
$\varphi \vect{F} \to \vect{F'}$ in $\Phi_{\mathfrak{D}}$.
\end{definition}

For each vertex $v \in \boldsymbol{V}(F)$ in a forest,
we write $F_v$ for the associated corolla.
Note that, given a $\mathfrak{C}$-coloring $\vect{F}$ on $F$,
one one likewise obtains colorings $\vect{F}_v$ on $F_v$.

\begin{notation}
Given $\vect{F} \in \Phi_{\mathfrak{C}}$
we define
\begin{equation}\label{GENSIGC EQ}
\Sigma_{\mathfrak{C}}[\vect{F}]=
\coprod\nolimits^{\mathfrak{C}}_{v \in \boldsymbol{V}(F)} 
\Sigma_{\mathfrak{C}}[\vect{F}_v]
\end{equation}
where we highlight that the coproduct $\amalg^{\mathfrak{C}}$ is fibered, i.e. it takes place in $\mathsf{Sym}_{\mathfrak{C}}$
rather than $\mathsf{Sym}_{\bullet}$.
\end{notation}

\begin{example}\label{COLFORES EX}
Let 
$\mathfrak{C} = \{ \mathfrak{a}, \mathfrak{b}, \mathfrak{c} \}$.
On the left below we depict a $\mathfrak{C}$-colored forest 
$\vect{F} = \vect{T} \amalg \vect{S}$
with tree components $\vect{T},\vect{S}$.
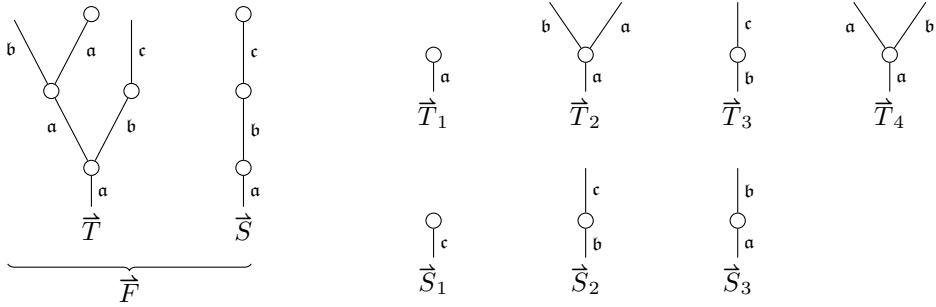
\begin{equation}
	\begin{tikzpicture}[auto,grow=up, level distance = 2.2em,
	every node/.style={font=\scriptsize,inner sep = 2pt}]%
		\tikzstyle{level 2}=[sibling distance=3em]%
			\node at (0,0) [font = \normalsize] {$\vect{T}$}%
				child{node [dummy] {}%
					child[level distance = 2.9em]{node [dummy] {}%
						child{node {}%
						edge from parent node [swap] {$\mathfrak{c}$}}%
					edge from parent node [swap,near end] {$\mathfrak{b}$}}%
					child[level distance = 2.9em]{node [dummy] {}%
						child[level distance = 2.9em]{node [dummy] {}%
						edge from parent node [swap,	near end] {$\mathfrak{a}\phantom{\mathfrak{b}}$}}%
						child[level distance = 2.9em]{node {}%
						edge from parent node [near end] {$\mathfrak{b}$}}%
					edge from parent node [near end] {$\phantom{\mathfrak{b}}\mathfrak{a}$}}%
				edge from parent node [swap] {$\mathfrak{a}$}};%
			\node at (2,0) [font = \normalsize] {$\vect{S}$}%
				child{node [dummy] {}%
					child[level distance = 2.9em]{node [dummy] {}%
						child{node [dummy] {}%
						edge from parent node [swap] {$\mathfrak{c}$}}%
					edge from parent node [swap] {$\mathfrak{b}$}}%
				edge from parent node [swap] {$\mathfrak{a}$}};%
			\node at (4.5,1.5) [font = \normalsize] {$\vect{T}_1$}%
				child{node [dummy] {}%
				edge from parent node [swap] {$\mathfrak{a}$}};%
			\node at (6.5,1.5) [font = \normalsize] {$\vect{T}_2$}%
				child{node [dummy] {}%
					child{node {}%
					edge from parent node [swap, near end] {$\mathfrak{a}\phantom{\mathfrak{b}}$}}%
					child{node {}%
					edge from parent node [near end] {$\mathfrak{b}$}}%
				edge from parent node [swap] {$\mathfrak{a}$}};%
			\node at (8.5,1.5) [font = \normalsize] {$\vect{T}_3$}%
				child{node [dummy] {}%
					child{node {}%
					edge from parent node [swap] {$\mathfrak{c}$}}%
				edge from parent node [swap] {$\mathfrak{b}$}};%
			\node at (10.5,1.5) [font = \normalsize] {$\vect{T}_4$}%
				child{node [dummy] {}%
					child{node {}%
					edge from parent node [swap, near end] {$\mathfrak{b}$}}%
					child{node {}%
					edge from parent node [near end] {$\phantom{\mathfrak{b}}\mathfrak{a}$}}%
				edge from parent node [swap] {$\mathfrak{a}$}};%
			\node at (4.5,-0.7) [font = \normalsize] {$\vect{S}_1$}%
				child{node [dummy] {}%
				edge from parent node [swap] {$\mathfrak{c}$}};%
			\node at (6.5,-0.7) [font = \normalsize] {$\vect{S}_2$}%
				child{node [dummy] {}%
					child{node {}%
					edge from parent node [swap] {$\mathfrak{c}$}}%
				edge from parent node [swap] {$\mathfrak{b}$}};%
			\node at (8.5,-0.7) [font = \normalsize] {$\vect{S}_3$}%
				child{node [dummy] {}%
					child{node {}%
					edge from parent node [swap] {$\mathfrak{b}$}}%
				edge from parent node [swap] {$\mathfrak{a}$}};%
		\draw[decorate,decoration={brace,amplitude=2.5pt}] (2.1,-0.5) -- (-1.1,-0.5) 
		node[midway,inner sep=4pt,font=\normalsize]{$\vect{F}$}; %
	\end{tikzpicture}%
\end{equation}%
Moreover, on the right we depict the $\mathfrak{C}$-profiles/corollas
$\vect{T}_i$ and $\vect{S}_j$
corresponding to the vertices of $T,S$, so that
\[
	\Sigma_{\mathfrak{C}}[\vect{F}]
=
	\Sigma_{\mathfrak{C}}[\vect{T}] 
	\amalg^{\mathfrak{C}}
	\Sigma_{\mathfrak{C}}[\vect{S}]
=
	\left(
	\coprod_{1\leq i \leq 4}^{\mathfrak{C}}
	\Sigma_{\mathfrak{C}}[\vect{T}_i] 
	\right)
	\amalg^{\mathfrak{C}}
	\left(
	\coprod_{1\leq j \leq 3}^{\mathfrak{C}}
	\Sigma_{\mathfrak{C}}[\vect{S}_j]
	\right)
\]
\end{example}

\begin{remark}
	Writing $\Phi^o_{\bullet} \hookrightarrow \Phi_{\bullet}$
	for the wide subcategory whose arrows are the outer face maps
	\cite[\S 3.2]{BP21} in each tree component
	(these are the maps sending vertices to vertices),
	\eqref{GENSIGC EQ} defines a functor
\begin{equation}\label{REPSFUN EQ}
	\Phi_{\bullet}^o 
	\xrightarrow{\Sigma_{\bullet}[-]}
	\mathsf{Sym}_{\bullet},
\end{equation}
and, by formula \eqref{GENSIGC EQ},
Proposition \ref{FIBYONPUSH PROP} immediately generalizes, i.e. one has identifications
\begin{equation}\label{PUSHEQAG EQ}
\varphi_! \Sigma_{\mathfrak{C}}[\vect{F}] \simeq 
\Sigma_{\mathfrak{D}}[\varphi \vect{F}].
\end{equation}
\end{remark}

\begin{notation}\label{TAUTCOL NOT}
We write $(-)^{\tau} \colon \Phi \to \Phi_{\bullet}$
for the \emph{tautological coloring} functor
that sends $F \in \Phi$ to 
$F^{\tau} \in \Phi_{\boldsymbol{E}(T)}$
where
$F^{\tau} = (F,\mathfrak{t})$ is the underlying forest $F$
together with the identity coloring
$\mathfrak{t} \colon \boldsymbol{E}(T) \xrightarrow{=} \boldsymbol{E}(T)$.
Moreover, we then abbreviate 
$\Sigma_{\tau}[F] = \Sigma_{\boldsymbol{E}(F)}[F^{\tau}]$.
\end{notation}

\begin{remark}
For any colored forest $\vect{F}=(F,\mathfrak{c})$,
regarding $\mathfrak{c} \colon \boldsymbol{E}(T) \to \mathfrak{C}$
as a change of color map, 
one has $\vect{F} = \mathfrak{c} F^{\tau}$,
so that \eqref{PUSHEQAG EQ} yields 
\begin{equation}\label{CANPUSH EQ}
\Sigma_{\mathfrak{C}}[\vect{F}] = 
\mathfrak{c}_! \Sigma_{\tau}[F]
\end{equation}
\end{remark}

\subsubsection*{Colored operads}

We now describe the category 
$\mathsf{Op}_{\bullet}(\V)$
of colored operads as the fiber algebras
(cf. \cite[Def. \ref{OC-FIBMON DEF}]{BP_FCOP};
see Remark \ref{FIBMON REM} below) over a 
certain fibered monad $\mathbb{F}$ on 
$\mathsf{Sym}_{\bullet}(\V)$, 
described using trees.

Following Definition \ref{COLFOR DEF},
we write 
$\Omega_{\mathfrak{C}} \subset \Phi_{\mathfrak{C}}$
for the subcategory of $\mathfrak{C}$-colored forests that are trees,
as well as 
$\Omega^0_{\mathfrak{C}} \subseteq \Omega_{\mathfrak{C}}$
for the wide subcategory whose arrows are the isomorphisms.
Next, as in \cite[Not. 3.38]{BP21},
there is an ``arity functor'',
which we call the \emph{leaf-root functor}, described as follows
\begin{equation}\label{LRDEF EQ}
\begin{tikzcd}[row sep = 0pt]
	\Omega_{\mathfrak C}^0 
	\ar{r}{\mathsf{lr}} &
	\Sigma_{\mathfrak{C}}
\\
	\vect{T} = 
	(T,\boldsymbol{E}(T) \xrightarrow{\mathfrak{c}} \mathfrak{C})
	\ar[mapsto]{r} &
	\left(
	\mathfrak{c}(l_1),\cdots,\mathfrak{c}(l_n);\mathfrak{c}(r)
	\right)
\end{tikzcd}
\end{equation}
where $r$ is the root of $T$ and
$l_1,\cdots,l_n$ the leaves
(ordered left to right following the planarization).

\begin{example}
	For $\vect{T},\vect{S}$ the trees in Example \ref{COLFORES EX}
	we have
	$\mathsf{lr}(\vect{T}) = (\mathfrak{b},\mathfrak{c};\mathfrak{a})$
	and 
	$\mathsf{lr}(\vect{S}) = (;\mathfrak{a})$.
\end{example}

For each $\mathfrak{C}$-profile $\vect{C}$,
we write $\vect{C} \downarrow \Omega_{\mathfrak{C}}^0$
for the undercategory with respect to $\mathsf{lr}$, 
whose objects consist of a tree $\vect{T}\in \Omega_{\mathfrak{C}}^0$
together with a choice of isomorphism 
$\vect{C} \to \mathsf{lr}(\vect{T})$.
Morally, $\vect{C} \downarrow \Omega_{\mathfrak{C}}^0$
is the ``groupoid of trees with arity $\vect{C}$''.
Adapting \cite[page 816]{BM07} we now have the following.

\begin{definition}\label{FREEOP DEF}
	Let $\mathcal{V}$ be a closed symmetric monoidal category.
	
	The \textit{fibered free operad monad} $\mathbb{F}$ on $\mathsf{Sym}_\bullet(\mathcal{V})$ 
	assigns to 
	$X \colon \Sigma_{\mathfrak{C}}^{op} \to \mathcal{V}$
	the functor
\begin{equation}\label{FROPEXP EQ}
\mathbb{F} X (\vect{C})
=
\coprod_{[\vect{T}] \in 
	\mathsf{Iso}(\vect{C} \downarrow \Omega^0_{\mathfrak{C}})}
\left(
\left(
\bigotimes_{v \in \boldsymbol{V}(T)} X(\vect{T}_v)
\right)
\cdot_{\mathsf{Aut}_{\Omega_{\mathfrak{C}}}(\vect{T})}
\mathsf{Aut}_{\Sigma_{\mathfrak{C}}}(\vect{C})
\right)
\end{equation}
where $\mathsf{Iso}(-)$ denotes isomorphism classes of objects.
\end{definition}

Formula \eqref{FROPEXP EQ}
is presented here only for the sake of completeness,
as this paper will not require a full understanding of 
$\mathbb{F}$.
A complete description of the monad 
$\mathbb{F}$
is given in the prequel \cite{BP_FCOP},
with \cite[Def. \ref{OC-FREEOP DEF}]{BP_FCOP}
providing an alternative description of
\eqref{FROPEXP EQ},
and the structure maps
$\mathbb{F}\mathbb{F} \Rightarrow \mathbb{F}$,
$id \Rightarrow \mathbb{F}$ 
discussed in \cite[App. \ref{OC-MONAD_APDX}]{BP_FCOP},
culminating in \cite[Def. \ref{OC-COLORMON_DEF}]{BP_FCOP}.

\begin{remark}\label{FIBMON REM}
	Following the construction in \cite[Def. \ref{OC-COLORMON_DEF}]{BP_FCOP},
	one has that the monad $\mathbb{F}$
	on $\mathsf{Sym}_{\bullet}(\V)$
	preserves color sets,
	and that the structure maps
	$\mathbb{F}\mathbb{F} \Rightarrow \mathbb{F}$,
	$id \Rightarrow \mathbb{F}$
	are the identity on colors.
	In other words,
	$\mathbb{F}$ is a \emph{fibered monad} for
	$\mathsf{Sym}_{\bullet} \to \mathsf{Set}$
	in the sense of 
	\cite[Def. \ref{OC-FIBMON DEF}]{BP_FCOP}.
	
	In particular, 
	by restriction one obtains monads
	$\mathbb{F}_{\mathfrak{C}}$
	on the fixed color fibers 
	$\mathsf{Sym}_{\mathfrak{C}}(\V)$.
\end{remark}

Recall \cite[Def. \ref{OC-FIBMON DEF}]{BP_FCOP},
that an $\mathbb{F}$ algebra $X$
is called a \emph{fiber algebra}
if the multiplication 
$\mathbb{F} X \to X$
is an identify on colors.

\begin{definition}
	The category 
	$\mathsf{Op}_{\bullet}(\V)$
	of colored operads
	is the category of fiber algebras for
	the fibered monad $\mathbb{F}$ on
	$\mathsf{Sym}_{\bullet}(\V) \to \mathsf{Set}$.
\end{definition}

\subsection{Equivariant colored symmetric sequences and colored operads}
\label{EQCOSYMSEQ SEC}

We now extend the discussion in the previous section to the equivariant context.

Letting $G$ be a group,
we will write
$\mathsf{Sym}^G_{\bullet}(\V)$,
which we call the category of
\emph{equivariant symmetric sequences},
for the category of $G$-objects in $\mathsf{Sym}_{\bullet}(\V)$.

By abstract nonsense, 
the color set functor
$\mathsf{Sym}^G_{\bullet}(\V) 
\to \mathsf{Set}^G$
is again a Grothendieck fibration
\cite[Rem. \ref{OC-FUNISGROTH REM}]{BP_FCOP},
and one has a fibered monad
$\mathbb{F}^G$ on $\mathsf{Sym}^G_{\bullet}(\V)$
(explicitly, $\mathbb{F}^G$
is simply $\mathbb{F}$ applied to $G$-objects)
whose fiber algebras
are the category 
$\mathsf{Op}^G_{\bullet}(\V)$
of $G$-objects on $\mathsf{Op}_{\bullet}(\V)$
\cite[Prop. \ref{OC-DIAGRAMFM_PROP}]{BP_FCOP}.
As a side note, we observe that, 
though $\mathbb{F}^G$ is again described by 
\eqref{FROPEXP EQ},
it can be tricky to describe the $G$-actions via that formula, 
since those are in general not the identity on colors.
Alternative descriptions can be found in 
\cite[Prop. \ref{OC-FGC PROP} and Rem. \ref{OC-FROPEXPG REM}]{BP_FCOP}.

For $\mathfrak{C} \in \mathsf{Set}^G$,
we then write
$\mathsf{Sym}^G_{\mathfrak{C}}(\V)$,
$\mathsf{Op}^G_{\mathfrak{C}}(\V)$
for the associated fibers of
$\mathsf{Sym}^G_{\bullet}(\V)$,
$\mathsf{Op}^G_{\bullet}(\V)$.

Extending \eqref{COLCHADJ EQ},
we then have the following, 
cf. \cite[Rem. \ref{OC-OP_MAP REM}]{BP_FCOP}.

\begin{remark}\label{OP_MAP REM}
	For any map of $G$-sets 
	$\varphi \colon \mathfrak C \to \mathfrak D$
	one has a pair of adjunctions
	\begin{equation}\label{GC_CHANGE_EQ}
	\begin{tikzcd}
	\Op^G_{\mathfrak C}(\V) 
	\arrow[shift left]{r}{\check{\varphi}_!}
	\arrow[d, "\mathsf{fgt}"']
	&
	\Op^G_{\mathfrak D}(\V) 
	\arrow[shift left]{l}{\varphi^{\**}}
	\arrow[d, "\mathsf{fgt}"]
	\\
	\Sym^G_{\mathfrak C}(\V) 
	\arrow[shift left]{r}{\varphi_!}
	&
	\Sym^G_{\mathfrak D}(\V) 
	\arrow[shift left]{l}{\varphi^{\**}}
	\end{tikzcd}
	\end{equation}
	where the right adjoints $\varphi^{\**}$ are
	both given by precomposition with
	$\varphi \colon 
	\Sigma_{\mathfrak{C}} \to \Sigma_{\mathfrak{D}}$,
	and are thus compatible with the forgetful functors, 
	i.e. $\varphi^{\**} \circ \mathsf{fgt} = 
	\mathsf{fgt} \circ \varphi^{\**}$, 
	while the left adjoints are not:
	$\varphi_!$ is simply a left Kan extension, while $\check{\varphi}_!$ is given by the coequalizer
	\begin{equation}\label{CFS_EQ}
	\check{\varphi}_! \O \simeq \mathop{\mathrm{coeq}}(\mathbb F_{\mathfrak D} \varphi_! \mathbb F_{\mathfrak C}\O \rightrightarrows \mathbb F_{\mathfrak D} \varphi_! \O).
	\end{equation}
	In general, we do not have a more explicit description of $\check{\varphi}_!$.
	However, when $\varphi$ is injective, 
	$\varphi_!X$ is the extension by $\emptyset$,
	from which it follows that 
	$\mathbb F_{\mathfrak D} \varphi_! = \varphi_! \mathbb F_{\mathfrak C}$,
	and \eqref{CFS_EQ}
	then says that
	$\check{\varphi}_! \O
	\simeq
	\mathop{\mathrm{coeq}} \left( \varphi_! \mathbb{F}_{\mathfrak{C}}  \mathbb{F}_{\mathfrak{C}} \O
	\rightrightarrows 
	\varphi_!  \mathbb{F}_{\mathfrak{C}} \O
	\right)
	\simeq 
	\varphi_! \left( \mathop{\mathrm{coeq}} \left( 
	\mathbb{F}_{\mathfrak{C}}  \mathbb{F}_{\mathfrak{C}} \O
	\rightrightarrows 
	\mathbb{F}_{\mathfrak{C}} \O
	\right) \right)
	\simeq 
	\varphi_! \O$,
	so that  	
	$\varphi_! \circ \mathsf{fgt} \simeq 
	\mathsf{fgt} \circ \check{\varphi}_!$.
\end{remark}

In \S \ref{GENCOF SEC} we will make use of the following,
which is an instance of 
\cite[Rem. \ref{OC-LIMINFIBSUP REM}]{BP_FCOP}.

\begin{remark}
	Given a diagram
	$I \xrightarrow{X_{\bullet}} 
	\mathsf{Sym}_{\bullet}^G(\V)$,
	and writing
	$\mathfrak{C} = \colim_{i \in I} \mathfrak{C}_{X_i}$
	and 
	$\varphi_i \colon \mathfrak{C}_{X_i} \to \mathfrak{C}$
	for the canonical maps,
	one has
	$	
	\colim_{i \in I} X_i = 
	\colim_{i \in I} \varphi_{i,!} X_i,
	$	
	where the second colimit is computed in 
	$\mathsf{Sym}_{\mathfrak{C}}^G(\V)$.
	Thus, for an arbitrary cocone
	$\psi_i \colon \mathfrak{C}_{X_i} \to \mathfrak{D}$
	with  
	$\psi \colon \mathfrak{C} \to \mathfrak{D}$ the induced map, one has
	\begin{equation}\label{LIMINFIBSUP EQ}
	\psi_!\left(\colim_{i \in I} X_i\right) 
	= 
	\psi_!
	\left(\colim_{i \in I} \varphi_{i,!} X_i\right)
	=
	\colim_{i \in I} \psi_! \varphi_{i,!} X_i
	=
	\colim_{i \in I} \psi_{i,!} X_i.
	\end{equation}
\end{remark}

To recall the model structures on 
$\mathsf{Sym}^G_{\mathfrak{C}}(\V)$,
$\mathsf{Op}^G_{\mathfrak{C}}(\V)$
in \S \ref{COLFIXMOD SEC}
we will need a more explicit description of 
$\mathsf{Sym}^G_{\mathfrak{C}}(\V)$
(the discussion above
only describes $\mathsf{Sym}^G_{\mathfrak{C}}(\V)$
abstractly as a fiber of $\mathsf{Sym}^G_{\bullet}(\V)$;
note that this 
\emph{is not} the category 
of $G$-objects on $\mathsf{Sym}_{\mathfrak{C}}(\V)$
unless the $G$-action on $\mathfrak{C}$
is trivial).

By \cite[Prop. \ref{OC-EQUIVFNCON PROP}]{BP_FCOP},
there is an identification 
$\mathsf{Sym}^G_{\mathfrak{C}}(\mathcal{V}) 
\simeq 
\V^{G \ltimes \Sigma^{op}_{\mathfrak{C}}}$
as the functors from a certain groupoid
$G \ltimes \Sigma^{op}_{\mathfrak{C}}$,
which can be described as an instance of
\cite[Ex. \ref{OC-GLTIMES EQ}]{BP_FCOP}.
Here we prefer an alternative description
of $G \ltimes \Sigma^{op}_{\mathfrak{C}}$,
which follows from 
\cite[Rem. \ref{OC-SIGACT REM}]{BP_FCOP},
and adapts Definition \ref{CSYM DEF} and
Remark \ref{GLOBSIG REM}.

\begin{remark}\label{GLTIMESSIG REM}
	Let $\mathfrak {C} \in \mathsf{Set}^G$ be a fixed $G$-set of colors.
	The groupoid $G \ltimes \Sigma^{op}_{\mathfrak{C}}$
	has objects the $\mathfrak {C}$-profiles
	$\vect C = (\mathfrak c_1, \dots, \mathfrak c_n; \mathfrak c_0)$
	and morphisms the action maps
	\begin{equation}\label{CSYMG EQ}
	\vect{C} =
	(\mathfrak c_1, \dots, \mathfrak c_n; \mathfrak c_0) \xrightarrow{(g,\sigma)} 
	(g \mathfrak c_{\sigma(1)}, \dots, 
	g \mathfrak c_{\sigma(n)}; 
	g \mathfrak c_0)
	= g \vect{C} \sigma
	\end{equation}
	for $(g,\sigma) \in G \times \Sigma_n^{op}$,
	with the natural notion of composition.
\end{remark}

\begin{remark}
	Setting $G=\**$ in Remark \ref{GLTIMESSIG REM}
	recovers the \emph{opposite} 
	$\Sigma_{\mathfrak{C}}^{op}$ of 
	Definition \ref{CSYM DEF}.
\end{remark}

\begin{remark}\label{SIGACT REM}
	Extending Remark \ref{GLOBSIG REM}, 
	the notation $g \vect{C} \sigma$ in \eqref{CSYMG EQ}
	encodes a $(G \times \Sigma_n^{op})$-action
	(i.e. $G$ acts on the left and $\Sigma_n$ on the right)
	on the set of $n$-ary $\mathfrak{C}$-profiles,
	via
	$g(\mathfrak{c}_i)\sigma =
	(g\mathfrak{c}_{\sigma(i)})$.
\end{remark}

\begin{definition}\label{STABS DEF}
	If a subgroup $\Lambda \leq G \times \Sigma_n^{op}$
	fixes a profile 
	$\vect{C} = (\mathfrak{c}_1,\cdots,\mathfrak{c}_n;\mathfrak{c}_0)$,
	i.e. if
	$g\mathfrak{c}_{\sigma(i)} = \mathfrak{c}_i$
	for all $(g, \sigma) \in \Lambda, 0 \leq i \leq n$,
	we say that \textit{$\Lambda$ stabilizes $\vect C$}. 
\end{definition}

\begin{remark}\label{CHOOSESIGN REM}
	For $\Lambda \leq G \times \Sigma_n^{op}$ 
	the projection to $\Sigma_n^{op}$ yields a
	right action of $\Lambda$ on 
	$\underline{n}_+ = \{0,1,\cdots,n\}$.
	
	Writing $\Lambda_i\leq \Lambda$ for the stabilizer of $i \in \underline{n}_{+}$ and $H_i = \pi_G(\Lambda_i)$
	for its projection onto $G$,
	one then has $H_i = g H_{\sigma(i)} g^{-1}$ for all
	$(g, \sigma) \in \Lambda$.
	Moreover, the profiles $\vect{C}$  
	stabilized by $\Lambda$ are in bijection with choices of 
	$H_i$-fixed colors $\mathfrak{c}_i$ 
	for $i$ ranging over a set of representatives of
	the orbits $\underline{n}_+ /\Lambda$.
\end{remark}

We now discuss the representable functors in 
$\mathsf{Sym}^G_{\mathfrak{C}} =
\mathsf{Sym}^G_{\mathfrak{C}}(\mathsf{Set}) \simeq 
\mathsf{Set}^{G \ltimes \Sigma^{op}_{\mathfrak{C}}}$.
However, some caution is needed,
as though $G \ltimes \Sigma^{op}_{\mathfrak{C}}$
and $\Sigma^{op}_{\mathfrak{C}}$
have the same objects $\vect{C}$,
the representable functor for $\vect{C}$ in
$\mathsf{Sym}^G_{\mathfrak{C}}$
is \emph{not} simply 
$\Sigma_{\mathfrak{C}}[\vect{C}]
\in \mathsf{Sym}_{\mathfrak{C}}$.
We first need the following construction,
where
$(-)^{\tau}\colon \Phi \to \Phi_{\bullet}$
is the tautological coloring, 
cf. Notation \ref{TAUTCOL NOT}.

\begin{definition}
        \label{GCC_DEF}
Let $G$ be a group,
$\mathfrak{C} \in \mathsf{Set}^G$
be a $G$-set of colors, 
and $\vect{C}  = (C, \mathfrak c) \in \Sigma_{\mathfrak{C}}$ be a $\mathfrak{C}$-profile/corolla,
with
$C\in \Sigma$ the underlying corolla
and 
$\mathfrak{c} \colon \boldsymbol{E}(C) \to \mathfrak{C}$ the coloring.

Writing 
$G \cdot \mathfrak{c} \colon G \cdot \boldsymbol{E}(C) \to \mathfrak{C}$
for the $G$-equivariant map adjoint to $\mathfrak c$, 
and $G \cdot C  \in \Phi^G$
for the $G$-free forest determined by $C$ 
(so that $\boldsymbol{E}(G \cdot C) = 
G \cdot \boldsymbol{E}(C)$),
we define $G \cdot_{\mathfrak{C}} \vect{C} \in \Phi^G_{\mathfrak{C}}$ to be
\begin{equation}\label{GCDTCC EQ}
G \cdot_{\mathfrak{C}} \vect{C} = 
(G \cdot \mathfrak{c})(G \cdot C)^{\tau},
\end{equation}
where we make use of the functor \eqref{PHIF_EQ}
for $\varphi = G \cdot \mathfrak{c}$.
\end{definition}

\begin{remark}
Writing $g \colon \mathfrak{C} \to \mathfrak{C}$
for the $G$-action maps,
one has the more explicit formulas 
(see Example \ref{GDOTCC EX})
\[
        G \cdot C = \coprod_{g \in G} C,
        \qquad
        \qquad
G \cdot_{\mathfrak{C}} \vect{C}
= 
\coprod_{g \in G}
g \vect{C}
\]
However, in practice we will prefer to use 
\eqref{GCDTCC EQ} for technical purposes.
\end{remark}

\begin{remark}
The definition \eqref{GCDTCC EQ}
extends to a functor
$\Phi_{\mathfrak{C}}
\xrightarrow{G \cdot_{\mathfrak{C}} (-)}
\Phi_{\mathfrak{C}}^G$,
left adjoint to the forgetful functor
$ \Phi_{\mathfrak{C}}^G
\to \Phi_{\mathfrak{C}}$.
\end{remark}

\begin{example}\label{GDOTCC EX}
Let $G = \{1,i,-1,-i\} \simeq \mathbb{Z}_{/4}$ 
be the group of quartic roots of unit and
$\mathfrak{C} = \{\mathfrak{a}, -\mathfrak{a}, i\mathfrak{a}, -i\mathfrak{a}, \mathfrak{b}, i \mathfrak{b} \}$ where we implicitly have
$-\mathfrak{b} = \mathfrak{b}$.
The following depicts the forest (of corollas)
$G \cdot_{\mathfrak{C}} \vect{C}$
in $\Phi_{\mathfrak{C}}^G$
for $\vect{C}$ in $\Sigma_{\mathfrak{C}}$ the leftmost corolla.
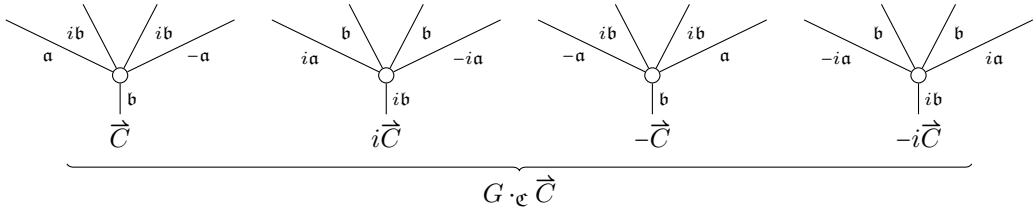
\begin{equation}
	\begin{tikzpicture}[auto,grow=up, level distance = 2.2em,
	every node/.style={font=\scriptsize,inner sep = 2pt}]%
		\tikzstyle{level 2}=[sibling distance=3em]%
			\node at (0,0) [font = \normalsize] {$\vect{C}$}%
				child{node [dummy] {}%
					child{node {}%
					edge from parent node [swap] {$-\mathfrak{a}$}}%
					child[level distance = 2.9em]{node {}%
					edge from parent node [swap,	near end] {$i\mathfrak{b}$}}%
					child[level distance = 2.9em]{node {}%
					edge from parent node [near end] {$i\mathfrak{b}$}}%
					child{node {}%
					edge from parent node  {$\mathfrak{a}$}}%
				edge from parent node [swap] {$\mathfrak{b}$}};%
			\node at (3.5,0) [font = \normalsize] {$i\vect{C}$}%
				child{node [dummy] {}%
					child{node {}%
					edge from parent node [swap] {$-i\mathfrak{a}$}}%
					child[level distance = 2.9em]{node {}%
					edge from parent node [swap,	near end] {$\mathfrak{b}$}}%
					child[level distance = 2.9em]{node {}%
					edge from parent node [near end] {$\mathfrak{b}$}}%
					child{node {}%
					edge from parent node  {$i\mathfrak{a}$}}%
				edge from parent node [swap] {$i\mathfrak{b}$}};%
			\node at (7,0) [font = \normalsize] {$-\vect{C}$}%
				child{node [dummy] {}%
					child{node {}%
					edge from parent node [swap] {$\mathfrak{a}$}}%
					child[level distance = 2.9em]{node {}%
					edge from parent node [swap,	near end] {$i\mathfrak{b}$}}%
					child[level distance = 2.9em]{node {}%
					edge from parent node [near end] {$i\mathfrak{b}$}}%
					child{node {}%
					edge from parent node  {$-\mathfrak{a}$}}%
				edge from parent node [swap] {$\mathfrak{b}$}};%
			\node at (10.5,0) [font = \normalsize] {$-i\vect{C}$}%
				child{node [dummy] {}%
					child{node {}%
					edge from parent node [swap] {$i\mathfrak{a}$}}%
					child[level distance = 2.9em]{node {}%
					edge from parent node [swap,	near end] {$\mathfrak{b}$}}%
					child[level distance = 2.9em]{node {}%
					edge from parent node [near end] {$\mathfrak{b}$}}%
					child{node {}%
					edge from parent node  {$-i\mathfrak{a}$}}%
				edge from parent node [swap] {$i\mathfrak{b}$}};%
		\draw[decorate,decoration={brace,amplitude=2.5pt}] (11.2,-0.4) -- (-0.7,-0.4) 
		node[midway,inner sep=4pt,font=\normalsize]{$G \cdot_{\mathfrak{C}} \vect{C}$}; %
	\end{tikzpicture}%
\end{equation}%
Note that the pairs $\vect{C},-\vect{C}$
and $i\vect{C},-i\vect{C}$ are isomorphic in $\Sigma_{\mathfrak{C}}$
while any other pair,
such as $\vect{C},i\vect{C}$, is not.
In general, it is moreover possible for two or more tree components of
$G \cdot_{\mathfrak{C}} \vect{C}$ to be equal.
\end{example}

Applying \eqref{REPSFUN EQ} to $G$-objects,
\cite[Prop. \ref{OC-REPALTDESC PROP}]{BP_FCOP}
gives, for each $\mathfrak{C}$-profile $\vect{C}$,
an identification
\begin{equation}\label{REPALTDESC EQ}
(G \ltimes \Sigma^{op}_{\mathfrak{C}})(\vect{C},-)
\simeq 
\Sigma_{\mathfrak{C}} [G \cdot_{\mathfrak{C}} \vect{C}].
\end{equation}
In other words, 
this coproduct
$\Sigma_{\mathfrak{C}} [G \cdot_{\mathfrak{C}} \vect{C}]$
of non-equivariant representables (see \eqref{GENSIGC EQ})
with its inherited $G$-action
is in fact
the representable functor for $\vect{C}$ in
$\mathsf{Sym}^G_{\mathfrak{C}} \simeq
\mathsf{Set}^{G \ltimes \Sigma^{op}_{\mathfrak{C}}}$.

\begin{remark}\label{GCDOTCATS REM}
If  $C \in \Sigma$ is the $n$-corolla,
one has a natural identification
$\boldsymbol{E}(G\cdot C) = G \times \underline{n}_+$
where
$\underline{n}_+ = \{0,1,\cdots,n\}$.
The automorphisms of
$G \cdot C$ in $\Phi^G$
are then naturally identified with the group
$G^{op} \times \Sigma_n$,
with the automorphism
${(g,\sigma)} \colon 
G \cdot C \to G \cdot C$
given on edges by
$(\bar{g},i) \mapsto (\bar{g}g,\sigma(i))$.
\end{remark}

\begin{remark}\label{COLCHSQ REM}
Let $g\vect{C} \sigma = \vect{C'}$
be as in \eqref{CSYMG EQ}.
Then $\vect{C},\vect{C'}$
have the same underlying corolla $C$ and, 
writing $\mathfrak{c},\mathfrak{c}'\colon 
\boldsymbol{E}(C)=\underline{n}_+ \to \mathfrak{C}$
for the colorings,
one can rewrite $g\vect{C} \sigma = \vect{C'}$
as 
$g \mathfrak{c} \sigma = \mathfrak{c}'$.

Definition \ref{GCC_DEF} then induces a diagram in $\Phi^G_{\bullet}$
as below,
with the vertical maps given by  
$G \cdot C \xrightarrow{(g,\sigma)} G \cdot C$
on the underlying forest.
\begin{equation}\label{COLCHSQ EQ}
\begin{tikzcd}
G \cdot C^{\tau} \ar{d}[swap]{(g,\sigma)} 
\ar{r}{G \cdot \mathfrak{c}'}
&
G \cdot_{\mathfrak{C}} \vect{C'}
\ar{d}{(g,\sigma)}
\\
G \cdot C^{\tau} \ar{r}[swap]{G \cdot \mathfrak{c}}
&
G \cdot_{\mathfrak{C}} \vect{C}
\end{tikzcd}
\end{equation}
The right vertical map 
is in fact in $\Phi^G_{\mathfrak{C}}$,
i.e. it respects colors.
Note that this reflects \eqref{REPALTDESC EQ},
which identifies a map $\vect{C} \to \vect{C'}$ in
$G \ltimes \Sigma_{\mathfrak{C}}^{op}$
with a map 
$\Sigma_{\mathfrak{C}}[G \cdot_{\mathfrak{C}} \vect{C'}]
\to
\Sigma_{\mathfrak{C}}[G \cdot_{\mathfrak{C}} \vect{C}]$
in $\mathsf{Sym}^G_{\mathfrak{C}}$.
\end{remark}

\begin{example}
In Example \ref{GDOTCC EX}
the permutation $(14)(23) \in \Sigma_4$
gives a map $\vect{C} \to -\vect{C}$ in $\Sigma_{\mathfrak{C}}$,
and thus induces an automorphism of
$\vect{C}$ in $G \ltimes \Sigma_{\mathfrak{C}}^{op}$.
\end{example}

\subsection{Homotopy theory of equivariant operads with fixed colors}\label{COLFIXMOD SEC}

In this section we recall the model structures on fixed color operads 
$\mathsf{Sym}^G_{\mathfrak{C}}(\V)$
in 
\cite[Thm. \ref{OC-THMI}]{BP_FCOP},
which was the main result therein.
We first recall and elaborate on the $(G,\Sigma)$-families in Definition \ref{FAM1ST DEF}.

\begin{definition}\label{GSFAM_DEF}
	A \emph{$(G,\Sigma)$-family} $\mathcal{F}$ is 
	a collection $\{\mathcal{F}_n\}_{n \geq 0}$ of families $\F_n$ of the groups 
	$G \times \Sigma_n^{op}$.
	
	Further, for $G \ltimes \Sigma^{op}_{\mathfrak{C}}$
	and $n$-ary $\mathfrak{C}$-profile $\vect{C}$,
	we write
	\begin{equation}\label{FVECTC EQ}
                \mathcal{F}_{\vect{C}} =
                \{\Lambda \in \mathcal{F}_n \ | \ \Lambda \text{ stabilizes } \vect{C}\}.
	\end{equation}
\end{definition}

\begin{remark}
	The collection $\mathcal{F} = \{\mathcal{F}_n\}_{n \geq 0}$
	can also be viewed as a family of subgroups
	in the groupoid $G \times \Sigma^{op}$,
	cf. \cite[Def. \ref{OC-FAMGROUPOID DEF}]{BP_FCOP}.
	Similarly,
	by \eqref{CSYMG EQ},
	the subgroups in $\mathcal{F}_{\vect{C}}$
	can be regarded as automorphisms of $\vect{C}$
	in $G \ltimes \Sigma^{op}_{\mathfrak{C}}$,
	so that
	$\mathcal{F}_{\mathfrak{C}}
	=\{ \mathcal{F}_{\vect{C}}\}_{\vect{C} \in \Sigma_{\mathfrak{C}}}$
	similarly defines a family in the groupoid
	$G \ltimes \Sigma^{op}_{\mathfrak{C}}$.
	For further discussion,
	see \cite[Def. \ref{OC-GSFAM_DEF} and Rem. \ref{OC-FAMC_DEF_EQ}]{BP_FCOP}.
\end{remark}

In this paper and the sequel \cite{BP_TAS}, 
we are interested in three main examples of $(G,\Sigma)$-families:

\begin{enumerate}[label = (\alph*)]
\item First, there is the family $\F_{all}$ of all the subgroups of $G \times \Sigma^{op}$
      (in which case the $\F_{all,\mathfrak{C}}$ are also the families of all subgroups), which is useful mainly for technical purposes.
      
\item Secondly, there is the family of $\F^{\Gamma}$
      of $G$-graph subgroups (e.g. \cite[Def. 6.36]{BP21}),
      where $\F^{\Gamma}_n$ consists of the subgroups
      $\Gamma \leq G \times \Sigma_n^{op}$
      such that $\Gamma \cap \Sigma_n^{op} = \{\**\}$.
      We note that the elements of such $\Gamma$
      have the form $(h,\phi(h)^{-1})$
      for $h$ ranging over some subgroup $H \leq G$
      and $\phi \colon H \to \Sigma_n$
      a homomorphism,
      motivating the ``graph subgroup'' terminology.
      Though secondary for the current paper, 
      we regard $\F^{\Gamma}$ as the ``canonical choice'' of
      $(G,\Sigma)$-family, 
      as it is the family featured in the Quillen equivalence
      $W_! \colon 
      \mathsf{dSet}^G \rightleftarrows 
      \mathsf{dSet}^G \colon hcN$
      in \cite[Thm. I]{BP_TAS} 
      (see \eqref{BPMAINTHM_EQ}).
      
\item Lastly, there are the indexing systems of Blumberg and Hill,
      which are special subfamilies of $\F^{\Gamma}$
      that share the key technical properties of 
      $\F^{\Gamma}$ itself,
      and are discussed in 
      \cite[\S \ref{OC-INDSYS_SEC}]{BP_FCOP}.
\end{enumerate}

\begin{example}
Let $G = \mathbb{Z}_{/2} = \{\pm 1\}$ and 
$\mathfrak{C} = \{\mathfrak{a}, -\mathfrak{a}, \mathfrak{b}\}$ where we implicitly have
$-\mathfrak{b} = \mathfrak{b}$.
Consider the two $\mathfrak{C}$-corollas 
$\vect{C},\vect{D} \in \Sigma_{\mathfrak{C}}$ below.
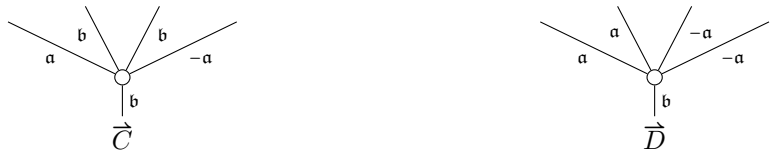
\begin{equation}
	\begin{tikzpicture}[auto,grow=up, level distance = 2.2em,
	every node/.style={font=\scriptsize,inner sep = 2pt}]%
		\tikzstyle{level 2}=[sibling distance=3em]%
			\node at (0,0) [font = \normalsize] {$\vect{C}$}%
				child{node [dummy] {}%
					child{node {}%
					edge from parent node [swap] {$-\mathfrak{a}$}}%
					child[level distance = 2.9em]{node {}%
					edge from parent node [swap,	near end] {$\mathfrak{b}$}}%
					child[level distance = 2.9em]{node {}%
					edge from parent node [near end] {$\mathfrak{b}$}}%
					child{node {}%
					edge from parent node  {$\mathfrak{a}$}}%
				edge from parent node [swap] {$\mathfrak{b}$}};%
			\node at (7,0) [font = \normalsize] {$\vect{D}$}%
				child{node [dummy] {}%
					child{node {}%
					edge from parent node [swap] {$-\mathfrak{a}$}}%
					child[level distance = 2.9em]{node {}%
					edge from parent node [swap,	near end] {$-\mathfrak{a}$}}%
					child[level distance = 2.9em]{node {}%
					edge from parent node [near end] {$\mathfrak{a}$}}%
					child{node {}%
					edge from parent node  {$\mathfrak{a}$}}%
				edge from parent node [swap] {$\mathfrak{b}$}};%
	\end{tikzpicture}%
\end{equation}%
The non-trivial $G$-graph subgroups of
$\F^{\Gamma}_{\vect{C}}$,
$\F^{\Gamma}_{\vect{D}}$
correspond to the possible $\mathbb{Z}_{/2}$-actions on the underlying trees $C,D$ that are compatible with the action on labels
(in that the composites
$\boldsymbol{E}(C) \xrightarrow{-1} \boldsymbol{E}(C) \to \mathfrak{C}$
and 
$\boldsymbol{E}(C) \to \mathfrak{C} \xrightarrow{-1} \mathfrak{C}$ coincide).
In this case, both 
$\F^{\Gamma}_{\vect{C}}$,
$\F^{\Gamma}_{\vect{D}}$
have exactly two non-trivial groups,
corresponding to the $\mathbb{Z}_{/2}$-actions on the underlying corollas depicted below.
\begin{equation}
	\begin{tikzpicture}[auto,grow=up, level distance = 2.2em,
	every node/.style={font=\scriptsize,inner sep = 2pt}]%
		\tikzstyle{level 2}=[sibling distance=3em]%
			\node at (-1.6,0) [font = \normalsize] {$C_1$}%
				child{node [dummy] {}%
					child{node {}%
					edge from parent node [swap] {$-a$}}%
					child[level distance = 2.9em]{node {}%
					edge from parent node [swap,	near end] {$c\phantom{b}$}}%
					child[level distance = 2.9em]{node {}%
					edge from parent node [near end] {$b$}}%
					child{node {}%
					edge from parent node  {$a$}}%
				edge from parent node [swap] {$r$}};%
			\node at (1.6,0) [font = \normalsize] {$C_2$}%
				child{node [dummy] {}%
					child{node {}%
					edge from parent node [swap] {$-a$}}%
					child[level distance = 2.9em]{node {}%
					edge from parent node [swap,	near end] {$-b$}}%
					child[level distance = 2.9em]{node {}%
					edge from parent node [near end] {$b$}}%
					child{node {}%
					edge from parent node  {$a$}}%
				edge from parent node [swap] {$r$}};%
			\node at (5.4,0) [font = \normalsize] {$D_1$}%
				child{node [dummy] {}%
					child{node {}%
					edge from parent node [swap] {$-a$}}%
					child[level distance = 2.9em]{node {}%
					edge from parent node [swap,	near end] {$-b$}}%
					child[level distance = 2.9em]{node {}%
					edge from parent node [near end] {$b$}}%
					child{node {}%
					edge from parent node  {$a$}}%
				edge from parent node [swap] {$r$}};%
			\node at (8.6,0) [font = \normalsize] {$D_2$}%
				child{node [dummy] {}%
					child{node {}%
					edge from parent node [swap] {$-b$}}%
					child[level distance = 2.9em]{node {}%
					edge from parent node [swap,	near end] {$-a$}}%
					child[level distance = 2.9em]{node {}%
					edge from parent node [near end] {$b$}}%
					child{node {}%
					edge from parent node  {$a$}}%
				edge from parent node [swap] {$r$}};%
	\end{tikzpicture}%
\end{equation}%
\end{example}

As discussed in 
\cite[Def. \ref{OC-SYMGFV DEF} and Rem. \ref{OC-VGSIGF REM}]{BP_FCOP},
we then have the following instance of 
\cite[Prop. \ref{OC-ALLEQ PROP}]{BP_FCOP},
where
$\mathcal{I}$ (resp. $\mathcal{J}$)
denotes the generating (trivial) cofibrations of $\V$.

\begin{proposition}\label{SYMGFV PROP}
	Let $\V$ satisfy
	(i),(ii) in Theorem \ref{THMA}.
	Fix $\mathfrak{C} \in \mathsf{Set}^G$
	and $(G,\Sigma)$-family $\F$.
	
	Then there exists a model structure on
	$\mathsf{Sym}^G_{\mathfrak{C}}(\mathsf{\V})$,
	which we call the \emph{$\mathcal{F}$-model structure}
	and denote $\mathsf{Sym}^G_{\mathfrak{C},\F}(\V)$,
	such that a map $X \to Y$
	is a weak equivalence (resp. fibration) if the maps
\begin{equation}
	X(\vect{C})^{\Lambda} \to Y(\vect{C})^{\Lambda}
\end{equation}
	are weak equivalences (fibrations) in $\V$
	for all $\mathfrak{C}$-profiles $\vect{C}$
	and $\Lambda \in \F_{\vect{C}}$.	
	
	Moreover, the generating (trivial) cofibrations of
	$\Sym^{G}_{\mathfrak{C},\F}(\V)$
	are the sets of maps
	\begin{equation}\label{VGSIGF EQ}
	\mathcal{I}_{\mathfrak{C},\mathcal{F}}
	=
	\left\{
	\Sigma_{\mathfrak{C}}[G \cdot_{\mathfrak{C}} \vect{C}]/\Lambda \cdot i
	\right\}
	\qquad \qquad
	\mathcal{J}_{\mathfrak{C},\mathcal{F}}
	=
	\left\{
	\Sigma_{\mathfrak{C}}[G \cdot_{\mathfrak{C}} \vect{C}]/\Lambda \cdot j
	\right\}
	\end{equation}
	where $\vect{C}$ ranges over $\Sigma_{\mathfrak{C}}$,
	$\Lambda$ ranges over $\F_{\vect{C}}$,
	$i$ ranges over $\mathcal{I}$ and
	$j$ ranges over $\mathcal{J}$.
\end{proposition}

Transfer along the adjunction
$
\mathbb{F}^G_{\mathfrak{C}} \colon
\mathsf{Sym}^G_{\mathfrak{C}}(\V)
\rightleftarrows
\mathsf{Op}^G_{\mathfrak{C}}(\V)
\colon \mathsf{fgt}
$
then yields the following.

\begin{theorem}
	[{\cite[Thm. \ref{OC-THMI}]{BP_FCOP}}]
	\label{THMIREST}
	Let $\V$ satisfy (i),(ii),(iii),(iv),(v) in Theorem \ref{THMA}.
	Fix $\mathfrak{C} \in \mathsf{Set}^G$
	and $(G,\Sigma)$-family $\F$.
	
	Then there exists a model structure on
	$\mathsf{Op}^G_{\mathfrak{C}}(\mathsf{\V})$,
	which we call the \emph{$\mathcal{F}$-model structure}
	and denote $\mathsf{Op}^G_{\mathfrak{C},\F}(\V)$,
	such that a map
	$\mathcal{O} \to \mathcal{P}$
	is a weak equivalence (resp. fibration) if the maps
	\begin{equation}\label{THMI_EQ}
	\O(\vect{C})^{\Lambda} \to \mathcal{P}(\vect{C})^{\Lambda}
	\end{equation}
	are weak equivalences (fibrations)
	in $\V$ for all $\mathfrak{C}$-profiles $\vect{C}$
	and $\Lambda \in \F_{\vect{C}}$.
	
	Moreover, the generating (trivial) cofibrations in
	$\mathsf{Op}^G_{\mathfrak{C},\F}(\V)$
	are the sets
	\begin{equation}\label{FVGSIGF EQ}
	\mathbb{F}^G_{\mathfrak{C}}\mathcal{I}_{\mathfrak{C},\mathcal{F}}
	=
	\left\{
	\mathbb{F}^G_{\mathfrak{C}}
	\left(\Sigma_{\mathfrak{C}}[G \cdot_{\mathfrak{C}} \vect{C}]/\Lambda \cdot i \right)
	\right\}
	\qquad \qquad
	\mathbb{F}^G_{\mathfrak{C}}\mathcal{I}_{\mathfrak{C},\mathcal{F}}
	=
	\left\{
	\mathbb{F}^G_{\mathfrak{C}}
	\left(\Sigma_{\mathfrak{C}}[G \cdot_{\mathfrak{C}} \vect{C}]/\Lambda \cdot j \right)
	\right\}
	\end{equation}
	where $\vect{C}$ ranges over $\Sigma_{\mathfrak{C}}$,
	$\Lambda$ ranges over $\F_{\vect{C}}$,
	$i$ ranges over $\mathcal{I}$,
	and $j$ ranges over $\mathcal{J}$.
\end{theorem}

\begin{remark}\label{FALLMAXMIN REM}
	When $\mathcal{F}=\mathcal{F}_{all}$ is the family of all subgroups,
	we refer to these model structures on 
	$\mathsf{Sym}^G_{\mathfrak{C}}(\V)$, $\mathsf{Op}^G_{\mathfrak{C}}(\V)$
	as the \emph{genuine model structures}.

	Further, note that the genuine model structures 
	minimize the classes of weak equivalences and fibrations
	and thus, conversely, 
	they maximize the classes of cofibrations and trivial cofibrations.
\end{remark}

We now recall the following,
where the $\phi^{\**} \bar{\F}$ families
are as defined in \cite[Rem. \ref{OC-PULLFAM REM}]{BP_FCOP}.

\begin{corollary}[{\cite[Cor. \ref{OC-OPADJ_COR}]{BP_FCOP}}]
	\label{OPADJ_COR}
	\begin{enumerate}[label=(\roman*)]
		\item \label{OPCOCHADJ_LBL}
		For any $(G,\Sigma)$-family $\F$ and map of colors 
		$\varphi \colon \mathfrak C \to \mathfrak D$, the induced adjunction
		\[
		\check{\varphi}_! \colon \mathsf{Op}^G_{\mathfrak{C},\F}
		\rightleftarrows
		\mathsf{Op}^G_{\mathfrak{D},\F} \colon \varphi^{\**}
		\]
		is a Quillen adjunction.
		\item \label{OPFIXSETCHGR_LBL}
		For any homomorphism $\phi \colon G \to \bar G$,
		$(G,\Sigma)$-family $\F$ and $(\bar G,\Sigma)$-family $\bar{\F}$,
		and $\bar G$-set of colors $\mathfrak C$,
		the adjunction
		\[
		\check{\phi}_! \colon \mathsf{Op}^G_{\mathfrak{C},\F}
		\rightleftarrows
		\mathsf{Op}^{\bar{G}}_{\mathfrak{C},\bar{\F}} \colon \phi^{\**}
		\]
		is a Quillen adjunction whenever $\F \subseteq \phi^{\**} \bar{\F}$, i.e.
		if $\Lambda \in \F_n$ implies 
		$\phi(\Lambda) \in \bar{\F}_n$.
		\item \label{OPCOMBADJ_LBL}
		For any homomorphism $\phi \colon G \to \bar G$,
		$(G,\Sigma)$-family $\F$ and $(\bar G,\Sigma)$-family $\bar{\F}$,
		and $G$-set of colors $\mathfrak C$,
		the adjunction
		\[
		\bar{G} \cdot_G (-) \colon \mathsf{Op}^G_{\mathfrak{C},\F}
		\rightleftarrows
		\mathsf{Op}^{\bar{G}}_{\bar{G} \cdot_G \mathfrak{C},\bar{\F}} \colon \mathsf{fgt}
		\]
		is a Quillen adjunction whenever $\F \subseteq \phi^{\**} \bar{\F}$, i.e.
		if $\Lambda \in \F_n$ implies 
		$\phi(\Lambda) \in \bar{\F}_n$.
	\end{enumerate}
\end{corollary}

The proof of Theorem \ref{THMA}
(cf. Proposition \ref{J_CELL_PROP})
will use an additional class of maps in
$\mathsf{Sym}^G_{\mathfrak C}(\V)$.

\begin{definition}[{\cite[Def. \ref{OC-GGENOTITC DEF}]{BP_FCOP}}]
	\label{GGENOTITC DEF}
	We write 
	$\mathcal{J}^{\otimes}_{\mathfrak{C}}=
	\{j \otimes X \ | \ j \in \mathcal{J}, X \in \mathsf{Sym}^G_{\mathfrak C}(\V)\}$,
	and refer to the saturation
	$\mathcal{J}^{\otimes}_{\mathfrak{C}}$-cof
	as the \emph{genuine $\otimes$-trivial cofibrations}
	in $\mathsf{Sym}^G_{\mathfrak C}(\V)$.	
\end{definition}

\begin{remark}[{\cite[Rem. \ref{OC-GOTC_REM}]{BP_FCOP}}]
	\label{GOTC_REM}
	$\F$-trivial cofibrations in $\Op^G_{\mathfrak C}(\V)$ are underlying genuine $\otimes$-trivial cofibrations
	in $\mathsf{Sym}^G_{\mathfrak C}(\V)$.
\end{remark}

\begin{proposition}\label{GOTC_PROP}
	\begin{itemize}
		\item [(i)] 
		If $\V$ satisfies the global monoid axiom
		((iv) in Theorem \ref{THMA}),
		genuine $\otimes$-trivial cofibrations
		in $\mathsf{Sym}^G_{\mathfrak C}(\V)$
		are genuine weak equivalences.
		\item[(ii)]
		The functors
		$\varphi_! \colon 
		\mathsf{Sym}^G_{\mathfrak C}(\V)
		\to 
		\mathsf{Sym}^G_{\mathfrak D}(\V)$
		and
		$\varphi^{\**} \colon 
		\mathsf{Sym}^G_{\mathfrak D}(\V)
		\to 
		\mathsf{Sym}^G_{\mathfrak C}(\V)$
		preserve genuine $\otimes$-trivial cofibrations
		for any color map
		$\varphi \colon \mathfrak{C} \to \mathfrak{D}$.
	\end{itemize}
\end{proposition}

\begin{proof}
	(i) just restates the global monoid axiom
	\cite[Def. \ref{OC-GLOBMONAX_DEF}]{BP_FCOP} in light of the observation that genuine $\otimes$-trivial cofibrations/genuine weak equivalences 
	are characterized levelwise (see
	\cite[Rem. \ref{OC-SIGMACOF_REM}, Def. \ref{OC-GGENOTITC DEF}]{BP_FCOP}).
	Part (ii) is \cite[Prop. \ref{OC-REGEOTCOF PROP}]{BP_FCOP}.
\end{proof}

\begin{remark}
        \label{GOTC_REM2}
	In practice, 
	the properties of genuine $\otimes$-trivial cofibrations
	given above
	serve a similar function to the compact generation condition
	in \cite[Def. 1.2]{BM13}.

    More broadly, the global monoid axiom, together with the cofibrant symmetric pushout powers condition
    (conditions (iv),(v) in Theorem \ref{THMA}),
    play a role analogous to the hypotheses of \textit{adequacy} in \cite[Prop. 1.4]{BM13},
    or, for example,
    \textit{$h$-monoidality} in \cite{BB17}.
	However, these conditions differ in their goals:
	the conditions herein serve to build model structures 
	with weak equivalences determined by fixed point conditions (cf. Theorem \ref{THMA}),
	whereas adequacy, $h$-monoidality, and its variants serve to build projective model structures.
	As such, from a conceptual point of view one may regard the conditions herein as ``genuine analogues'' of adequacy and $h$-monoidality,
	though we caution that in practice
	our conditions exhibit different formal properties
	that often require notable modifications to proofs
	(cf. Remark \ref{GTRIV REM}, \cite[Rem. 6.18]{BP21}, \cite[Rem. 4.8]{BP_FCOP}).
\end{remark}

\section{Model structures on all equivariant colored operads}
\label{MS_SEC}

\renewcommand{\C}{\mathfrak C}

Theorem \ref{THMIREST} provides,
for each $(G,\Sigma)$-family $\F$,
a model structure on each category
$\mathsf{Op}_{\mathfrak{C}}^G(\V)$
of $G$-equivariant operads with a fixed $G$-set of colors $\mathfrak{C}$.
Adapting \cite{BM13,Cav,CM13b},
our main goal in this section is 
to prove our main result, Theorem \ref{THMA},
which uses the model structures of Theorem \ref{THMIREST}
to build,
for suitable $(G,\Sigma)$-families $\F$ 
(see Definition \ref{FAMRESUNI DEF}),
a model structure on the full category $\mathsf{Op}^G_\bullet(\mathcal{V})$
of $G$-equivariant operads with varying $G$-sets of colors.

As stated in the formulation of Theorem \ref{THMA},
the weak equivalences and trivial fibrations 
in $\mathsf{Op}^G_\bullet(\V)$
are described by combining
a ``local condition'' that involves the fixed color categories $\mathsf{Op}_{\mathfrak{C}}^G(\V)$,
as in \eqref{THMIII1ST EQ},
with a form of ``surjectivity on objects'',
as in \eqref{THMIII2ND EQ}.
However, the essential surjectivity condition for 
weak equivalences in \eqref{THMIII2ND EQ}
has a key technical drawback: 
when using this condition 
it is unclear how to select a generating set of trivial cofibrations 
for $\mathsf{Op}^G_\bullet(\V)$.

For this reason, 
throughout the bulk of this section we will actually work with 
an alternate (and a priori distinct)
notion of weak equivalence,
defined using a more abstract notion of essential surjectivity,
which will also allow us to characterize the fibrations (Definition \ref{MODEL_DEFN}).

This model structure will be built in Sections \ref{MAPSOPG_SEC} through \ref{HMTYEQ SEC}.
\S \ref{MAPSOPG_SEC} introduces the relevant classes of maps of operads,
\S \ref{GENCOF SEC} produces a generating set of (trivial) cofibrations in Definition \ref{OPGENCOF DEF},
and \S \ref{TRIVCOF_SEC} proves in Proposition \ref{J_CELL_PROP} that trivial cofibrations are in fact weak equivalences.
Sections \ref{EQUIVOBJ_SEC} and \ref{HMTYEQ SEC} explore several notions of essential surjectivity in order to prove 2-out-of-3
(Proposition \ref{2OUTOF3 PROP})
and show that the weak equivalences in Definition \ref{MODEL_DEFN}
and Theorem \ref{THMA} indeed match
(Corollary \ref{WEDKEQ COR}).

Lastly, \S \ref{ISOFIB_SEC}
gives (under mild conditions) a more familiar description of fibrations in $\mathsf{Op}^G_\bullet(\V)$.

\subsection{Classes of maps in $\mathsf{Op}^G_\bullet(\mathcal V)$}
\label{MAPSOPG_SEC}

We now discuss the several types of maps in 
$\mathsf{Op}^G_\bullet(\V)$ we will be interested in,
starting with the ``local'' notions, 
i.e. those notions determined by the fixed color categories $\mathsf{Op}_{\mathfrak{C}}^G(\V)$.

\begin{definition}
Let $\F$ be a $(G, \Sigma)$-family.
We say a map $F: \O \to \P$ in $\mathsf{Op}^G_\bullet(\V)$
is a \emph{local $\F$-weak equivalence (resp. local $\F$-fibration, local $\F$-trivial fibration)}
if the induced fixed color map
($F^{\**}$ is as in \eqref{GC_CHANGE_EQ})
\[\O \to F^{\**} \P\]
is a $\F_{\mathfrak{C}_{\O}}$-weak equivalence (resp. $\F_{\mathfrak{C}_{\O}}$-fibration, $\F_{\mathfrak{C}_{\O}}$-trivial fibration) in the fiber $\mathsf{Op}^G_{\mathfrak{C}_{\O}}(\V)$.
\end{definition}

Local (trivial) $\F$-fibrations admit the following alternative characterization.

\begin{proposition}\label{LOCALTCHAR PROP}
Suppose $\V$ is as in Theorem \ref{THMIREST}.
The local $\F$-fibrations
and local trivial $\F$-fibrations
in $\mathsf{Op}^G_\bullet(\V)$
are characterized as the maps with the right lifting property against the generating sets of maps 
$\mathbb{F}^G_{\mathfrak{C}}\mathcal{J}_{\mathfrak{C},\mathcal{F}}$ and  $\mathbb{F}^G_{\mathfrak{C}}\mathcal{I}_{\mathfrak{C},\mathcal{F}}$
(cf. \eqref{FVGSIGF EQ})
of the fibers 
$\mathsf{Op}^G_{\mathfrak{C}}(\V) \hookrightarrow \mathsf{Op}^G_\bullet(\V)$
for all $\mathfrak{C} \in \mathsf{Set}^G$.
\end{proposition}

\begin{proof}
Note first that,
for a square in $\mathsf{Op}^G_{\bullet}(\V)$
as on the left below and where $A_1 \to A_2$ is a color fixed map, 
the lifting problems for all three squares given below are equivalent.
\begin{equation}
\begin{tikzcd}
A_1 \arrow[d] \arrow[r, "a"]
&
\mathcal{O} \arrow[d, "F"]
&&
A_1 \arrow[d] \arrow[r, "a"]
&
\mathcal{O} \arrow[d]
&&
A_1 \arrow[d] \arrow[r]
&
a^{\**} \mathcal{O} \arrow[d]
\\
A_2 \arrow[r]
&
\P
&&
A_2 \arrow[r]
&
F^{\**} \P
&&
A_2 \arrow[r]
&
a^{\**} F^{\**} \P
\end{tikzcd}
\end{equation}
Writing $\mathfrak{C}$ for the colors of the $A_i$ and
$\mathfrak{D}$ for the colors of $\O$,
the result follows since the pullback functors
$a^{\**} \colon \mathsf{Op}^G_{\mathfrak{D},\F}(\V)
\to \mathsf{Op}^G_{\mathfrak{C},\F}(\V)$
preserve (trivial) fibrations.
\end{proof}

We next turn to the homotopical notions of essential surjectivity and isofibration, 
which concern equivalences between objects within some 
$\O \in \mathsf{Op}^G_\bullet(\V)$.
We first recall some notions from \cite{BM13}.
As usual, we let $1_\V$ and $\emptyset$ denote, respectively, the unit object and initial object of $\V$.

\begin{notation}\label{1_NOT}
      We write $\mathbbm{1}$ (resp. $\widetilde{\mathbbm{1}}$)
      for the $\V$-category that represents arrows (resp. isomorphisms):
      it has two objects $0,1$
      and mapping objects
      $\mathbbm{1}(i,j)= 1_{\V}$
      if $i \leq j$
      and 
      $\mathbbm{1}(1,0)= \emptyset$
      (resp. $\widetilde{\mathbbm{1}}(i,j)= 1_{\V}$ for all $i,j$)
      and composition defined by the unit isomorphisms of $\otimes$.

      Further, we write $\eta$ for the $\V$-category that represents objects,
      with a single object $\**$ and $\eta(\**,\**) = 1_\V$.
\end{notation}

In the following, and throughout, 
we give $\Cat_{\set{0,1}}(\V)$
its projective model structure.

\begin{definition}\label{VINTER DEF}
      A {\em $\V$-interval} is a cofibrant object $\mathbb{J}$ in $\Cat_{\set{0,1}}(\V)$
      that is equivalent to $\widetilde{\mathbbm{1}}$.
\end{definition}

\begin{example}
	The prototypical example of a $\V$-interval is the simplicial category $W_!J \in \Cat_{\set{0,1}}(\sSet)$,
	where
	$J = N \widetilde{[1]} = N(0 \rightleftarrows 1) $ is the nerve of the walking isomorphism category 
	$\widetilde{[1]} =  (0 \rightleftarrows 1)$ and
	$W_! \colon \sSet \to \Cat_{\bullet}(\sSet)$ is the left adjoint to the homotopy coherent nerve
	of \cite{Cor82} (see e.g. \cite[\S 1]{Joy02}).
\end{example}

\begin{remark}
	We note that, since $\widetilde{\mathbbm{1}}$ is typically not fibrant,
	an arbitrary interval $\mathbb{J}$
	needs not admit a map to $\widetilde{\mathbbm{1}}$,
	but only a map $\mathbb{J} \to \widetilde{\mathbbm{1}}_f$,
	where $\widetilde{\mathbbm{1}}_f$ denotes some fixed chosen fibrant replacement.
\end{remark}

Informally, $\V$-intervals detect ``homotopical isomorphisms'' in a $\V$-category $\mathcal{C}$ 
(this idea is formalized in Definition \ref{EQUIV_DEF} below).
Mimicking the definitions of isofibration and essentially surjective functor of (unenriched) categories, we have the following.

\begin{definition}\label{PL_ES_DEFN}
We say a functor $F: \mathcal C \to \mathcal D$ in $\Cat(\V)$ is
\begin{itemize}
\item \textit{path-lifting}
	if it has the right lifting property against all maps of the form
	$\eta \xrightarrow{0} \mathbb{J}$, $\eta \xrightarrow{1} \mathbb{J}$
	where $\mathbb{J}$ is a $\V$-interval;
\item \textit{essentially surjective} 
	if, for any object $d \in \mathcal{D}$,
	there is an object $c \in \mathcal{C}$,
	$\V$-interval $\mathbb{J}$,
	and map $i \colon \mathbb{J} \to \mathcal D$
	such that $i(0) = F(c)$ and $i(1)=d$.
      \end{itemize}
\end{definition}

We now adapt the previous definition for $G$-operads.
Recall that $j^{\**} \colon \mathsf{Op}^G_\bullet(\V) \to \mathsf{Cat}^G_{\bullet}(\V)$
denotes the functor that forgets all non-unary operations and that, moreover, 
$j^{\**}$ commutes with all fixed points $(-)^H$.

\begin{definition}\label{FESSENSURJ DEF}
Let $\F$ be a $(G, \Sigma)$-family that has enough units
(Definition \ref{FAMRESUNI DEF}).

We say a map $F: \O \to \P$ in $\mathsf{Op}^G_\bullet(\V)$
is $\F$-essentially surjective (resp. $\F$-path-lifting)
if the maps
$j^{\**}\O^H \to j^{\**} \P^H$
in $\mathsf{Cat}(\V)$ are essentially surjective (path-lifting) for all $H \in \F_1$.
\end{definition}

We can finally define the classes of maps in the desired model structures on $\mathsf{Op}^G_\bullet(\V)$.

\begin{definition}\label{MODEL_DEFN}
Let $\F$ be a $(G, \Sigma)$-family that has enough units
(Definition \ref{FAMRESUNI DEF}).

We say a map $F: \O \to \P$ in $\mathsf{Op}^G_\bullet(\V)$ is:
\begin{itemize}
	\item a {\em $\F$-fibration} if it is both a local $\F$-fibration and $\F$-path lifting;
	\item a {\em $\F$-weak equivalence} if it is both a local $\F$-weak equivalence and $\F$-essentially surjective;
	\item a \textit{$\F$-cofibration} if it has the left lifting property against all trivial $\F$-fibrations (i.e. $\F$-fibrations that are also $\F$-weak equivalences).
\end{itemize}
\end{definition}

Throughout the remainder of \S \ref{MS_SEC}
we will prove that 
Definition \ref{MODEL_DEFN} describes the model structure on 
$\mathsf{Op}^G_\bullet(\V)$
in Theorem \ref{THMA},
which we denote by
$\mathsf{Op}^G_{\bullet, \F}(\V)$.

We first show that $\F$-trivial fibrations
indeed satisfy the characterization given in Theorem \ref{THMA},
adapting \cite[4.8]{Cav}, \cite[2.3]{BM13}, \cite[1.18]{CM13b}.

\begin{proposition}\label{FTRIVCHAR PROP}
A map in $F: \O \to \P$ in $\mathsf{Op}^G_\bullet(\V)$ is 
a $\F$-trivial fibration (i.e. both a $\F$-fibration and a $\F$-weak equivalence) 
iff it is a local $\F$-trivial fibration
such that the induced map (of sets) $\mathfrak C_\O^H \to \mathfrak C_\P^H$ on $H$-fixed colors is surjective
for all $H \in \F_1$.
\end{proposition}

\begin{proof}
It is enough to show that,
if $F \colon \O \to \P$ is a local $\F$-trivial fibration,
then $F$ is both $\F$-path lifting and $\F$-essentially surjective
iff  the induced map on $H$-fixed colors is surjective for all $H \in \F_1$.

For the ``if'' direction, it is immediate that $F$ is
$\F$-essentially surjective, so it remains to show that
the maps $\O^H\to \P^H$ for $H \in \F_1$
have the right lifting property against the maps
$\eta \to \mathbb{J}$.
But this follows by factoring the latter maps as
$\eta \to \eta \amalg \eta \to \mathbb{J}$, 
since the lifting property against 
$\eta \to \eta \amalg \eta$ follows from surjectivity on $H$-fixed objects while the lifting property against
$\eta \amalg \eta \to \mathbb{J}$
follows from Proposition \ref{LOCALTCHAR PROP}
and the fact that 
$\eta \amalg \eta \to \mathbb{J}$
is a cofibration in $\mathsf{Cat}_{\{0,1\}}(\V)$
(given that $\eta \amalg \eta$ is the initial object of $\mathsf{Cat}_{\{0,1\}}(\V)$
while $\mathbb{J}$ is cofibrant by definition of $\V$-interval).

For the ``only if'' direction, let $y \in \P^H$ be an $H$-fixed object with $H \in \F_1$.
$\F$-essential surjectivity yields an $x \in \O^H$ and map 
$i \colon \mathbb{J} \to \P^H$
with $i(0)=F(x)$, $i(1)=y$.
The $\F$-path lifting property then gives a lift $\tilde{i}$ as below, 
so that $\tilde{i}(1)$ gives the desired lift of $y$.
\[
\begin{tikzcd}
	\eta \ar{d}[swap]{0} \ar{r}{x}  
&
	\O^H \arrow{d}{F}
\\
	\mathbb{J} \ar{r}[swap]{i} \ar[dashed]{ru}[swap]{\tilde{i}}
&
	\P^H
\end{tikzcd}
\]   
\end{proof}

\begin{proposition}\label{FIBERGLMOD PROP}
A fixed color map
$\O \to \P$ in
$\mathsf{Op}^G_{\mathfrak{C}}(\V) \subseteq \Op^G_\bullet(\V)$ is:
\begin{enumerate}[label=(\roman*)]
\item a $\F$-weak equivalence in the fiber
$\mathsf{Op}_{\mathfrak{C}}^G(\V)$
iff it is a $\F$-weak equivalence in
$\mathsf{Op}_{\bullet}^G(\V)$;
\item a $\F$-cofibration in the fiber 
$\mathsf{Op}_{\mathfrak{C}}^G(\V)$
iff it is a $\F$-cofibration in
$\mathsf{Op}_{\bullet}^G(\V)$;
\item a $\F$-fibration in the fiber 
$\mathsf{Op}_{\mathfrak{C}}^G(\V)$
whenever it is a $\F$-fibration in
$\mathsf{Op}_{\bullet}^G(\V)$.
\end{enumerate}
\end{proposition}

\begin{proof}
	(i) follows since fixed color maps are certainly essentially surjective while (iii) is tautological since 
	$\F$-fibrations in $\mathsf{Op}_{\bullet}^G(\V)$ must be local $\F$-fibrations.
	As for (ii), Proposition \ref{LOCALTCHAR PROP} yields
	the ``if'' direction,
	while the ``only if'' direction of (ii) follows from 
	Proposition \ref{FTRIVCHAR PROP},
	which implies that all $\F$-trivial fibrations in $\mathsf{Op}_{\mathfrak{C}}^G(\V)$
	are $\F$-trivial fibrations in 
	$\mathsf{Op}_{\bullet}^G(\V)$,
	together with the usual lifting 
	property characterization of cofibrations.
\end{proof}

\begin{remark}
In contrast to the other parts of 
Proposition \ref{FIBERGLMOD PROP},
the implication in part (iii) 
only holds in one direction. 
As a counterexample to its converse, 
consider the map $\eta \amalg \eta \to \widetilde{\mathbbm{1}}$
in $\mathsf{Cat}_{\bullet}(\mathsf{sSet})$. This is a local fibration, 
and thus a fibration in $\mathsf{Cat}_{\{0,1\}}(\mathsf{sSet})$,
but not path-lifting, and thus not a fibration in $\mathsf{Cat}_{\bullet}(\mathsf{sSet})$.

Nonetheless, Proposition \ref{FTRIVCHAR PROP}
guarantees that the analogue of Proposition \ref{FIBERGLMOD PROP}
for $\F$-trivial fibrations is indeed an iff.
\end{remark}

\begin{corollary}
      $\O \in \Op^G_{\mathfrak C}(\V)$ is cofibrant in
      $\O \in \Op^G_{\mathfrak C,\mathcal{F}}(\V)$
      iff $\O$ is cofibrant in $\Op^G_{\bullet,\F}(\V)$.
\end{corollary}

\begin{proof}
	The ``if'' direction follows since
	$\F$-trivial fibrations in $ \Op^G_{\mathfrak C}(\V)$
	are
	$\F$-trivial fibrations in $ \Op^G_{\bullet}(\V)$.
	The ``only if'' direction follows since
	$\F$-trivial fibrations in $ \Op^G_{\bullet}(\V)$
	are local $\F$-trivial fibrations.
\end{proof}

The proof of Theorem \ref{THMA} will occupy most of 
the remainder of \S \ref{MS_SEC},
where we will show that the maps in Definition \ref{MODEL_DEFN}
do indeed define a model structure on $\mathsf{Op}^G{\V}$.
The following is the outline of the proof.

\begin{proof}[Proof of Theorem \ref{THMA}]
      As $\sSet$ satisfies all 
      hypotheses in Theorem \ref{THMA}, we prove the general case.

      We will verify the conditions in 
\cite[Theorem 2.1.19]{Hov99}, and we write (1),(2),(3),etc for the conditions therein.

Firstly, in \S \ref{GENCOF SEC} we identify
the generating (resp. trivial) cofibrations of $\mathsf{Op}^G_{\F}(\V)$,
which are given by the sets 
(C1) and (C2) (resp. (TC1) and (TC2))
found in Definition \ref{OPGENCOF DEF}.

The implicit claim that the maps
with the right lifting property against 
(TC1) and (TC2)
are the $\F$-fibrations as given by 
Definition \ref{MODEL_DEFN}
follows from Propositions \ref{LOCALTCHAR PROP} and \ref{GENIN PROP}.

Likewise, the fact that the maps
with the right lifting property against 
(C1) and (C2) are the 
$\F$-trivial fibrations as given by 
Definition \ref{MODEL_DEFN}
is Proposition \ref{FTRIVCHAR PROP}, 
establishing conditions (5),(6).

Lemma \ref{POINT_4_LEMMA} and Proposition \ref{J_CELL_PROP} establishes (4).

(2),(3) follow since colimits in $\mathsf{Op}^G_\bullet(\V)$ are created in $\Op_{\bullet}(\V)$, and it holds non-equivariantly.

Condition (1), i.e. the $2$-out-of-$3$ condition for $\F$-weak equivalences, 
is Proposition \ref{2OUTOF3 PROP}.

Lastly, the fact that the weak equivalences in 
Theorem \ref{THMA}
match the weak equivalences in
Definition \ref{MODEL_DEFN} is given by 
Corollary \ref{WEDKEQ COR}.
\end{proof}

\subsection{Generating cofibrations and trivial cofibrations}
\label{GENCOF SEC}

We next turn to the task of identifying sets of generating cofibrations and generating trivial cofibrations
for the desired model structures
on $\mathsf{Op}^G_{\bullet}(\V)$ determined by Definition \ref{MODEL_DEFN}.

Proposition \ref{LOCALTCHAR PROP} suggests that
the generating sets of maps
$\mathbb{F}^G_{\mathfrak{C}} \mathcal{I}_{\mathfrak{C},\mathcal{F}},
\mathbb{F}^G_{\mathfrak{C}} \mathcal{J}_{\mathfrak{C},\mathcal{F}}$
(cf. \eqref{FVGSIGF EQ}) of the fibers 
$\mathsf{Op}^G_{\mathfrak{C},\F}(\V)$
should be included
in the generating sets of maps for 
$\mathsf{Op}^G_{\bullet,\F}(\V)$.
However, it is inefficient to include all such maps, 
as there is a subset of those maps
that generates the remaining maps under pushouts along change of colors.

To see why, 
consider a representable functor
$\Sigma_{\mathfrak{C}}[G \cdot_{\mathfrak{C}} \vect{C}]$
in
$\mathsf{Sym}^G_{\mathfrak{C}}$
(cf. \eqref{REPALTDESC EQ}),
and write
$\mathfrak{c} \colon 
\boldsymbol{E}(G \cdot C) \to \mathfrak{C}$
for the coloring on the underlying forest $G \cdot C$.
By Remark \ref{GCDOTCATS REM},
the group $G \times \Sigma^{op}$
has a right action on $G \cdot C$ and, 
moreover, a subgroup
$\Lambda \leq G \times \Sigma^{op}$ stabilizes $\vect{C}$
precisely if $\mathfrak{c}$ is $\Lambda$-equivariant,
i.e. if it induces a map 
$\bar{\mathfrak{c}} \colon 
\boldsymbol{E}(G \cdot C)/\Lambda \to \mathfrak{C}$
on orbits
(indeed, this is simply the observation that the right vertical map in 
\eqref{COLCHSQ EQ} respects colors, specified to the case
$\vect{C'} = \vect{C}$).

Combining the identification
$\mathfrak{c}_! \Sigma_{\tau}[G \cdot C]
=
\Sigma_{\mathfrak{C}}[G \cdot_{\mathfrak{C}} \vect{C}]$
in 
\eqref{CANPUSH EQ}
with 
\eqref{LIMINFIBSUP EQ},
we now obtain that
\begin{equation}\label{CANPUSHQ EQ}
\bar{c}_{!} 
\left(
\Sigma_{\tau}[G \cdot C]/\Lambda
\right)
=
\Sigma_{\mathfrak{C}}[G \cdot_{\mathfrak{C}} \vect{C}]/\Lambda
\end{equation}
where we note that the quotient
$\Sigma_{\mathfrak{C}}[G \cdot_{\mathfrak{C}} \vect{C}]/\Lambda$
occurs in the fiber $\mathsf{Sym}^G_{\mathfrak{C}}$
while $\Sigma_{\tau}[G \cdot C]/\Lambda$
is not a fiber quotient.
In particular, the colors of the latter are
$\boldsymbol{E}(G \cdot C)/\Lambda$
rather than $\boldsymbol{E}(G \cdot C)$.

\eqref{CANPUSHQ EQ} now readily implies similar identifications for the generating sets
$\mathbb{F}^G_{\mathfrak{C}} \mathcal{I}_{\mathfrak{C},\mathcal{F}},
\mathbb{F}^G_{\mathfrak{C}} \mathcal{J}_{\mathfrak{C},\mathcal{F}}$.

Before describing the generating sets for 
$\mathsf{Op}^G_{\bullet,\mathcal{F}}(\V)$, however,
we need also address the path-lifting condition, 
requiring fibrations in $\mathsf{Op}^G_{\bullet,\F}(\V)$ to have the right lifting property against all maps 
$G/H \cdot (\eta \to \mathbb J)$
with $\mathbb{J}$ a $\V$-interval and $H \in \F_1$.
As the collection of all intervals form a class, one must be able to select a suitable representative set of intervals, leading to the following (cf. \cite{BM13}).

\begin{definition}\label{INTGENSET DEF}
	A set $\mathscr{G}$ of $\V$-intervals is \textit{generating} if,
	in the projective model category on $\Cat_{\set{0,1}}(\V)$,
	any $\V$-interval $\mathbb{J}$ is
	a retract of a trivial extension of some element
	$\mathbb{G} \in \mathscr{G}$.
	More explicitly, this means that there is a diagram in 
	$\Cat_{\set{0,1}}(\V)$ as below,
	where the left arrow is a trivial cofibration and
	$ri = id_{\mathbb{J}}$.
\begin{equation}\label{GTILGI EQ}
	\begin{tikzcd}
		\mathbb{G} \arrow[r,rightarrowtail, "\sim"]
	&
		\widetilde{\mathbb{G}} \arrow[r,yshift=-.3em, "r"']
	&
		\mathbb{J} \arrow[l,yshift=.3em, "i"']
	\end{tikzcd}
\end{equation}
\end{definition}

The following essentially recalls \cite[1.20]{CM13b}, \cite[\S 4.3]{Cav}.

\begin{remark}
      \label{SSETINT_REM}
When $\V$ is either $\mathsf{sSet}$ or $\mathsf{sSet}_{\**}$
one can take $\mathscr{G}$ to be a set of representatives of isomorphism classes of intervals with countably many cells.
Indeed, since in both cases the mapping spaces of a $\V$-interval
$\mathbb{J}$ are a simplicial set with (either one or two) contractible components,
a standard argument (see e.g. the argument between \cite[Lemmas 4.2,4.3]{Ber07b})
shows that $\mathbb{J}$ has a countable subcomplex $\mathbb{G}$ with contractile components and
for which the inclusion 
$\mathbb{G} \to \mathbb{J}$
is an equivalence in $\mathsf{Cat}_{\{0,1\}}(\V)$.
But then, forming the cofibration followed by trivial fibration factorization
$\mathbb{G} \rightarrowtail \widetilde{\mathbb{G}}
\overset{\sim}{\twoheadrightarrow} \mathbb{J}$
in $\mathsf{Cat}_{\{0,1\}}(\V)$,
one has that the first map is a trivial cofibration by $2$-out-of-$3$
and that the second has a section since $\mathbb{J}$ is cofibrant by assumption, yielding \eqref{GTILGI EQ}.

More generally, a more careful argument \cite[Lemma 1.12]{BM13}
shows that every combinatorial monoidal model category
has a generating set of intervals.
\end{remark}

\begin{proposition}\label{GENIN PROP}
If $\V$ has a generating set of intervals $\mathscr{G}$ then a local $\F$-fibration 
$F \colon \O \to \P$ in 
$\mathsf{Op}^G_{\mathfrak{C}}(\V)$
is $\F$-path lifting iff it has the right lifting property against the maps 
$\{G/H \cdot (\eta \to \mathbb{G})\}_{\mathbb{G}\in \mathscr{G},H \in \mathcal{F}_1}$.
\end{proposition}

\begin{proof}
        The ``only if'' direction is immediate.
        Conversely, given some chosen interval $\mathbb{J}$,
let $\mathbb{G}, \widetilde{\mathbb{G}}$
be as in \eqref{GTILGI EQ}.
A standard argument concerning retractions shows that,
to solve a lifting problem against $\eta \to \mathbb{J}$,
it suffices to solve the induced lifting problem against
$\eta \to \widetilde{\mathbb{G}}$.
But now given a lifting problem against 
$G/H \cdot \left(\eta \to \widetilde{\mathbb{G}}\right)$,
we consider the diagram below, where the solid lift exists by hypothesis on $F$.
\[
\begin{tikzcd}
	\eta \ar{d} \ar{rr}  
&&
	\O^H 	\ar{d}{F}
\\
	\mathbb{G} \ar[rightarrowtail]{r}{\sim} \ar{rru}
&
	\widetilde{\mathbb{G}} \ar{r} \ar[dashed]{ru}
&
	\P^H
\end{tikzcd}
\]
But then, since $\mathbb{G} \overset{\sim}{\rightarrowtail} \widetilde{\mathbb{G}}$
is a trivial cofibration in $\mathsf{Cat}_{\{0,1\}}(\V)$
and $F$ is a local fibration,
the desired dashed lift exists
by Proposition \ref{LOCALTCHAR PROP}.
\end{proof}

We can now finally identify the generating (trivial) cofibrations of
$\mathsf{Op}^G_{\bullet,\F}$.

In the following we write $C_n \in \Sigma$ for the $n$-corolla.

\begin{definition}\label{OPGENCOF DEF}
Suppose that $\V$ has a generating set of intervals $\mathscr{G}$.

Then the generating cofibrations in $\mathsf{Op}^G_{\F}$
are the maps
\begin{itemize}
\item[(C1)] $\emptyset \to G/H \cdot \eta$ for $H \in \F_1$,
\item[(C2)] $\mathbb{F} \left( \Sigma_{\tau}[G \cdot C_n]/\Lambda \cdot i\right)$
for $n \geq 0$, $\Lambda \in \F_n$ and $i \in \mathcal{I}$,
\end{itemize}
while the generating trivial cofibrations are the maps 
\begin{itemize}
\item[(TC1)] 
$G/H \cdot \left(\eta \to \mathbb{G}\right)$ for $H \in \F_1$ and $\mathbb{G} \in \mathscr{G}$,
\item[(TC2)] 
$\mathbb{F} \left( \Sigma_{\tau}[G \cdot C_n]/\Lambda \cdot j\right)$
for $n \geq 0$, $\Lambda \in \F_n$ and $j \in \mathcal{J}$.
\end{itemize}
\end{definition}


\begin{lemma}[{cf. \cite[1.19]{CM13b}}]\label{POINT_4_LEMMA}
	The maps in (TC1),(TC2) are in the saturation of (C1),(C2).
\end{lemma}

\begin{proof}
        Clearly (TC2) is in the saturation of (C2).
	For (TC1), one has factorizations
\begin{equation}
	\begin{tikzcd}
	G/H \cdot \eta \arrow[r, rightarrowtail]
&
	G/H \cdot (\eta \amalg \eta) \arrow[r, rightarrowtail]
&
	G/H \cdot \mathbb{G}
	\end{tikzcd}
\end{equation}
with the first map a pushout of a map in (C1) and 
the second map in the saturation of (C2),
as
$\mathbb F(\Sigma_\tau[G \cdot C_1]/H \cdot i) = G/H \cdot \mathbb F(\Sigma_\tau[C_1] \cdot i)$
and thus the saturation of $(C2)$ contains $G/H \cdot F$ for all cofibrations $F \in \Cat_{\set{0,1}}\V$.
\end{proof}

\subsection{Interval cofibrancy and trivial cofibrations}
\label{TRIVCOF_SEC}

In this section we establish Proposition \ref{J_CELL_PROP},
stating that maps built cellularly out of
(TC1) and (TC2) are $\F$-weak equivalences.
We first recall the following
technical result from \cite{BM13}.

\begin{theorem}
[Interval Cofibrancy Theorem {\cite[Thm. 1.15]{BM13}}]
\label{INTCOF THM}
Let $(\V,\otimes)$ be a cofibrantly generated monoidal model category that
satisfies the monoid axiom and has cofibrant unit.

If $\mathbb{J} \in \mathsf{Cat}_{\{0,1\}}(\V)$
is cofibrant then 
$\mathbb{J}(0,0)$ 
is a cofibrant monoid, i.e. cofibrant
in $ \mathsf{Cat}_{\{0\}}(\V)$.
\end{theorem}

Our assumptions on $\V$
in Theorem \ref{INTCOF THM} differ slightly
from those in the original formulation \cite[Thm. 1.15]{BM13},
as we replace the \emph{adequacy} condition in
\cite[Def. 1.1]{BM13} with the monoid axiom.
Nonetheless, the proof therein (occupying \S 3.6,\S 3.7,\S 3.8 in \cite{BM13}) still follows as written.
This is because adequacy is never used directly, 
serving only to guarantee existence of the 
model structures 
on $\mathsf{Cat}_{\{0,1\}}(\V),\mathsf{Cat}_{\{0\}}(\V)$
and 
on modules $\mathsf{Mod}_{R}(\V)$, $_R\mathsf{Mod}(\V)$
over a monoid $R$.
However, the monoid axiom suffices for these claims
\cite[Thm. 1.3]{Mur11},\cite[Thm. 4.1]{SS00}.

%

\begin{remark}
By symmetry, one also has that $\mathbb{J}(1,1)$ is cofibrant.
Moreover, the formulation in \cite[Thm. 1.15]{BM13}
includes additional cofibrancy conditions for
$\mathbb{J}(0,1),\mathbb{J}(1,0)$
as modules over $\mathbb{J}(0,0),\mathbb{J}(1,1)$.
These conditions are essential for their proof, 
but not needed for our application.
\end{remark}

We note that the Interval Cofibrancy Theorem is a particular case of the following conjecture when $\mathfrak{C} \to \mathfrak{D}$
is the inclusion $\{0\} \to \{0,1\}$.

\begin{conjecture}\label{CATOP CONJ}
Let $\varphi \colon \mathfrak{C} \to \mathfrak{D}$
be an injection of colors.
Then the pullback functors
\[
	\mathsf{Cat}_{\mathfrak{D},\F}^G(\V)
	\xrightarrow{\varphi^{\**}}
	\mathsf{Cat}_{\mathfrak{C},\F}^G(\V)
\qquad
	\mathsf{Op}_{\mathfrak{D},\F}^G(\V)
	\xrightarrow{\varphi^{\**}}
	\mathsf{Op}_{\mathfrak{C},\F}^G(\V)
\]
preserve cofibrations between cofibrant objects.
\end{conjecture}

\begin{remark}
To see why Conjecture \ref{CATOP CONJ} is at least plausible,
we argue that $\varphi^{\**}$ sends free objects to free objects, 
which is essentially tantamount to
sending generating cofibrations \eqref{VGSIGF EQ} to generating cofibrations.
To see this, consider the simplest example 
with
$\varphi \colon \{0\} \to \{0,1\}$
and a free $\mathbb{F}_{\{0,1\}}X $ in $\mathsf{Cat}_{\{0,1\}}(\V)$.
Then one can check that 
$\varphi^{\**} \left( \mathbb{F}_{\{0,1\}}X \right)$ in $\mathsf{Cat}_{\{0\}}(\V)$ is the free monoid
\begin{equation}\label{PULLFREEEX EQ}
	\varphi^{\**} \left( \mathbb{F}_{\{0,1\}}X \right)
\simeq
	\mathbb{F}_{\{0\}}
	\left(X(0,0) \amalg 
	\coprod_{n \geq 0}
	X(0,1)\otimes X(1,1)^{\otimes n} \otimes X(1,0) 
	\right)
\end{equation}
where we note that the expression inside
$\mathbb{F}_{\{0\}}$
in \eqref{PULLFREEEX EQ}
can be intuitively described as the formal composites
$0 \to 1 \to 1 \to \cdots \to 1 \to 0$
of ``arrows'' in $X$ that start and end at $0$ and where all intermediate objects are $1$.
More generally, for an inclusion of colors 
$\varphi \colon \mathfrak{C} \to \mathfrak{D}$
one has that 
$\varphi^{\**} \left(\mathbb{F}_{\mathfrak{D}} X\right)$
is similarly free on formal composites
$c_0 \to d_1 \to d_2 \to \cdots \to d_n \to c_{n+1}$
of arrows in $X$
where $c_i \in \mathfrak{C}$
and $d_j \in \mathfrak{D} \setminus \mathfrak{C}$,
while for operads the analogue claim involves labeled trees whose root and leaves are labeled by $\mathfrak{C}$
and whose inner edges are labeled by 
$\mathfrak{D} \setminus \mathfrak{C}$.
 
It is then straightforward to check that, under mild assumptions on $\V$,
$\varphi^{\**} \left(\mathbb{F}_{\mathfrak{D}} X \right)$
will be a (trivial) cofibration in 
$\mathsf{Cat}^G_{\mathfrak{C},\F}(\V)$
(resp. $\mathsf{Op}^G_{\mathfrak{C},\F}(\V)$))
when $\mathbb{F}_{\mathfrak{D}} X$
is a generating (trivial) cofibration
in $\mathsf{Cat}^G_{\mathfrak{D},\F}(\V)$
(resp. $\mathsf{Op}^G_{\mathfrak{D},\F}(\V)$)).
However, the argument just given \emph{does not} outline a proof of Conjecture \ref{CATOP CONJ},
due to $\varphi^{\**}$ not preserving pushouts, 
so that, to actually prove Conjecture \ref{CATOP CONJ},
one would need a careful analysis of the interaction of $\varphi^{\**}$ with pushouts of free categories/operads, 
as in the proof of \cite[Thm. 1.15]{BM13}.

Lastly, we make note of a very similar conjecture:
\begin{conjecture}
        Let $\mathfrak C$ be a $G$-set of colors.
        Then the restriction functor
        \[
                j^{\**} \colon \mathsf{Op}_{\mathfrak{C},\F}^G 
                \longto
                \mathsf{Cat}_{\mathfrak{C},\F}^G
        \]
        preserves cofibrations between cofibrant objects.
\end{conjecture}
Again,
one has that $j^{\**}$ 
sends generating (trivial) cofibrations to (trivial) cofibrations.
However, since our operads have $0$-ary operations, 
$j^{\**}$ does not preserve pushouts
(indeed, this would be tantamount to the claim that trees with a single leaf are linear trees, which is not true if we allow for trees with stumps).
\end{remark}

In the next result we write 
$\partial_i \colon \{0,1\} \to \{0,1,2\}$
for the ordered inclusion that omits $i$,
and  
$\widetilde{\mathbbm{2}} \in
\mathsf{Cat}_{\{0,1,2\}}(\V)$
for the ``double isomorphism category''
where all mapping objects are $\widetilde{\mathbbm{2}}(i,j)=1_{\V}$.

In the following note that, since the $\partial_i$
are injective, one has 
$\check{\partial}_{i,!} \simeq \partial_{i,!}$,
cf. Remark \ref{OP_MAP REM}.

\begin{lemma}[Interval Amalgamation Lemma {\cite[Lemma 1.16]{BM13}}]
\label{AMALGLEM LEM}

Let $\V$ be as in Theorem \ref{INTCOF THM} and
$\mathbb{I},\mathbb{J}$ be $\V$-intervals.

Then the coproduct 
$\mathbb{K} = \partial_{2,!} \mathbb{I} \amalg \partial_{0,!} \mathbb{J}$
in $\mathsf{Cat}_{\{0,1,2\}}(\V)$
is cofibrant and weakly equivalent to $\widetilde{\mathbbm{2}}$.

In particular, 
$\mathbb{I} \star \mathbb{J} = \partial_1^{\**}\mathbb{K}$
is weakly equivalent to $\widetilde{\mathbbm{1}}$
in $\mathsf{Cat}_{\{0,1\}}(\V)$.
\end{lemma}

\begin{remark}
$\mathbb{I} \star \mathbb{J} = 
\partial_1^{\**} \mathbb{K} =
\partial_1^{\**} \left(\partial_{2,!} \mathbb{I} \amalg \partial_{0,!} \mathbb{J}\right)$
is called the \emph{amalgamation} of $\mathbb{I}$ and $\mathbb{J}$,
so that Lemma \ref{AMALGLEM LEM} can be phrased as saying that an amalgamation of intervals is,
up to cofibrant replacement, again an interval
(Conjecture \ref{CATOP CONJ} would imply that $\mathbb{I} \star \mathbb{J}$ is already cofibrant, 
but we will not need to know this).
\end{remark}

The original proof of this result \cite[Lemma 1.16]{BM13} uses the cofibrancy of modules conditions on
$\mathbb{J}(0,1),\mathbb{J}(1,0)$
in \cite[Thm 1.15]{BM13}.
Here we present an alternative argument requiring only the cofibrancy of $\mathbb{J}(0,0)$ as a monoid,
as stated in our formulation of Theorem \ref{INTCOF THM}.

\begin{proof}
Since $\mathbb{I},\mathbb{J}$ are cofibrant and the $\partial_{i,!}$
preserve cofibrations by
(the category version of)
Corollary \ref{OPADJ_COR}(i), 
the coproduct 
$\partial_{2,!} \mathbb{I} \amalg \partial_{0,!} \mathbb{J}$
is a homotopy coproduct, 
so we are free to replace $\mathbb{I},\mathbb{J}$
with any chosen intervals.
In particular, we may thus assume there are (local) trivial fibrations
$\mathbb{I} \overset{\sim}{\twoheadrightarrow} \widetilde{\mathbbm{1}}$,
$\mathbb{J} \overset{\sim}{\twoheadrightarrow} \widetilde{\mathbbm{1}}$.
One then has a map
\[
\mathbb{K}=
\partial_{2,!} \mathbb{I} \amalg \partial_{0,!} \mathbb{J}
\to
\partial_{2,!} \widetilde{\mathbbm{1}} \amalg \partial_{0,!} \widetilde{\mathbbm{1}}
= \widetilde{\mathbbm{2}}
\]
which we will show to be a weak equivalence.

Firstly, by applying \cite[Cor. \ref{OC-FGTPUSH_COR}]{BP_FCOP} twice,
one has that $\mathbb{K}(1,1)= \mathbb{I}(1,1) \amalg \mathbb{J}(0,0)$,
where the coproduct is taken in $\Cat_{\{1\}}(\V)$.
Since Theorem \ref{INTCOF THM} says $\mathbb{I}(1,1),\mathbb{J}(0,0)$
are acyclic cofibrant 
(i.e. the map from the initial object 
$\eta \in \mathsf{Cat}_{\{0\}}(\V)$
is a trivial cofibration),
so is $\mathbb{K}(1,1)$.

Next, the trivial fibrations
$\mathbb{I} \overset{\sim}{\twoheadrightarrow} \widetilde{\mathbbm{1}}$,
$\mathbb{J} \overset{\sim}{\twoheadrightarrow} \widetilde{\mathbbm{1}}$
allow us to find factorizations
of the identity $1_{\V} \xrightarrow{=} 1_{\V}$
\[
1_{\V} \xrightarrow{\alpha}
\mathbb{I}(0,1)
\to 
\widetilde{\mathbbm{1}}(0,1)
\quad
1_{\V} \xrightarrow{\beta}
\mathbb{I}(1,0)
\to 
\widetilde{\mathbbm{1}}(1,0)
\quad
1_{\V} \xrightarrow{\bar{\alpha}}
\mathbb{J}(0,1)
\to 
\widetilde{\mathbbm{1}}(0,1)
\quad
1_{\V} \xrightarrow{\bar{\beta}}
\mathbb{I}(1,0)
\to 
\widetilde{\mathbbm{1}}(1,0).
\]
For any choice of $i,j$ in $\{0,1,2\}$,
by pre and postcomposing with $\alpha,\beta,\bar{\alpha},\bar{\beta}$
as appropriate, one gets a commutative diagram
\begin{equation}\label{INTAM EQ}
\begin{tikzcd}[column sep = 45pt]
\mathbb{K}(1,1)
\arrow{d}[swap]{\sim}
\arrow{r}
&
\mathbb{K}(i,j) 
\arrow[d]
\\
\widetilde{\mathbbm{2}}(1,1)
\arrow[equal]{r}
&
\widetilde{\mathbbm{2}}(i,j).
\end{tikzcd}
\end{equation}
We now note that  
$\alpha,\beta,\bar{\alpha},\bar{\beta}$
are homotopy equivalences
in the sense of \cite[Def. 2.6]{BM13}
(or Definition \ref{EQUIV_DEF} below),
since they represent the 
identity homotopy class of maps
$[id_{1_{\V}}] \in \Ho \V(1_{\V},\widetilde{\mathbbm {1}}(k,l))$,
which are isomorphisms between $0,1$ in the category
$\pi_0 \widetilde{\mathbbm {1}}$
of \cite[Rem. 2.7]{BM13}
(or Definition \ref{HTPY_DEFN} below). 
See also Remark \ref{NATISO REM}.
But now \cite[Lemma 2.12]{BM13}
(or its generalization Corollary \ref{ALBEETA COR})
implies that the top horizontal map in 
\eqref{INTAM EQ} is a weak equivalence,
and thus so are the maps
$\mathbb{K}(i,j) \to \widetilde{\mathbbm{2}}(i,j)$,
establishing that $\mathbb{K} \to \widetilde{\mathbbm{2}}$
is indeed a weak equivalence.
\end{proof}

\begin{lemma}[cf. {\cite[4.17]{Cav}}] \label{TRANSCOMP_ES_LEM}
	A transfinite composition of $\F$-essentially surjective maps in $\Op^G_\bullet(\V)$ is $\F$-essentially surjective.
\end{lemma}

\begin{proof}
Let $\kappa$ be a limit ordinal and consider a transfinite composition
$\O_0 \to \O_1 \to \cdots 
\to \colim_{\alpha < \kappa} \O_{\alpha} = \O_{\kappa}$
of $\F$-essentially surjective maps
$F_{\alpha} \colon \O_{\alpha} \to \O_{\alpha +1}$, where, as usual, 
we assume that $\O_{\alpha} = \colim_{\beta < \alpha} \O_{\beta}$ whenever $\alpha < \kappa$ is a limit ordinal.
We argue by transfinite induction on $\alpha \leq \kappa$
that the composite maps $\bar{F}_{\alpha} \colon \O_0 \to \O_{\alpha}$
are $\F$-essentially surjective. Fix $H \in \F_1$.

For a successor ordinal $\alpha+1$, given $c \in \O_{\alpha+1}^H$
one can find $b \in \O_\alpha^H$ and map from an interval 
$\mathbb{J} \xrightarrow{j} \O_{\alpha+1}^H$ with $j(0) = F_{\alpha}(b)$, $j(1)=c$
and, by the induction hypothesis, can likewise find
$a \in \O_0^H$ and map from an interval $\mathbb{I} \xrightarrow{i} \O_\alpha^H$
with $i(0)=\bar{F}_{\alpha}a$, $i(1)=b$.
The amalgamated map
$\mathbb{I} \star \mathbb{J} \xrightarrow{F_{\alpha}i \star j} \O_{\alpha+1}^H$
now connects $\bar{F}_{\alpha+1}(a)$ and $c$, as desired.

The case of a limit ordinal $\alpha$ is immediate: any object $b \in \O_{\alpha}^H$
lifts to an object $\bar{b} \in \O_{\bar{\alpha}}^H$ for some $\bar{\alpha} < \alpha$,
so noting that by induction there exists $a \in \O_0^H$ and map from an interval $\mathbb{I} \xrightarrow{i} \O_{\bar{\alpha}}^H$
with $i(0) = F_{\bar{\alpha}}a$, $i(1) = \bar{b}$ yields the result.
\end{proof}

\begin{remark}\label{LOCOTIMESTRI REM}
	Say a map $F \colon X \to Y$ in 
	$\mathsf{Sym}^{G}_{\bullet}(\V)$
	is a \emph{local genuine $\otimes$-trivial cofibration}
	if  
	$X \to F^{\**}Y$ is a 
	genuine $\otimes$-trivial cofibration in 
	$\mathsf{Sym}^{G}_{\mathfrak{C}_X}(\V) = \V^{G \ltimes \Sigma^{op}_{\mathfrak{C}_{X}}}$
	(cf. Definition \ref{GGENOTITC DEF}).
One then has that local genuine $\otimes$-trivial cofibrations are closed under transfinite composition.
Indeed, given such a transfinite composite as on the left
\[
X_0 \xrightarrow{F_1} 
X_1 \xrightarrow{F_2}
X_2 \xrightarrow{F_3}
X_3 \to \cdots 
\qquad \qquad 
X_0 \to  
F_1^{\**} X_1 \to
F_1^{\**}F_2^{\**}X_2 \to
F_1^{\**}F_2^{\**}F_3^{\**}X_3 \to \cdots 
\]
the induced transfinite composite in $\mathsf{Sym}^{G}_{\mathfrak{C}_{X_0}}(\V)$
on the right consists of 
genuine $\otimes$-trivial cofibrations,
since these are preserved under pullback
(Proposition \ref{GOTC_PROP}).
\end{remark}

\begin{proposition}[{c.f. \cite[4.20]{Cav}}]\label{J_CELL_PROP}
	Suppose $\V$ satisfies the conditions in Theorem \ref{THMIREST}.
      
	Then maps in the saturation of (TC1),(TC2) 
	are $\F$-weak equivalences in $\mathsf{Op}^G_{\bullet}(\V)$.
\end{proposition}

\begin{proof}
	We reduce to the case $\F=\F_{all}$,
	as that makes (TC1),(TC2) in Definition \ref{OPGENCOF DEF} as large as possible 
	and $\mathcal{F}$-weak equivalences as small as possible,
	cf. Remark \ref{FALLMAXMIN REM}.
		
	By Proposition \ref{GOTC_PROP}(i)
	and the closure under transfinite composition properties in Lemma \ref{TRANSCOMP_ES_LEM} and Remark \ref{LOCOTIMESTRI REM},
	it suffices to show that, for every pushout
\begin{equation}\label{JUSTAPUSH EQ}
\begin{tikzcd}
	J_1 \ar{r}{a} \arrow{d}[swap]{j}
&
	\O \arrow{d}
\\
	J_2 \ar{r} 
&
	\P
\end{tikzcd}
\end{equation}
	where $j$ is one of the generating trivial cofibrations 
	in (TC1),(TC2),
	one has that 
	$\O \to \mathcal{P}$ is both a local genuine $\otimes$-trivial cofibration and $\F_{all}$-essentially surjective.

	Firstly, if $j$ happens to be a map in (TC2),
	then this pushout
	can be alternatively calculated as the pushout below
	in the fixed color category $\mathsf{Op}^{G}_{\mathfrak C_{\O}}(\V)$.
\begin{equation}
\begin{tikzcd}
\check{a}_!J_1 \ar{r} \arrow{d}[swap]{\check{a}_!j}
&
\O \arrow{d}
\\
\check{a}_!J_2 \ar{r} 
&
\P
\end{tikzcd}
\end{equation}      
      And, since $\check{a}_!$ is left Quillen (cf. Corollary \ref{OPADJ_COR}\ref{OPCOCHADJ_LBL}), 
      this is the pushout of a trivial cofibration in the fiber model structure
      $\mathsf{Op}^{G}_{\mathfrak C_{\O},\F}(\V)$.
	The essential surjectivity claim is then obvious,
	while the $\otimes$-trivial cofibration
	claim follows from 
	Remark \ref{GOTC_REM} and Proposition \ref{GOTC_PROP}.

	Secondly, in the more interesting case of $j$ a map in (TC1), 
	i.e. of the form $G/H \cdot (\eta \to \mathbb{G})$ for $\mathbb{G}$ a generating $\V$-interval,
	we split the pushout \eqref{JUSTAPUSH EQ} as a composition of two pushouts 
\begin{equation}\label{BROKENPUSH EQ}
\begin{tikzcd}
	G/H \cdot \eta \arrow[r, "a"] \arrow[d, "G/H \cdot \phi"']
&
	\O \arrow[d,"\phi'"]
\\
	G/H \cdot \mathbb{G}_{\set{0}} \arrow[r] \arrow[d, "G/H \cdot \psi"']
&
	\O' \arrow[d,"\psi'"]
\\
	G/H \cdot \mathbb{G} \arrow[r]
&
	\P
\end{tikzcd}
\end{equation}
	where $\mathbb{G}_{\set{0}}$ is the full $\V$-subcategory of $\mathbb{G}$ spanned by the object $0$.
	It now suffices to show that the desired properties hold individually for $\phi'$ and $\psi'$.

	For the top pushout in \eqref{BROKENPUSH EQ}, Theorem \ref{INTCOF THM} implies that $\eta \to \mathbb{G}$
	is a trivial cofibration in $\mathsf{Cat}_{\{0\}}(\V)$
	so that, since Corollary \ref{OPADJ_COR}\ref{OPCOMBADJ_LBL} 
      says that $G/H \cdot (-) \colon \mathsf{Cat}_{\{0\}}(\V) \to \mathsf{Cat}^G_{G/H,\F_{all}}(\V)$ is left Quillen,
      we have that $G/H \cdot \phi$ is a $\F_{all}$-trivial cofibration with fixed objects,
      and thus the (TC2) argument above implies $\phi'$ is
      a local genuine $\otimes$-trivial cofibration and
      $\F_{all}$-essentially surjective.
      
	Now consider the bottom pushout in \eqref{BROKENPUSH EQ}.
	Since $\psi$ is a local isomorphism (i.e. $\mathbb{G}_{\{0\}} \to \psi^{\**}\mathbb{G}$ is an isomorphism),
	so is $G/H \cdot \psi$ and thus, by 
	\cite[Cor. \ref{OC-LOCALISO_COR}]{BP_FCOP}
	(see also \cite[Prop. B.22]{Cav}),
	the map $\psi' \colon \O' \to \mathcal{P}$
	is itself a local isomorphism, and thus certainly a local genuine $\otimes$-trivial cofibration.
	To address $\F_{all}$-essential surjectivity, 
	we write $[g]_0$ and $[g]_1$ for $[g] \in G/H$
	to denote the objects of $G/H \cdot \mathbb{G}$,
	so that $\C_{\P} = \C_{\O} \amalg \{[g]_1\}_{[g] \in G/H}$.
	Clearly one needs only verify the essential surjectivity condition for the $[g]_1$.
	Given $K \leq G$,
	$[g]_1$ is $K$-fixed in $\P$ iff it is $K$-fixed in $G/H \cdot \mathbb{G}$,
	in which case $[g]$ induces a map
	$\mathbb{G} \xrightarrow{[g]} \left(G/H \cdot \mathbb{G}\right)^K$.
	Writing $i$ for the composite 
	$\mathbb{G} \xrightarrow{[g]} \left(G/H \cdot \mathbb{G}\right)^K \to \P^K$,
	one then has $i(0) = a([g]_0)$ and $i(1)=[g]_1$,
	establishing $\F_{all}$-essential surjectivity.
\end{proof}

\begin{remark}\label{JCELLSM REM}
	If $\O \in \mathsf{sOp}^G_{\bullet}$
	in \eqref{JUSTAPUSH EQ}
	is underlying cofibrant in $\mathsf{Sym}^G_{\mathfrak{C}_{\O}}(\V)$
	then, by \cite[Rem. \ref{OC-THMISM REM}]{BP_FCOP},
	the map $\O \to \P$ in \eqref{JUSTAPUSH EQ}
	is actually a local genuine trivial cofibration 
	(rather than just a local genuine $\otimes$-trivial cofibration).
	Hence, the claim that 
	``a map \emph{with $\F$-cofibrant domain} in the saturation
	of (TC1),(TC2) is a $\F$-weak equivalence in
	$\mathsf{sOp}^G_{\bullet}$'' 
	does not require the global monoid axiom
	in \cite[Def. \ref{OC-GLOBMONAX_DEF}]{BP_FCOP}
	and in (iv) of Theorem \ref{THMA}.
\end{remark}

\subsection{Equivalences of objects}\label{EQUIVOBJ_SEC}

Our next task is to show that the
$\F$-weak equivalences
in Definition \ref{MODEL_DEFN}
satisfy $2$-out-of-$3$,
with the main difficulty coming 
from the fact that essential surjectivity is 
defined using $\V$-intervals.
To address this, this section relates the $\F$-weak equivalences
in Definition \ref{MODEL_DEFN} with the $\F$-Dwyer-Kan equivalences
in the statement of Theorem \ref{THMA},
for which $2$-out-of-$3$ is easier to establish
(though we note that this claim is more subtle in 
the equivariant setting;
see Proposition \ref{23HARDCASE PROP}).

\begin{definition}\label{HTPY_DEFN}
	Suppose that $(\V,\otimes)$ has a cofibrant unit.

	Given $\mathcal C \in \Cat_{\mathfrak{C}}(\V)$,
	we define $\pi_0 \mathcal C \in \Cat_{\mathfrak{C}}(\mathsf{Set})$ 
	to be the ordinary category with the same objects and
\[
	\pi_0(\mathcal{C})(c,d)=
	\Ho(\V)(1_\V, \mathcal C(c,c'))=
	[1_\V, \mathcal{C}_f(c,c')]
\]
where $[-,-]$ denotes homotopy equivalence classes of maps,
and $\mathcal{C}_f$ denotes some fibrant replacement of
$\mathcal C$ in $\Cat_{\mathfrak{C}}(\V)$.
The composition $[g]\circ [f]$
in $\pi_0(\mathcal{C})$
of classes $[f],[g]$
represented by
$1_{\mathcal{V}} \xrightarrow{f} \mathcal{C}_f({c,c'})$
and 
$1_{\mathcal{V}} \xrightarrow{g} \mathcal{C}_f({c',c''})$
is given by the class $[gf]$, where $gf$ denotes the composite
\begin{equation}\label{COMPI0 EQ}
1_{\mathcal{V}} \simeq
1_{\mathcal{V}} \otimes 1_{\mathcal{V}} \xrightarrow{g \otimes f}
\mathcal{C}_f({c',c''}) \otimes  \mathcal{C}_f({c,c'}) \xrightarrow{\circ}
\mathcal{C}_f({c,c''}).
\end{equation}
\end{definition}

The assumption that $1_{\mathcal{V}}$ is cofibrant
is needed to prove that \eqref{COMPI0 EQ}
respects equivalence classes. 

Moreover, since any two fibrant replacements are connected by a zigzag of weak equivalences,
(the isomorphism class of) $\pi_0 \mathcal{C}$ does not depend on the choice of fibrant replacement $\mathcal{C}_f$.

\begin{remark}
      The assignment
      $\pi_0\colon \mathsf{Cat}_{\mathfrak{C}}(\V)
\to \mathsf{Cat}_{\mathfrak{C}}(\mathsf{Set})$ is functorial,
i.e. a $\V$-functor
$\mathcal{C} \to \mathcal{D}$
induces a functor 
$\pi_0\mathcal{C} \to \pi_0\mathcal{D}$.
Moreover, $\pi_0$ sends weak equivalences to isomorphisms.
\end{remark}

\begin{remark}\label{NATISO REM}
The map $\widetilde{\mathbbm{1}} \to \widetilde{\mathbbm{1}}_f$
shows that the identity
$id_{1_{\V}}\colon 1_{\V} \xrightarrow{=} 1_{\V} = \widetilde{\mathbbm{1}}(0,1) = \widetilde{\mathbbm{1}}(1,0)$
induces two inverse arrows
$[id_{1_{\V}}] \in \pi_0 \widetilde{\mathbbm{1}}(0,1)$
and
$[id_{1_{\V}}] \in \pi_0 \widetilde{\mathbbm{1}}(1,0)$.

We refer to these arrows as the
\emph{natural isomorphisms} between $0$ and $1$ in $\pi_0 \widetilde{\mathbbm{1}}$.
\end{remark}

Following \cite[Def. 2.6]{BM13} (also \cite{Cav}),
we make the following definitions.

\begin{definition}\label{EQUIV_DEF}
	Given $\mathcal{C}$ in  $\Cat(\V)$ and $c,c'\in \mathcal C$, we say $c$ and $c'$ are
\begin{itemize}
	\item {\em equivalent} if there exists a $\V$-interval $\mathbb{J}$
	and map $i: \mathbb{J} \to \mathcal C$ such that
	$i(0)= c$, $i(1)= c'$;
	\item {\em virtually equivalent} if $c,c'$ are equivalent in some fibrant replacement
	$\mathcal C_f$ of $\mathcal C$ in $\Cat_{\mathfrak{C}_{\mathcal{C}}}(\V)$;
	\item {\em homotopy equivalent} if $c,c'$ are isomorphic in the unenriched category $\pi_0 \mathcal C$.

	Explicitly, this means there are maps 
	$1_\V \xrightarrow{\alpha} \mathcal C_f(c,c')$, 
	$1_\V \xrightarrow{\beta} \mathcal C_f(c',c)$ such that
	$1_{\V} \xrightarrow{\beta \alpha} \mathcal C_f(c,c)$,
	$1_{\V} \xrightarrow{\alpha \beta} \mathcal C_f(c',c')$
	are homotopic to the identities
	$1_{\V} \xrightarrow{id_c} \mathcal C_f(c,c)$,
	$1_{\V} \xrightarrow{id_{c'}} \mathcal C_f(c',c')$.
\end{itemize}
\end{definition}

\begin{remark}\label{VIRTEQRESTA REM}
Given $c,c' \in \mathcal{C}$, write $\iota_{c,c'} \ \colon \{0,1\} \to \mathfrak{C}_{\mathcal{C}}$
for the induced map
(for $\mathfrak{C}_{\mathcal{C}}$ the object set of 
$\mathcal{C}$).
The condition that $c,c'$ are equivalent can then be restated as saying that
there is a $\V$-interval $\mathbb{J}$
together with some map $\mathbb{J} \to \iota^{\**}_{c,c'}\mathcal{C}$ in $\mathsf{Cat}_{\{0,1\}}(\V)$. 

We can similarly restate the notion of virtual equivalence.
Since a $\V$-interval $\mathbb{J}$ is a cofibrant replacement of $\widetilde{\mathbbm{1}}$ in $\mathsf{Cat}_{\{0,1\}}(\V)$
while $\iota^{\**}_{c,c'} \mathcal{C}_f$ is a fibrant replacement of $\iota^{\**}_{c,c'} \mathcal{C}$,
the condition that $c,c'$ are virtually equivalent is precisely the statement that
\[
\Ho \left(\mathsf{Cat}_{\{0,1\}}(\V)\right)\left(\widetilde{\mathbbm{1}},\iota^{\**}_{c,c'} \mathcal{C}\right)
=
[\mathbb{J},\iota^{\**}_{c,c'} \mathcal{C}_f] 
\neq
 \emptyset
\]
i.e. that, up to homotopy, there is at least one map from $\widetilde{\mathbbm{1}}$ to $\iota^{\**}_{c,c'} \mathcal{C}$
in $\mathsf{Cat}_{\{0,1\}}(\V)$.

Note in particular that this does not depend on the choice of 
replacements 
$\mathbb{J}$ and $\mathcal{C}_f$.
\end{remark}

\begin{remark}
That homotopy equivalence of objects is an equivalence relation 
follows from the fact that composites of isomorphisms are isomorphisms.

On the other hand, when checking that equivalence and virtual equivalence are likewise equivalence relations,
the transitive property is not elementary, being instead 
a straightforward consequence of Interval Amalgamation, cf. Lemma \ref{AMALGLEM LEM}.
\end{remark}

\begin{remark}\label{EQUIVNEST_REM}
The notions in Definition \ref{EQUIV_DEF}
are nested: the map $\mathcal{C} \to \mathcal{C}_f$
yields that equivalence implies virtual equivalence;
a map $\mathbb{J} \to \mathcal{C}_f$
with $\mathbb{J}$ a $\mathcal{V}$-interval
induces a map
$\pi_0 \widetilde{\mathbbm{1}} \simeq \pi_0 \mathbb{J} \to \pi_0 \mathcal{C}$,
so that (since $0,1 \in \pi_0 \widetilde{\mathbbm{1}}$ are isomorphic)
virtual equivalence implies homotopy equivalence.

Moreover, \cite{BM13,Cav} show that, under suitable assumptions on $\V$, the converse implications also hold, as summarized below.
We discuss these converse results in what follows.
\begin{equation}\label{BASCOMP EQ}
\begin{tikzcd}[column sep = large]
            \mbox{equivalent}
            \arrow[r, Rightarrow, shift left=2, "\mathrm{always}"]
            &
            \mbox{ virtually equivalent}
            \arrow[r, Rightarrow, shift left=2, "\mathrm{always}"]
            \arrow[l, Rightarrow, shift left = 2, "\substack{$\V\phantom{ }$\mathrm{right} \\ \mathrm{proper}}"]
            &
            \mbox{ homotopy equivalent}
            \arrow[l, Rightarrow, shift left = 2, "\substack{\mathrm{coherence} \\ \mathrm{condition}}"]
\end{tikzcd}
\end{equation}
\end{remark}

\begin{proposition}[{cf. \cite[4.12]{Cav}, \cite[2.10]{BM13}}] \label{RIGHTPROPER PROP}
If $\V$ is right proper, then two colors in a $\V$-category $\mathcal{C}$ are virtually equivalent iff they are equivalent. 
\end{proposition}

\begin{proof}
Let $c,c'\in \mathcal{C}$ be two colors and (following Remark \ref{VIRTEQRESTA REM})
let $\mathbb{J} \overset{\sim}{\twoheadrightarrow} \iota_{c,c'} \mathcal{C}_F$
exhibit a virtual equivalence between them, where we note
that we can assume the map is a fibration in $\mathsf{Cat}_{\{0,1\}}(\mathcal{C})$
by using the factorization of any map as a ``trivial cofibration followed by a fibration''.
Forming the pullback
\begin{equation}\label{RIGHTPROPER EQ}
\begin{tikzcd}
		 \mathbb{J}' \ar{r} \arrow{d}
	&
		\iota_{c,c'} \mathcal{C} \arrow{d}{\sim}
\\
		\mathbb{J} \ar[twoheadrightarrow]{r} 
	&
		\iota_{c,c'}\mathcal{C}_f
\end{tikzcd}
\end{equation}
right properness of $\V$ implies 
that $\mathbb{J}' \to \mathbb{J}$ is a weak equivalence in 
$\mathsf{Cat}_{\{0,1\}}(\V)$,
and thus choosing a cofibrant replacement
$\mathbb{J}'_c \to \mathbb{J}'$
yields that $c,c'$ are equivalent.
\end{proof}

We now discuss the requirement for homotopy equivalence and virtual equivalence to coincide.
Informally, this will be the case provided any isomorphism 
$c \to c'$ in $\pi_0 \mathcal{C}$
can be suitably lifted to a map 
$\mathbb{J} \to \mathcal C_f$ for some $\V$-interval $\mathbb{J}$.
The following makes this idea precise, cf. \cite[\S 2]{BM13}.

\begin{definition}\label{COH DEF}
	Let $\mathbbm{1}$ be the free $\V$-arrow category and $\mathbb{J}$ a $\V$-interval.
	A cofibration $\mathbbm{1} \to \mathbb{J}$ in $\Cat_{\set{0,1}}(\V)$
	is called \textit{natural} if it fits into a commutative diagram in $\Cat_{\set{0,1}}(\V)$ as on the left below.
\begin{equation}
\begin{tikzcd}
	\mathbbm{1} \arrow[d, tail] \arrow[r]
&
	\widetilde{\mathbbm{1}} \arrow[d, "\sim"]
&&
	\mathbbm{1} \arrow[r, "\alpha"] \arrow[d, tail, dashed]
&
	\mathcal C_f
\\
	\mathbb{J} \arrow[r, "\sim"']
&
	\widetilde{\mathbbm{1}}_f
&&
	\mathbb{J} \arrow[ur, dashed]
\end{tikzcd}
\end{equation}
A homotopy equivalence between two objects in a $\V$-category $\mathcal C$ is called \textit{coherent} if
there is a representative $\alpha: \mathbbm{1} \to \mathcal C_f$ that factors along a natural cofibration,
as on the right above.

Lastly, the monoidal model category $\V$ is said to satisfy the \textit{coherence axiom} if
all homotopy equivalences in every $\V$-category are coherent.
\end{definition}

\begin{remark}
      \label{COH_EX_REM}
      The coherence axiom originates with the work of Boardman-Vogt, who showed that it holds for compactly-generated weak Hausdorff spaces $(\Top, \times)$ \cite[Lem. 4.16]{BV73}.
      
      The coherence axiom is also a consequence of Lurie's \textit{invertibility hypothesis} \cite[A.3.2.12]{Lur09}, a stronger hypothesis, by an argument of Berger-Moerdijk \cite[Rem. 2.19]{BM13}.
\end{remark}

\begin{remark}
Ignoring the more technical requirements,
Definition \ref{COH DEF} loosely says that  
$\mathbbm{1} \rightarrowtail \mathbb{J}$
is natural if the map $1_{\V} = \mathbbm{1}(0,1) \to \mathbb{J}(0,1)$
represents the natural isomorphism
$[id_{1_{\V}}]$ from $0$ to $1$ in $ \pi_0 \widetilde{\mathbbm{1}} \simeq \pi_0 \mathbb{J}$ (cf. Remark \ref{NATISO REM}),
and that the homotopy equivalence $\alpha$ is coherent
if there exists a map $\mathbb{J} \to \mathcal C_f$
such that
$\pi_0 \widetilde{\mathbbm{1}} \simeq
\pi_0 \mathbb{J} \to \pi_0 \mathcal C$
sends the natural isomorphism $[id_{1_{\V}}]$ to $[\alpha]$.

If, in addition, $\V$ is also right proper, we can slightly strengthen this observation, as follows.
\end{remark}

\begin{proposition}\label{ALTCOH PROP}
Suppose $\V$ is right proper and satisfies the coherence axiom.

Then, for any $\V$-category $\mathcal{C}$ and isomorphism  
$[\alpha]$ in $\pi_0 \mathcal{C}$,
there exists a map from an interval 
$\mathbb{J} \to \mathcal{C}$
such that 
$\pi_0 \mathbb{J} \to \pi_0\mathcal{C}$
sends the natural isomorphism $[id_{1_{\V}}]$ to $[\alpha]$.
\end{proposition}

\begin{proof}
By coherence, we can find a factorization of 
$\alpha$ as 
$\1 \rightarrowtail \mathbb J \twoheadrightarrow \mathcal C_f$
where, just as in the proof of
Proposition \ref{RIGHTPROPER PROP},
we are free to assume the second map is a fibration.
The result then follows by forming the pullback \eqref{RIGHTPROPER EQ} and arguing as in the proof of Proposition \ref{RIGHTPROPER PROP}.
\end{proof}

Berger and Moerdijk then prove the following,
which depends on a careful technical analysis of the $W$-construction applied to free isomorphism $\V$-category
$\widetilde{\mathbbm{1}}$.

\begin{proposition}[{cf. \cite[Prop. 2.24]{BM13}}]
\label{COHAX PROP}
If a cofibrantly generated monoidal model category $\V$
satisfies the monoid axiom and has cofibrant unit,
then $\V$ satisfies the coherence axiom.
\end{proposition}

Just as in Theorem \ref{INTCOF THM},
we have replaced the adequacy assumption \cite[Def. 1.1]{BM13}
with the monoid axiom, which suffices for the existence of the relevant model structures. 
Yet again the proof in loc. cit. makes no direct use of adequacy, 
though special mention should be made of \cite[Lemma 2.23]{BM13}, 
which builds the \emph{interval} \cite[Def. 4.1]{BM06} in $\V$
required to define the $W$-construction.

The pointed category $\V_{\**}=\V_{1_{\V}//1_{\V}}$
of factorizations
$1_{\V} \to X \to 1_{\V}$
of the identity
$1_{\V} \xrightarrow{=} 1_{\V}$
has a monoidal structure
$X \wedge Y  = \mathop{\mathrm{coeq}} \left(1_{\V} \leftarrow X \amalg_{1_{\V}} Y \to X \otimes Y\right)$ with unit $1_{\V} \amalg 1_{\V}$,
which is readily seen to define a cofibrantly generated monoidal model category with cofibrant unit whenever $(\V,\otimes)$ is one.
A \emph{segment} \cite[Def. 4.1]{BM06} is then a monoid
in $(\V_{\**},\wedge)$,
while an interval is a segment $H$ for which the two natural maps
$1_{\V} \amalg 1_{\V} \rightarrowtail H \xrightarrow{\sim} 1_{\V}$
are a cofibration and weak equivalence in $\V$.
Our hypothesis are not quite strong enough 
to guarantee that the category of segments $\mathsf{Seg}_{\V}$
has a full model structure 
(the monoid axiom for $(\V,\otimes)$ does not imply the monoid axiom for $(\V_{\**},\wedge)$) but,
by Remarks \ref{SEMI_REM}, \ref{GTRIV REM}
we nonetheless have a semi-model structure. This is enough to build an interval as a cofibrant replacement
$H \xrightarrow{\sim} 1_{\V}$
of the terminal segment $1_{V}$, with the fact that the forgetful functor
$\mathsf{Seg}_{\V} \to \V_{\**}$
preserves cofibrations between cofibrant objects 
(by the semi-model structure version of \cite[Thm. \ref{OC-THMII}]{BP_FCOP};
see Remarks \ref{SEMI_REM} and \ref{GTRIV REM})
implying that
$1_{\V} \amalg 1_{\V} \rightarrowtail H$
is indeed a cofibration in $\V$.

\vskip 10pt

Replacing the notion of equivalence of objects
in Definition \ref{MODEL_DEFN}
with that of homotopy equivalence, 
one obtains the notion of Dwyer-Kan equivalence
in the formulation of Theorem \ref{THMA}.

\begin{definition}\label{DKEQUIV_DEF}
Let $\F$ be a $(G,\Sigma)$-family with enough units.
We say a map $\O \to \P$ in $\Op^G_\bullet(\V)$ is:
\begin{itemize}
\item \textit{$\F$-$\pi_0$-essentially surjective} if
	$j^{\**}\pi_0 \O^H \to j^{\**}\pi_0 \P^H$
	is essentially surjective for $H \in \F_1$;
\item a \textit{$\F$-Dwyer-Kan equivalence} if
	it is a local $\F$-weak equivalence and $\F$-$\pi_0$-essentially
	surjective.
\end{itemize}
\end{definition}

\begin{remark}\label{CATEQUIV REM}
The requirement that 
a $\F$-Dwyer-Kan equivalence $\O \to \P$
be a local $\F$-weak equivalence
implies that the maps of categories 
$j^{\**}\pi_0 \O^H \to j^{\**}\pi_0 \P^H$
must be local isomorphisms, 
and thus equivalences in the category $\mathsf{Cat}$
of (unenriched) categories in the usual sense.
\end{remark}

\begin{corollary}\label{WEDKEQ COR}
$\F$-weak equivalences in $\mathsf{Op}^G_\bullet(\V)$
are $\F$-Dwyer-Kan equivalences.
	
Further, the converse holds provided that
$\V$ is a cofibrantly generated monoidal model category
that satisfies the monoid axiom, is right proper,
and has a cofibrant unit.	
\end{corollary}

\begin{proof}
	This follows by \eqref{BASCOMP EQ} for the first claim
	and by Propositions \ref{RIGHTPROPER PROP} and \ref{COHAX PROP}
	for the converse.
\end{proof}

\begin{proposition}\label{2OUTOF3 PROP}
Suppose that $\V$ satisfies conditions
(i) through (v) and (vii) in Theorem \ref{THMA} and 
let $\F$ be a $(G,\Sigma)$-family that has enough units. 
Then:
\begin{enumerate}[label=(\roman*)]
\item $\F$-weak equivalences between fibrant objects in $\mathsf{Op}^G_\bullet(\V)$
satisfy $2$-out-of-$3$;
\item if $\V$ is right proper, $\F$-weak equivalences in $\mathsf{Op}^G_\bullet(\V)$
satisfy $2$-out-of-$3$;
\item $\F$-Dwyer-Kan equivalences in $\mathsf{Op}^G_\bullet(\V)$
satisfy $2$-out-of-$3$.
\end{enumerate}
\end{proposition}

\begin{proof}
Consider the diagram in $\mathsf{Op}^G_\bullet(\V)$ on the left below,
and the induced maps on the right for each
$\mathfrak{C}_{\O}$-profile $\vect{C}$ and 
$\Lambda \in \F_{\vect{C}}$.
\begin{equation}\label{23EASYDIAG EQ}
	\begin{tikzcd}[row sep=5pt]
		\O \arrow{rr}{\bar{F}F}
		\arrow{dr}[swap]{F}
	&&
		\mathcal{Q} 
&&%
		\O(\vect{C})^{\Lambda} \arrow{rr}
		\arrow{dr}
	&&
		\mathcal{Q} (\bar{F}F(\vect{C}))^{\Lambda}
\\
	&
		\mathcal{P} \ar{ru}[swap]{\bar{F}}
	&
&&%
	&
		\mathcal{P}(F(\vect{C}))^{\Lambda} \ar{ru}
	&
	\end{tikzcd}
\end{equation}
We first address (i) and (iii) in parallel
(where for (i) we assume $\O,\P,\mathcal{Q}$ are fibrant).

Suppose first that $F,\bar{F}$ are $\F$-weak equivalences/$\F$-Dwyer Kan equivalences.
Then $2$-out-of-$3$ applied to the right diagram in 
\eqref{23EASYDIAG EQ}
implies that $\bar{F}F$ is a 
local $\F$-weak equivalence.
Moreover, in the $\F$-weak equivalence case
Lemma \ref{TRANSCOMP_ES_LEM}
implies $\bar{F}F$ is $\F$-essentially surjective,
while in the $\F$-Dwyer-Kan equivalence case
Remark \ref{CATEQUIV REM}
and $2$-out-of-$3$ for (unenriched categories)
implies 
$\bar{F}F$ is $\F$-$\pi_0$-essentially surjective.

Suppose next that $\bar{F},\bar{F}F$ are $\F$-weak equivalences/$\F$-Dwyer Kan equivalences.
It is again immediate that $F$ is a 
local $\F$-weak equivalence and that, 
in the $\F$-Dwyer Kan equivalence case,
$F$ is $\F$-$\pi_0$-essentially surjective.
It remains to establish the
$\F$-essential surjectivity of $F$ in the $\F$-weak equivalence case.
Given $b \in \mathcal{P}^H$,
the $\F$-essential surjectivity of $\bar{F}F$
yields $a \in \O^H$ and a map
$\mathbb{J} 
	\to 
\iota^{\**}_{\bar{F}Fa,\bar{F}b} 
\left( j^{\**} \mathcal{Q}^H \right)
	=
\iota^{\**}_{Fa,b} 
\left( j^{\**} \bar{F}^{\**}\mathcal{Q}^H \right)$
in
$\mathsf{Cat}_{\{0,1\}}(\V)$
(see Remark \ref{VIRTEQRESTA REM}).
But since 
$\P \to \bar{F}^{\**} \mathcal{Q}$
is a weak equivalence of fibrant objects in 
$\mathsf{Op}^G_{\mathfrak{C}_{\P}}(\V)$,
the map
$
\iota^{\**}_{Fa,b} 
\left( j^{\**} \mathcal{P}^H \right)
	\to
\iota^{\**}_{Fa,b} 
\left( j^{\**} \bar{F}^{\**}\mathcal{Q}^H \right)$
is likewise a weak equivalence of
fibrant objects in $\mathsf{Cat}_{\{0,1\}}(\V)$,
and one thus also has a map
$\mathbb{J} 
	\to 
\iota^{\**}_{Fa,b} 
\left( j^{\**} \mathcal{P}^H \right)
$,
showing that $F$ is $\F$-essential surjective.

Consider now the remaining case where
$F,\bar{F}F$ are
$\F$-weak equivalences/$\F$-Dwyer Kan equivalences.
The $\F$-essential surjectivity/$\F$-$\pi_0$-essential surjectivity
of $\bar{F}F$
states that, for any $c \in \mathcal{Q}^H$,
there exists $a \in \mathcal{O}^H$ and a map
$\mathbb{J} \to \iota^{\**}_{\bar{F}Fa,c} j^{\**}\mathcal{Q}^H$
/isomorphism between $\bar{F}Fa$ and 
$c$ in $j^{\**}\pi_0\mathcal{Q}^H$.
Hence, setting $b = Fa \in \mathcal{P}^H$,
we see that $\bar{F}$ is also
$\F$-essential surjective/$\F$-$\pi_0$-essential surjective.
It remains to show that
$F$ is a local $\F$-weak equivalence.
In contrast with the previous cases,
applying $2$-out-of-$3$ to the right diagram in \eqref{23EASYDIAG EQ}
only yields equivalences 
$\P(\vect{D})^{\Lambda} \to \mathcal{Q}(\bar{F}(\vect{D}))^{\Lambda}$
for $\mathfrak{C}_{\P}$-profiles of the form
$\vect{D} = \bar{F} (\vect{C})$ for some $\mathfrak{C}_{\O}$-profile $\vect{C}$ 
and $\Lambda \in \F_{\vect{C}}$.
To show that we have equivalences for all
$\mathfrak{C}_{\P}$-profiles $\vect{D}$
and all $\Lambda \in \F_{\vect{D}}$
(note that, even if $\vect{D} = F(\vect{C})$,
one can only guarantee $\F_{\vect{C}} \subseteq \F_{\vect{D}}$,
rather than $\F_{\vect{C}} = \F_{\vect{D}}$),
one needs a careful modification of the analogous non-equivariant argument \cite[Lemma 4.14]{Cav}.
As such, we postpone this claim to 
Proposition \ref{23HARDCASE PROP} below
(note that, by Corollary \ref{WEDKEQ COR}, 
only the $\F$-Dwyer Kan equivalence case needs be considered),
and dedicate the entirety of \S \ref{HMTYEQ SEC}
to proving that result.
We note that this is the case that requires that $\F$ has enough units (Definition \ref{FAMRESUNI DEF}).

Lastly, we address (ii).
Given a map of operads $F\colon \O \to \P$ and a color fixed fibrant replacement
$F_f \colon \O_f \to \P_f$,
it is immediate that 
$F$ is a $\F$-local weak equivalence iff $F_f$ is.
Moreover, if $\V$ is right proper, Proposition \ref{RIGHTPROPER PROP}
implies that $F$ is $\F$-essentially surjective iff $F_f$ is.
In other words, if $\V$ is right proper,
$F$ is a $\F$-weak equivalence iff $F_f$ is.
But (ii) is now immediate from (i).
\end{proof}

\subsection{Homotopy equivalences and fully faithfulness}
\label{HMTYEQ SEC}

This section is dedicated to proving the following, 
which is the remaining claim in the proof of Proposition \ref{2OUTOF3 PROP},
and is the claim requiring
the ``enough units'' condition in Definition \ref{FAMRESUNI DEF}.

\begin{proposition}\label{23HARDCASE PROP}
Suppose $\V$ satisfies the conditions 
(i) through (v) and (vii) in Theorem \ref{THMA},
and let $\F$ be a $(G,\Sigma)$-family that has enough units. 
Consider the diagram below 
in $\mathsf{Op}^G_\bullet(\V)$.
\begin{equation}\label{23HARDDIAG EQ}
	\begin{tikzcd}[row sep=5pt]
		\O \arrow{rr}{\bar{F}F}
		\arrow{dr}[swap]{F}
	&&
		\mathcal{Q} 
	\\
	&
		\mathcal{P} \ar{ru}[swap]{\bar{F}}
	\end{tikzcd}
\end{equation}
If $F$ and $\bar{F}F$ are $\F$-Dwyer-Kan equivalences
then $\bar{F}$ is a local $\F$-weak equivalence.
\end{proposition}

The proof of this result will adapt the proof of the non-equivariant analogue \cite[Lemma 4.14]{Cav},
but one must be more careful since equivariance introduces a number of subtleties.
For the sake of motivation, 
we first discuss a concrete example.

\begin{example}
Let $\V=\mathsf{sSet}$, 
$G = \{1,i,-1,-i\} \simeq \mathbb{Z}_{/4}$ 
be the group of quartic roots of unit,
and $\mathfrak{C} = \{\mathfrak{a}, \mathfrak{b}, i \mathfrak{b}, 
\mathfrak{c} \}$,
where we implicitly have
$i\mathfrak{a} = \mathfrak{a}$,
$-\mathfrak{b} = \mathfrak{b}$,
$i\mathfrak{c} = \mathfrak{c}$.
Consider the $\mathfrak{C}$-corollas below.
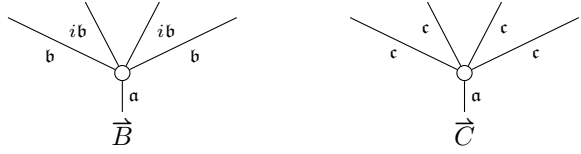
\begin{equation}
	\begin{tikzpicture}[auto,grow=up, level distance = 2.2em,
	every node/.style={font=\scriptsize,inner sep = 2pt}]%
		\tikzstyle{level 2}=[sibling distance=3em]%
			\node at (0,0) [font = \normalsize] {$\vect{B}$}%
				child{node [dummy] {}%
					child{node {}%
					edge from parent node [swap] {$\mathfrak{b}$}}%
					child[level distance = 2.9em]{node {}%
					edge from parent node [swap,	near end] {$i\mathfrak{b}$}}%
					child[level distance = 2.9em]{node {}%
					edge from parent node [near end] {$i\mathfrak{b}$}}%
					child{node {}%
					edge from parent node  {$\mathfrak{b}$}}%
				edge from parent node [swap] {$\mathfrak{a}$}};%
			\node at (4.5,0) [font = \normalsize] {$\vect{C}$}%
				child{node [dummy] {}%
					child{node {}%
					edge from parent node [swap] {$\mathfrak{c}$}}%
					child[level distance = 2.9em]{node {}%
					edge from parent node [swap,	near end] {$\mathfrak{c}$}}%
					child[level distance = 2.9em]{node {}%
					edge from parent node [near end] {$\mathfrak{c}$}}%
					child{node {}%
					edge from parent node  {$\mathfrak{c}$}}%
				edge from parent node [swap] {$\mathfrak{a}$}};%
	\end{tikzpicture}%
\end{equation}%
Let
$\O \xrightarrow{F} \P \xrightarrow{\bar{F}} \mathcal{Q}$
be as in \eqref{23HARDDIAG EQ},
and suppose 
$\mathfrak{C}_{\O} = \{\mathfrak{a},\mathfrak{c}\}$,
$\mathfrak{C}_{\P} = \mathfrak{C} = \{\mathfrak{a},\mathfrak{b},i\mathfrak{b},\mathfrak{c}\}$.
Then, if $F$ and $\bar{F}F$ are local Kan equivalences,
it is clear that
$\P(\vect{C}) \to \mathcal{Q}(\bar{F}(\vect{C}))$
is a Kan equivalence, since 
$\vect{C}$ is in the image of $F$.
However, to ensure that
$\P(\vect{B}) \to \mathcal{Q}(\bar{F}(\vect{B}))$
is also a Kan equivalence, 
we need an essential surjectivity condition on $F$.
For concreteness,
suppose there was a homotopy equivalence
$\alpha \colon \mathfrak{b} \to \mathfrak{c}$
in $\P$
(cf. Definition \ref{EQUIV_DEF}; note that $\alpha \in \P(\mathfrak{b};\mathfrak{c})$).
Then, by $G$-equivariance, one also has a homotopy equivalence
$i\alpha \colon i\mathfrak{b} \to \mathfrak{c}$ in $\P$,
so that
(cf. Corollary \ref{ALBEETA COR}; 
see also \cite[Lemma 2.1]{BM13},
\cite[Lemma 4.14]{Cav})
precomposing  
with $\alpha$, $i \alpha$ gives a string of Kan equivalences
\begin{equation}\label{ITERWE EQ}
\P(\vect{C})
=
\P(\mathfrak{c},\mathfrak{c},\mathfrak{c},\mathfrak{c};\mathfrak{a})
\sim 
\P(\mathfrak{b},\mathfrak{c},\mathfrak{c},\mathfrak{c};\mathfrak{a})
\sim
\P(\mathfrak{b},i\mathfrak{b},\mathfrak{c},\mathfrak{c};\mathfrak{a})
\sim
\P(\mathfrak{b},i\mathfrak{b},i\mathfrak{b},\mathfrak{c};\mathfrak{a})
\sim
\P(\mathfrak{b},i\mathfrak{b},i\mathfrak{b},\mathfrak{b};\mathfrak{a})
=
\P(\vect{B})
\end{equation}
and similarly
$\mathcal{Q}(\bar{F}(\vect{C})) 
\sim
\cdots 
\sim
\mathcal{Q}(\bar{F}(\vect{B}))$,
so that
$\P(\vect{B}) \to \mathcal{Q}(\bar{F}(\vect{B}))$
is indeed a Kan equivalence.

Our discussion thus far has ignored a key feature of the equivariant setting: the choice of $(G,\Sigma)$-family $\F$.
Given a general such $\F$, 
and assuming $F$ and $\bar{F} F$ are local $\F$-Kan equivalences,
it is again clear that
$\P(\vect{C})^{\Lambda} \to \mathcal{Q}(\bar{F}(\vect{C}))^{\Lambda}$
is a Kan equivalence for all $\Lambda \in \F_{\vect{C}}$.
And, yet again, to conclude that
$\P(\vect{B})^{\Lambda} \to \mathcal{Q}(\bar{F}(\vect{B}))^{\Lambda}$
is also a Kan equivalence for $\Lambda \in \F_{\vect{B}}$ we further need an essential surjectivity requirement on $F$.
As it turns out, this requirement on $F$ depends on 
$\Lambda \leq G \times \Sigma_4^{op}$ itself so that, for concreteness, we set (writing $(g,\sigma) \in G \times \Sigma_4^{op}$ simply as $g\sigma$)
\[
\Lambda = \langle (14)(23), i (12)(34) \rangle.
\]
and assume that $\Lambda \in \F_4$.
Note that $\Lambda$ stabilizes both 
$\vect{B}$ and $\vect{C}$
(cf. Definition \ref{STABS DEF}),
so that 
$\Lambda \leq \mathsf{Aut}_{G \ltimes \Sigma^{op}_{\mathfrak{C}}}(\vect{B})$,
$\Lambda \leq \mathsf{Aut}_{G \ltimes \Sigma^{op}_{\mathfrak{C}}}(\vect{C})$
and $\P(\vect{B})^{\Lambda}, \P(\vect{C})^{\Lambda}$
are well-defined.
At this point it may be tempting to think that, 
given a homotopy equivalence  
$\alpha \colon \mathfrak{b} \to \mathfrak{c}$,
one may simply apply $\Lambda$-fixed points 
to \eqref{ITERWE EQ}
to obtain a Kan equivalence 
$\P(\vect{B})^{\Lambda} \simeq \P(\vect{C})^{\Lambda}$.
However, this argument fails: not only are the intermediate objects
in \eqref{ITERWE EQ} not $\Lambda$-equivariant, so that the intermediate maps therein can not possibly be $\Lambda$-equivariant,
neither is it necessarily the case that the composite Kan equivalence
$\P(\vect{C}) \simeq \P(\vect{B})$ is $\Lambda$-equivariant,
unless one makes a further assumption on the homotopy equivalence
$\alpha\colon \mathfrak{b} \to \mathfrak{c}$.
Namely, since $\Lambda$
must contain the square of the element $i(12)(34)$,
which is $-1 \in G \leq G \times \Sigma_4^{op}$
one must have that 
$-\alpha =\alpha$
(note that $-\alpha \colon \mathfrak{b} \to \mathfrak{c}$
since $-\mathfrak{b}=\mathfrak{b},-\mathfrak{c}=\mathfrak{c}$).

It is then easy to check that, if $-\alpha = \alpha$,
the composite \eqref{ITERWE EQ}
is indeed $\Lambda$-equivariant
(this also follows from Lemma \ref{LAMBEQMAPS LEM},
which covers the general case), so that one indeed has
$\P(\vect{C})^{\Lambda} \simeq \P(\vect{B})^{\Lambda}$.
\end{example}

In the following, 
recall from Definition \ref{STABS DEF} that
for 
$\vect{C} = (\mathfrak{c}_1,\cdots,\mathfrak{c}_n;\mathfrak{c}_0)$
one has that $\Lambda$ stabilizes $\vect{C}$
iff 
$g \mathfrak{c}_{\sigma(i)} = \mathfrak{c}_i$
for all $(g,\sigma) \in \Lambda, 0\leq i \leq n$,
motivating the condition $g \kappa_{\sigma(i)} = \kappa_i$
in Lemma \ref{LAMBEQMAPS LEM}(i).

\begin{lemma}\label{LAMBEQMAPS LEM}
Fix a $G$-set of colors $\mathfrak{C}$.
Let $\P \in \mathsf{Op}^G_{\mathfrak{C}}(\V)$ be an operad,
$\vect{B}=(\mathfrak{b}_1,\cdots,\mathfrak{b}_n;\mathfrak{b}_0),\vect{C}=(\mathfrak{c}_1,\cdots,\mathfrak{c}_n;\mathfrak{c}_0)$ be $\mathfrak{C}$-profiles,
and suppose 
$\Lambda \leq G \times \Sigma_n^{op}$
stabilizes $\vect{B},\vect{C}$, i.e. 
$\Lambda \leq \mathsf{Aut}_{G\ltimes \Sigma_{\mathfrak{C}}^{op}}(\vect{B})$, 
$\Lambda \leq \mathsf{Aut}_{G\ltimes \Sigma_{\mathfrak{C}}^{op}}(\vect{C})$.

Moreover, suppose that for some $\mathbb{K} \in \V$
one has maps 
$\mathbb{K} \xrightarrow{\kappa_i} \P(\mathfrak{b}_i;\mathfrak{c}_i)$,  $0\leq i \leq n$
such that $g\kappa_{\sigma(i)} = \kappa_i$ for all $(g,\sigma) \in \Lambda$, $0\leq i \leq n$. Then:
\begin{enumerate}[label=(\roman*)]
\item if $\vect{B},\vect{C}$ have a common target $\mathfrak{b}_0=\mathfrak{c}_0$,
one has
$\Lambda$-equivariant maps as below, where the right map is the composition in $\P$
and the action of $(g,\sigma) \in \Lambda$ on $\mathbb{K}^{\otimes n}$
is the permutation action of $\sigma$.
\begin{equation}\label{LAMBEQMAPS EQ}
	\P(\vect{C}) \otimes 
	\mathbb{K}^{\otimes n}
\xrightarrow{id \otimes \underset{1\leq i \leq n}{\bigotimes} \kappa_i}
	\P(\vect{C}) \otimes \bigotimes_{1\leq i \leq n} \P(\mathfrak{b}_i;\mathfrak{c}_i)
	\xrightarrow{\circ}
	\P(\vect{B})
\end{equation}
\item
if $\vect{B},\vect{C}$ have common sources $\mathfrak{b}_i=\mathfrak{c}_i,1\leq i \leq n$,
one has
$\Lambda$-equivariant maps as below, where the right map is the composition in $\P$.
\[
	\mathbb{K} \otimes \P(\vect{B}) 
\xrightarrow{\kappa \otimes id}
	\P(\mathfrak{b}_0;\mathfrak{c}_0) \otimes \P(\vect{B}) 
	\xrightarrow{\circ}
	\P(\vect{C})
\]
\end{enumerate}
\end{lemma}

Note that, if $\mathbb{K}=1_{\V}$
in Lemma \ref{LAMBEQMAPS LEM},
we get $\Lambda$-equivariant maps
$\P(\vect{C}) \to \P(\vect{B})$,
$\P(\vect{B}) \to \P(\vect{C})$.

\begin{proof}
We discuss only (i), with (ii) following from a similar but easier argument. 

The fact that the left map in \eqref{LAMBEQMAPS EQ} is
$\Lambda$-equivariant can be deduced
from the $g \kappa_{\sigma(i)} = \kappa_i$ requirement
via direct calculation, but we prefer a more abstract argument.
The maps 
$\P(\mathfrak{b}_i;\mathfrak{c}_i) \to 
\P(g\mathfrak{b}_{\sigma(i)};g\mathfrak{c}_{\sigma(i)}) $
for $(g,\sigma) \in \Lambda$
make the tuple
$\left(\P(\mathfrak{b}_i;\mathfrak{c}_i) \right)_{1\leq i \leq n}$
into a $\Lambda$-equivariant object of 
$(\Sigma_n \wr \V^{op})^{op}$,
and likewise
$(\mathbb{K})_{1\leq i \leq n}
\xrightarrow{(\kappa_i)}
\left(\P(\mathfrak{b}_i;\mathfrak{c}_i) \right)_{1\leq i \leq n}
$
into a $\Lambda$-equivariant map in $(\Sigma_n \wr \V^{op})^{op}$.
Hence $\Lambda$-equivariance of the left map in 
\eqref{LAMBEQMAPS EQ}
follows by functoriality of 
$(\Sigma_n \wr \V^{op})^{op} \xrightarrow{\otimes} \V$.

To check that the right map in \eqref{LAMBEQMAPS EQ}
is also $\Lambda$-equivariant,
consider the $\mathfrak{C}$-trees below, including the profiles
$\vect{B},\vect{C}$.
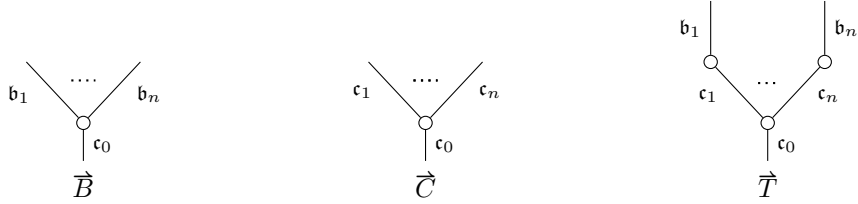
\[
\begin{tikzpicture}
      [grow=up,auto,level distance=2.3em,every node/.style = {font=\footnotesize},dummy/.style={circle,draw,inner sep=0pt,minimum size=1.75mm}]
      
      \node at (0,0) [font=\normalsize]{$\vect{B}$}
		child{node [dummy] {}
			child{
			edge from parent node [swap,near end] {$\mathfrak{b}_n$} node [name=Kn] {}}
			child{
			edge from parent node [near end] {$\mathfrak{b}_1$}
node [name=Kone,swap] {}}
		edge from parent node [swap] {$\mathfrak{c}_0$}
		};
		\draw [dotted,thick] (Kone) -- (Kn) ;
	\node at (4.5,0) [font=\normalsize]{$\vect{C}$}
		child{node [dummy] {}
			child{
			edge from parent node [swap,near end] {$\mathfrak{c}_n$} node [name=Kn] {}}
			child{
			edge from parent node [near end] {$\mathfrak{c}_1$}
node [name=Kone,swap] {}}
		edge from parent node [swap] {$\mathfrak{c}_0$}
		};
		\draw [dotted,thick] (Kone) -- (Kn) ;
		\draw [dotted,thick] (Kone) -- (Kn) ;
	\node at (9,0) [font=\normalsize]{$\vect{T}$}
		child{node [dummy] {}
			child{node [dummy] {}
				child{
				edge from parent node [swap] {$\mathfrak{b}_n$} node {}}
			edge from parent node [swap,near end] {$\mathfrak{c}_n$} node [name=Kn] {}}
			child{node [dummy] {}
				child{
				edge from parent node {$\mathfrak{b}_1$} node {}}
			edge from parent node [near end] {$\mathfrak{c}_1$}
node [name=Kone,swap] {}}
		edge from parent node [swap] {$\mathfrak{c}_0$}
		};
		\draw [dotted,thick] (Kone) -- (Kn) ;
\end{tikzpicture}
\]
Clearly the fact that $\Lambda$ stabilizes
$\vect{B},\vect{C}$
implies that $\Lambda$ stabilizes $\vect{T}$ as well.
Thus, the $\Lambda$-equivariance of the right map in 
\eqref{LAMBEQMAPS EQ}
follows by noting that said map is the multiplication
$\bigotimes_{v \in \boldsymbol{V}(T)} \P(\vect{T}_v)
\to 
\mathcal{P}(\mathsf{lr}(\vect{T}))
=
\mathcal{P}(\vect{B})$
encoded by the tree $\vect{T}$.
\end{proof}

\begin{remark}\label{CHOOSEKAPPA REM}
Generalizing Remark \ref{CHOOSESIGN REM},
a choice of 
$\kappa_i$ as in Lemma \ref{LAMBEQMAPS LEM}
is in bijection with a 
choice of $H_i$-equivariant maps
$\kappa_i \colon \mathbb{K} \to \mathcal{P}(\mathfrak{b}_i;\mathfrak{c}_i)$ for $i$ ranging over a set of representatives of
$\underline{n}_+/\Lambda$.
\end{remark}

In the next result we let $\mathbb{C}$ denote a good cylinder object for $1_{\V}$ (cf. \cite[Def. 4.2]{DS95}), 
meaning that one has a factorization
$1_{\V} \amalg 1_{\V} \rightarrowtail \mathbb{C} \xrightarrow{\sim} 1_{\V}$
of the fold map, where the first map is a cofibration
and the second map is a weak equivalence.

\begin{corollary}\label{ALBEETA COR}
Assume that $\V$ is a closed monoidal model category with cofibrant pushout powers,
and such that fixed points in $\V^G$ send genuine trivial cofibrations to trivial cofibrations
(i.e. $\V$ satisfies (ii),(iii),(v),(vii) in Theorem \ref{THMA}). Additionally, suppose 
that $\V$ satisfies the usual monoid axiom of \cite{SS00} (see also \cite[Rem. \ref{OC-MONAX_REM}]{BP_FCOP}).
      
Let $\P$, 
$\vect{B}=(\mathfrak{b}_1,\cdots,\mathfrak{b}_n;\mathfrak{b}_0)$,
$\vect{C}=(\mathfrak{c}_1,\cdots,\mathfrak{c}_n;\mathfrak{c}_0)$ 
and $\Lambda$ be as in 
Lemma \ref{LAMBEQMAPS LEM}.

Moreover, suppose 
$\vect{B},\vect{C}$
are ``$\Lambda$-homotopy equivalent'',
by which we mean that there exist 
\[
\alpha_i \colon 1_{\V} \to \P(\mathfrak{b}_i;\mathfrak{c}_i)
\qquad 
\beta_i \colon 1_{\V} \to \P(\mathfrak{c}_i;\mathfrak{b}_i)
\qquad
\eta_i \colon \mathbb{C} \to \P(\mathfrak{b}_i;\mathfrak{b}_i)
\qquad
\bar{\eta}_i \colon \mathbb{C} \to \P(\mathfrak{c}_i;\mathfrak{c}_i)
\qquad
0\leq i \leq n
\] 
with $\eta_i$ (resp. $\bar{\eta}_i$) a left homotopy between 
$\beta_i\alpha_i$ and $id_{\mathfrak{b}_i}$
(resp. 
$\alpha_i\beta_i$ and $id_{\mathfrak{c}_i}$)
and such that
\[
g \alpha_{\sigma(i)} = \alpha_i \qquad
g \beta_{\sigma(i)} = \beta_i \qquad
g \eta_{\sigma(i)} = \eta_i \qquad
g \bar{\eta}_{\sigma(i)} = \bar{\eta}_i \qquad
(g,\sigma) \in \Lambda,0\leq i \leq n.
\]
Then:
\begin{enumerate}[label=(\roman*)]
\item if $\vect{B},\vect{C}$ have a common target $\mathfrak{b}_0=\mathfrak{c}_0$, the precomposition maps
\[
\P(\vect{C})^{\Lambda} \xrightarrow{(\alpha_i)^{\**}} \P(\vect{B})^{\Lambda}
\qquad
\P(\vect{B})^{\Lambda} \xrightarrow{(\beta_i)^{\**}}
\P(\vect{C})^{\Lambda}
\]
induced by $\alpha_i,\beta_i,1\leq i \leq n$ 
(cf. Lemma \ref{LAMBEQMAPS LEM}(i)) are weak equivalences in $\V$;
\item if $\vect{B},\vect{C}$ have common sources $\mathfrak{b}_i=\mathfrak{c}_i, 1 \leq i \leq n$, the postcomposition maps
\[
\P(\vect{B})^{\Lambda} \xrightarrow{(\alpha_0)_{\**}}
\P(\vect{C})^{\Lambda}
\qquad
\P(\vect{C})^{\Lambda} \xrightarrow{(\beta_0)_{\**}}
\P(\vect{B})^{\Lambda}
\]
induced by $\alpha_0,\beta_0$
(cf. Lemma \ref{LAMBEQMAPS LEM}(ii))
are weak equivalences in $\V$.
\end{enumerate}
\end{corollary}

\begin{proof}
We address only (i), with (ii) being similar but easier.

Applying Lemma \ref{LAMBEQMAPS LEM},
one obtains diagrams as below,
which will show that
$(\alpha_i)^{\**}$ and
$(\beta_i)^{\**}$
are inverse up to homotopy provided we show
$\P(\vect{B})^\Lambda \otimes 
\left(\mathbb{C}^{\otimes n}\right)^{\Lambda}$
is a cylinder on $\P(\vect{B})^\Lambda$, and likewise for
$\P(\vect{C})^\Lambda$.
\begin{equation}\label{ALBEETA EQ}
\begin{tikzcd}[column sep=31pt]
	\P(\vect{B})^\Lambda \amalg \P(\vect{B})^\Lambda 
	\arrow[d] \arrow[r, "{((\beta_i)^{\**}(\alpha_i)^{\**}, id)}"]
&
	\P(\vect{B})^\Lambda
&
	\P(\vect{C})^\Lambda \amalg \P(\vect{C})^\Lambda 
	\arrow[d] \arrow[r, "{((\alpha_i)^{\**}(\beta_i)^{\**}, id)}"]
&
	\P(\vect{C})^\Lambda
\\                  
	\P(\vect{B})^\Lambda \otimes 
	\left(\mathbb{C}^{\otimes n}\right)^{\Lambda}
	\arrow[r]
&
	\left(\P(\vect{B}) \otimes 
	\mathbb{C}^{\otimes n}\right)^{\Lambda}
	\arrow{u}[swap]{(\eta_i)^{\**}}
&
	\P(\vect{C})^\Lambda \otimes 
	\left(\mathbb{C}^{\otimes n}\right)^{\Lambda}
	\arrow[r]
&
	\left(\P(\vect{C}) \otimes 
	\mathbb{C}^{\otimes n}\right)^{\Lambda}
	\arrow{u}[swap]{(\bar{\eta}_i)^{\**}}
\end{tikzcd}
\end{equation}
This amounts to checking that the canonical map
$\P(\vect{B})^\Lambda \otimes 
\left(\mathbb{C}^{\otimes n}\right)^{\Lambda}
\to \P(\vect{B})^\Lambda$
induced by $\mathbb{C} \to 1_{\V}$
is a weak equivalence, and by $2$-out-of-$3$ it suffices to show this for one of the sections
$\P(\vect{B})^\Lambda \to
\P(\vect{B})^\Lambda \otimes 
\left(\mathbb{C}^{\otimes n}\right)^{\Lambda}$.
But, since 
$1_{\V} \simeq 1_{\V}^{\otimes} \to \mathbb{C}^{\otimes n}$
is a $\Sigma_n$-genuine trivial cofibration  
\cite[Prop. \ref{OC-SIGMAWRGF PROP}(ii)]{BP_FCOP},
the fixed point map
$1_{\V} \to \left(\mathbb{C}^{\otimes n}\right)^{\Lambda}$
is a trivial cofibration in $\V$ by assumption on $\V$,
and thus the required claim follows by the monoid axiom.
\end{proof}


\begin{remark}\label{MONAXSUP REM}
	The monoid axiom assumption in Corollary \ref{ALBEETA COR} 
	is actually superfluous.
	To see this, note first that if
	$\mathcal{P}(\vect{B})^{\Lambda}$ happens to be cofibrant in $\V$,
	then the claim that
	$\mathcal{P}(\vect{B})^{\Lambda} \otimes (\mathbb{C}^{\otimes n})^{\Lambda}$
	is a cylinder follows from $\V$ being a monoidal model category.
	But then, writing $Q \xrightarrow{\sim} \mathcal{P}(\vect{B})^{\Lambda}$
	for a cofibrant replacement,
	by prepending 
	$Q \amalg Q \to Q \otimes (\mathbb{C}^{\otimes n})^{\Lambda}$
	to the left square in \eqref{ALBEETA EQ}
	one still has that 
	$(\beta_i)^{\**}(\alpha_i)^{\**}$
	and
	$id$
	are homotopic, and similarly for the right square in \eqref{ALBEETA EQ}.
\end{remark}

\begin{remark}\label{ALBEETA_REM}
If in Corollary \ref{ALBEETA COR}
one has that $\P \in \mathsf{Cat}^G_{\mathfrak{C}}(\V)$
is a $\V$-category,
the only interesting case is that of
$\vect{B},\vect{C}$ unary profiles.
But then in the proof it is $n=1$,
and $\Lambda$ necessarily acts trivially on
$\mathbb{C}^{\otimes n} = \mathbb{C}$,
so in this case the result follows without using either condition (v) or (vii) in Theorem \ref{THMA}.
\end{remark}

\begin{proof}[Proof of Proposition \ref{23HARDCASE PROP}]
Set $\mathfrak{C} = \mathfrak{C}_{\P}$.
We need to show that,
for every $\mathfrak{C}$-profile $\vect{C}$ and
$\Lambda \in \F_{\vect{C}}$,
the map 
$\P(\vect{C})^{\Lambda} \to \mathcal{Q}(\bar{F}(\vect{C}))^{\Lambda}$
is a weak equivalence in $\V$.

Moreover, since one has a functorial $\F$-fibrant replacement functor
(fixing object sets)
and natural transformation $\O \to \O_f$,
we reduce to the case where all of $\O,\P,\mathcal{Q}$
are $\F$-fibrant.

As in Remarks \ref{CHOOSESIGN REM} and \ref{CHOOSEKAPPA REM},
write $\Lambda_i$ for the stabilizer of $i \in \underline{n}_+$ under the action of $\Lambda$,
and $H_i =\pi_G(\Lambda_i)$ for the projection.
Note that the requirement that $\F$ has enough units 
says precisely that $H_i \in \F_1$ for all $i$.
Using $\F$-$\pi_0$-essential surjectivity allows us to,
for $i$ ranging over a set of representatives of
$\underline{n}_+/\Lambda$,
find $H_i$-fixed $\mathfrak{a}_i\in \mathfrak{C}_{\O}$
together with $H_i$-fixed isomorphisms 
$F(\mathfrak{a}_i) =  \mathfrak{b}_i \xrightarrow{[\alpha_i]} \mathfrak{c}_i$
in $\pi_0 j^{\**}\P^H$.
These now yield $\alpha_i,\beta_i,\eta_i,\bar{\eta}_i$
as in Corollary \ref{ALBEETA COR}
(note that we first choose these for $i$ in the set of representatives of $\underline{n}_+/\Lambda$,
then extend them to all $i$ by conjugation,
cf. Remark \ref{CHOOSEKAPPA REM})
so that by Corollary \ref{ALBEETA COR}
we have the following commutative square,
where the horizontal maps are weak equivalences in $\V$.
\begin{equation}\label{23SQDIAG EQ}
\begin{tikzcd}[column sep = 45pt]
	\P(\vect{C})^{\Lambda}
	\arrow[d]
	\arrow{r}{\sim}[swap]{(\alpha_i)^{\**}(\beta_0)_{\**}}
&
	\P(\vect{B})^{\Lambda}
	\arrow[d]
\\
	\mathcal{Q}\left(\bar{F}(\vect{C})\right)^{\Lambda}
	\arrow{r}{\sim}[swap]{(\bar{F}\alpha_i)^{\**}(\bar{F}\beta_0)_{\**}}
&
	\mathcal{Q}\left(\bar{F}(\vect{B})\right)^{\Lambda}
\end{tikzcd}
\end{equation}
%
%
Next write $\vect{A} = \{\mathfrak{a}_1,\cdots,\mathfrak{a}_n;\mathfrak{a}_0\}$
for the $\mathfrak{C}_{\O}$-profile determined by the 
$\mathfrak{a}_i$, where we again use Remark \ref{CHOOSESIGN REM}
(recall that the $\mathfrak{a}_i$ were only chosen for $i$ in a set of representatives of $\underline{n}_+/\Lambda$),
which moreover implies that $\Lambda$ stabilizes $\vect A$, and thus $\Lambda \in \F_{\vect{A}}$.
The result now follows by applying $2$-out-of-$3$ to both the following diagram (where $\vect{B} = F(\vect{A})$, so that the arrows marked $\sim$ are weak equivalences by assumption) and \eqref{23SQDIAG EQ},
yielding that $\P(\vect{C})^{\Lambda} \to \mathcal{Q}(\bar{F}(\vect{C}))^{\Lambda}$
is also a weak equivalence, as desired.
\begin{equation}
	\begin{tikzcd}[row sep=0]
		\O(\vect{A})^{\Lambda} \arrow{rr}{\sim}
		\arrow{dr}[swap]{\sim}
	&&
		\mathcal{Q}(\bar{F}(\vect{B}))^{\Lambda}
	\\
	&
		\P(\vect{B}) \ar{ru}
	\end{tikzcd}
\end{equation}
\end{proof}

\subsection{Characterizing fibrations}\label{ISOFIB_SEC}

In addition to the fibrations in Definition \ref{MODEL_DEFN},
and as noted in Remark \ref{FIBSALT REM},
there is another natural notion of fibration in $\mathsf{Op}^G_{\bullet}(\mathcal{V})$,
which parallels the notion of Dwyer-Kan equivalence
by replacing the 
$\F$-path lifting condition
with the analogous condition on the
$j^{\**} \pi_0(-)$ categories.

\begin{definition}
	Let $\F$ be a $(G,\Sigma)$-family that has enough units
	(Definition \ref{FAMRESUNI DEF}).
	
	We say a map $\O \to \P$ in $\Op_\bullet^G(\V)$ is an \textit{$\F$-isofibration} if
	it is a local $\F$-fibration and
	$j^{\**}\pi_0 \O^H \to j^{\**}\pi_0\P^H$
	is an isofibration of categories for all $H \in \F_1$.
\end{definition}

Our goal in this section is to compare the 
notions of $\F$-fibration and $\F$-isofibration,
with the relevant results given by Propositions \ref{ISOFIBEASY PROP} and \ref{ISOFIBHARD PROP}.

We start with the easier direction.

\begin{proposition}[{cf. \cite[Prop. 2.3]{Ber07b}}]
	\label{ISOFIBEASY PROP}
	$\F$ be a $(G,\Sigma)$-family, 
	suppose $\V$ satisfies the coherence axiom,
	and let 
	$F: \O \to \P$ in $\Op^G_\F(\V)$
	be $\F$-path lifting.
	
	Then $F$ is also an $\F$-isofibration
	provided that either $\mathcal{P}$ is $\F$-fibrant
	or $\V$ is right proper.
\end{proposition}

\begin{proof}
	Considering each map 
	$j^{\**} \pi_0\left(\O \to \P \right)^H$,
	we reduce to 
	the case of $F \colon \mathcal{C} \to \mathcal{D}$
	a map in $\mathsf{Cat}_{\bullet} (\V)$.
	Given $a \in \mathcal{C}$
	and isomorphism $[\alpha]$ in $\pi_0 \mathcal{D}$
	with $[\alpha](0)=F(a)$,
	using either the definition of coherence 
	if $\mathcal{D}$ is fibrant or 
	Proposition \ref{ALTCOH PROP},
	we obtain a bottom horizontal map below 
	\begin{equation}
	\begin{tikzcd}[column sep = 45pt]
	\eta \ar{d} \ar{r}{a}
	&
	\mathcal{C} \ar{d}{F}
	\\
	\mathbb{J} \ar{r} \ar[dashed]{ru}
	&
	\mathcal{D}
	\end{tikzcd}
	\end{equation}
	such that $\pi_0 \mathbb{J} \to \pi_0\mathcal{D}$
	maps the natural isomorphism $[id_{1_{\V}}]$ to $[\alpha]$.
	And,
	since the lift in the diagram exists due to $F$ being path lifting,
	the image $[\bar{\alpha}]$
	of the natural isomorphism $[id_{1_{\V}}]$
	under $\pi_0 \mathbb{J} \to \pi_0\mathcal{C}$
	lifts $[\alpha]$
	and satisfies $[\bar{\alpha}](0) = a$,
	showing that
	$\pi_0 \mathcal{C} \to \pi_0 \mathcal{D}$
	is indeed an isofibration.
\end{proof}

Proposition \ref{ISOFIBEASY PROP}
implies that, if $\V$ is right proper,
all $\F$-fibrations $F \colon \mathcal{O} \to \mathcal{P}$ are $\F$-isofibrations.
%
The converse direction
requires some preparation.
We first list two necessary lemmas.

\begin{lemma}[{cf. \cite[Lemma 2.6]{Ber07b}}]
	\label{TRANSFLIFT LEMMA}
	Let $\V$ be a model category.
        \begin{enumerate}[(i)]
        \item Consider the diagram below, 
                where the map marked $\sim$ is a weak equivalence, 
                $\rightarrowtail$ is a cofibration
                and $\twoheadrightarrow$ is a fibration.
                \begin{equation}\label{TRANSFLIFT EQ}
                        \begin{tikzcd}[column sep = 45pt]
                                A
                                \ar{d}
                                \arrow[equal]{r}
                                &
                                A \ar{r}
                                \arrow[rightarrowtail]{d}
                                &
                                X \arrow[twoheadrightarrow]{d}
                                \\
                                B'
                                \arrow{r}[swap]{\sim}
                                &
                                B\ar{r}
                                &
                                Y
                        \end{tikzcd}
                \end{equation}
                Then a lift $B \to X$ 
                exists iff
                a lift $B' \to X$ 
                exists.
        \item Dually, given the diagram below (and following the same conventions for arrows)
                \begin{equation}
                        \label{TRANSLIFTDUAL EQ}
                        \begin{tikzcd}[column sep = 45pt]
                                A \arrow[r] \arrow[d, rightarrowtail]
                                &
                                X \arrow[r, "\sim"] \arrow[d, twoheadrightarrow]
                                &
                                X' \arrow[d]
                                \\
                                B \arrow[r]
                                &
                                Y \arrow[r, equal]
                                &
                                Y
                        \end{tikzcd}
                \end{equation}
                a lift $B \to X$ exists iff a lift $B \to X'$ exists.
        \end{enumerate}
\end{lemma}

\begin{remark}\label{UNDEROVER REM}
	Recall (cf. \cite[Rem. 3.10]{DS95})
	that for any $A \in \V$ the undercategory $\V_{A/}$
	has a model structure such that a map is a 
	weak equivalence or (co)fibration iff it becomes one under the 
	forgetful functor
	$\V_{A/} \to \V$ given by $(A \to X) \mapsto X$,
	and dually for overcategories $\V_{/Y}$.
	
	Hence, for any map $A \to Y$,
	the category 
	$\V_{A//Y} = \left(\V_{A/}\right)_{/(A\to Y)}$
	of factorizations $A \to C \to Y$
	likewise has a model structure determined by the forgetful functor
	$\V_{A//Y} \to \V$.
\end{remark}

\begin{proof}[Proof of Lemma \ref{TRANSFLIFT LEMMA}]
        It suffices to prove (i).
	We will argue using Remark \ref{UNDEROVER REM}
	(\cite{Ber07b} gives a more explicit argument).
	
	Only the ``if'' direction requires proof.
	A lift $B' \to X$ implies that,
	in the homotopy category of $\V_{A//Y}$, it is
	$\Ho \V_{A//Y}(B',X) \neq \emptyset$. 
	And since $B,B'$ are weak equivalent in 
	$\V_{A//Y}$, it is also 
	$\Ho \V_{A//Y}(B,X) \neq \emptyset$.
	But the (co)fibrancy assumptions in \eqref{TRANSFLIFT EQ}
	say that $B$ is cofibrant in $\V_{A//Y}$ while 
	$X$ is fibrant, so any map in $\Ho \V_{A//Y}(B,X) \neq \emptyset$
	is represented by an actual lift $B \to X$.
\end{proof}

\begin{remark}\label{NOTMATCH REM}
	We caution that, in the previous proof, when given a lift $f'\colon B' \to X$,
	the induced lift $f \colon B \to X$
	needs not be such that the composite $B' \to B \xrightarrow{f} X$
	equals $f'$. Rather, these need only be homotopic in 
	$\V_{A//Y}$.
\end{remark}

\begin{lemma}\label{HOMINPULL LEM}
	Consider a diagram 
	in $\mathsf{Cat}_{\{0,1\}}(\V)$
	\begin{equation}
	\begin{tikzcd}[column sep = 45pt]
	\mathbbm{1}
	\arrow{d}[swap]{\bar{\alpha}}
	\arrow{r}{\alpha}
	&
	\mathcal{C}
	\arrow[twoheadrightarrow]{d}{F}
	\\
	\bar{\mathcal{C}}
	\arrow{r}[swap]{\bar{F}}
	&
	\mathcal{D}
	\end{tikzcd}
	\end{equation}
	such that $\alpha,\bar{\alpha}$ encode homotopy equivalences in 
	$\mathcal{C}, \bar{\mathcal{C}}$ and
	$F \colon \mathcal{C} \twoheadrightarrow \mathcal{D}$
	is a (local) fibration.
	
	Then the induced map 
	$\mathbbm{1} \xrightarrow{(\alpha,\bar{\alpha})}
	\mathcal{C} \times_{\mathcal{D}}\bar{\mathcal{C}}$
	encodes a homotopy equivalence in $\mathcal{C} \times_{\mathcal{D}}\bar{\mathcal{C}}$
	provided that either:
	\begin{enumerate*}[label = (\roman*)]
		\item $\mathcal{D}$, $\bar{\mathcal{C}}$ are fibrant or;
		\item $\V$ is right proper.
	\end{enumerate*}
\end{lemma}

The proof of Lemma \ref{HOMINPULL LEM} requires preparation,
and is postponed to the end of the section.

We now discuss a further assumption on 
the unit $1_{\V}$ that is needed to guarantee that 
\emph{all} $\F$-isofibrations are $\F$-fibrations.

\begin{definition}
	Let $\V$ be a model category with $A$ a cofibrant object and 
	$X$ any object.
	We say \emph{$X$ is fibrant with respect to $A$}
	if a fibrant replacement $X\to X_f$
	induces an isomorphism
	\[
	[A,X] \xrightarrow{\simeq} [A,X_f] = \Ho \V (A,X)
	\]
	of left homotopy classes of maps.
	Explicitly, this means that any 
	map $A \xrightarrow{[f]} X$ in $\Ho \V (A,X)$
	is represented by a map 
	$A \xrightarrow{f} X$ in $\V$
	and that, if $[f]=[g]$, there is an exhibiting left homotopy
	$A \otimes \mathbb{C} \xrightarrow{H} X$
	(where, cf. Corollary \ref{ALBEETA COR}, 
	$\mathbb{C}$ denotes a good cylinder object for $1_{\mathcal{V}}$ \cite[Def. 4.2]{DS95}). 
\end{definition}

\begin{example}
	In the Kan model structure on $\V=\mathsf{sSet}$   
	all objects are fibrant with respect to $\** = 1_{\V}$.
	Further, fibrant objects are always fibrant with respect to any cofibrant object.
	Thus, in the canonical model structures on 
	$\mathsf{Set}, \mathsf{Cat}, \mathsf{Top}$,
	all objects are fibrant with respect to the unit $1_{\V}$.
\end{example}

\begin{remark}\label{RELFIBLIFT REM}
	Suppose $F\colon X \to Y$ is a fibration between two objects
	that are fibrant with respect to $A$.
	Then, given 
	$f\colon A \to X$ and $g \colon A \to Y$
	such that
	$[Ff] = [g]$
	one can find 
	$f'\colon A \to X$ such that $Ff' =g$.
	Indeed, this follows from the existence of a 
	lift in 
	\begin{equation}
	\begin{tikzcd}[column sep = 45pt]
	A \ar[>->]{d}[swap]{\sim} \ar{r}{f}
	&
	X \ar[twoheadrightarrow]{d}{F}
	\\
	A \otimes \mathbb{C} \ar{r}[swap]{H} \ar[dashed]{ru}
	&
	Y
	\end{tikzcd}
	\end{equation}
	where $H$ is a left homotopy between $Ff$ and $g$.
\end{remark}

We can now prove a partial converse to Proposition \ref{ISOFIBEASY PROP}, 
adapting \cite[Prop. 2.5]{Ber07b}.

\begin{proposition}
	\label{ISOFIBHARD PROP}
	Let $\F$ be a $(G,\Sigma)$-family that has enough units.
	Moreover, suppose $\V$ has cofibrant unit
	and satisfies the coherence axiom,
	and let $F: \O \to \P$ in $\Op^G_\F(\V)$
	be an $\F$-isofibration.
	
	Then $F$ is also an $\F$-fibration
	provided that either:
	\begin{enumerate*}[label = (\roman*)]
		\item $\P$ is fibrant or;
		\item $\V$ is right proper and
		all objects in $\V$ are fibrant with respect to $1_{\V}$.
	\end{enumerate*}
\end{proposition}

\begin{proof}
	As in the proof of Proposition \ref{ISOFIBEASY PROP},
	we reduce to the case of $F \colon \mathcal{C} \to \mathcal{D}$
	an isofibration in $\mathsf{Cat}_{\bullet}(\V)$.
	And, since $F$ is a local fibration by assumption, 
	the task is to show that if
	$F$ is a $\pi_0$-isofibration then it is also path lifting, 
	i.e. that we can solve any lifting problem as on the left below
	(where $\mathbb{J}$ is an interval, as usual).
	\begin{equation}\label{FIRSTRED EQ}
	\begin{tikzcd}[column sep = 45pt]
	\eta \ar{d} \ar{r}{a}
	&
	\mathcal{C} \ar{d}{F}
	&
	&
	&
	\iota^{\**}_{a,b} \mathcal{C} \ar[twoheadrightarrow]{d}
	\\
	\mathbb{J} \ar{r} \ar[dashed]{ru}
	&
	\mathcal{D}
	&
	\mathbb{J} \ar[rightarrowtail]{r}{\sim}
	&
	\mathbb{J}' \ar[twoheadrightarrow]{r} \ar[dashed]{ru}
	&
	\iota^{\**}_{Fa,Fb} \mathcal{D}
	\end{tikzcd}
	\end{equation}
	
	Writing $[\alpha]$ for the image in
	$\pi_0 \mathcal{D}$
	of the natural isomorphism (cf. Remark \ref{NATISO REM})
	in $\pi_0 \mathbb{J} \simeq \pi_0 \widetilde{\mathbbm{1}}$,
	the fact that $F$ is an isofibration
	yields a lifted isomorphism
	$[\bar{\alpha}]$ in $\pi_0 \mathcal{C}$
	between $a \in \mathcal{C}$ and some other object $b \in \mathcal{C}$.
	This allows us to form the lifting problem 
	in $\mathsf{Cat}_{\{0,1\}}(\V)$
	on the right of 
	\eqref{FIRSTRED EQ}
	(where we factor the bottom map as a trivial cofibration followed by a fibration),
	and it clearly suffices to solve this alternate problem.
	
	We now claim that we can form a diagram in $\mathsf{Cat}_{\{0,1\}}(\V)$ as below
	\begin{equation}\label{GETMAP EQ}
	\begin{tikzcd}
	\mathbbm{1} \ar{d} \ar{r}
	&
	\iota^{\**}_{a,b} \mathcal{C} \ar[twoheadrightarrow]{d}
	\\
	\mathbb{J}' \ar[twoheadrightarrow]{r} 
	&
	\iota^{\**}_{Fa,Fb} \mathcal{D}
	\end{tikzcd}
	\end{equation}
	and for which $\mathbbm{1} \to \mathbb{J}'$
	represents the natural isomorphism $[id_{1_{\V}}]$ of 
	$\pi_0 \mathbb{J}' \simeq \pi_0 \widetilde{\mathbbm{1}}$ 
	and 
	$\mathbbm{1} \to \iota^{\**}_{a,b} \mathcal{C}$
	represents the isomorphism $[\bar{\alpha}]$ in $\mathcal{C}$.
	Indeed, in either case (i) or (ii)
	our assumptions guarantee that 
	the mapping objects of all three of 
	$\mathbb{J}',\iota^{\**}_{Fa,Fb} \mathcal{D},\iota^{\**}_{a,b} \mathcal{C}$
	are fibrant with respect to $1_{\V}$,
	so one can certainly choose maps 
	$\mathbbm{1} \to \mathbb{J}'$,
	$\mathbbm{1} \to \iota^{\**}_{a,b} \mathcal{C}$
	representing the natural isomorphism $[id_{1_{\V}}]$ and $[\bar{\alpha}]$, 
	though a priori one has no guarantee that the composites 
	$\mathbbm{1} \to \mathbb{J}' \to \iota^{\**}_{Fa,Fb} \mathcal{D}$,
	$\mathbbm{1} \to \iota^{\**}_{a,b} \mathcal{C} \to
	\iota^{\**}_{Fa,Fb} \mathcal{D}$
	coincide.
	Nonetheless, since both composites represent 
	$[\alpha]$,
	by Remark \ref{RELFIBLIFT REM}
	we can make the square commute by replacing one of the maps up to homotopy.
	
	We now form the following solid square
	\begin{equation}
	\begin{tikzcd}
	\mathbb{J}'' \ar[dashed]{r}
	&
	\mathcal{E} \ar{d} \ar{r}
	&
	\iota^{\**}_{a,b} \mathcal{C} \ar[twoheadrightarrow]{d}
	\\
	&
	\mathbb{J}' \ar[twoheadrightarrow]{r} 
	&
	\iota^{\**}_{Fa,Fb} \mathcal{D}
	\end{tikzcd}
	\end{equation}
	where $\mathcal{E}$ is simply the pullback.
	The diagram 
	\eqref{GETMAP EQ}
	yields a map $\mathbbm{1} \to \mathcal{E}$ that,
	by Lemma \ref{HOMINPULL LEM},
	encodes a homotopy equivalence.
	Therefore, using either the definition of coherence in case (i) or Proposition \ref{ALTCOH PROP} in case (ii),
	we obtain a map out of an interval 
	$\mathbb{J}'' \to \mathcal{E}$
	with the property that 
	the composite 
	$\pi_0\mathbb{J}'' \to \pi_0\mathcal{E} \to \pi_0\mathbb{J}'$
	sends the natural isomorphism to itself.
	In particular, this
	$\mathbb{J}''(0,1) \xrightarrow{\sim} \mathbb{J}'(0,1)$ 
	is a weak equivalence in $\V$
	so that \cite[Lemma 2.12]{BM13}
	(or its generalization Corollary \ref{ALBEETA COR})
	implies
	$\mathbb{J}'' \xrightarrow{\sim} \mathbb{J}'$ 
	is itself a weak equivalence.
	
	The required lift in \eqref{FIRSTRED EQ} now follows from 
	Lemma \ref{TRANSFLIFT LEMMA}(i)
	applied to the category
	$\mathsf{Cat}_{\{0,1\}}(\V)$
	with $A$ the initial object,
	$B' \xrightarrow{\sim} B$ the map
	$\mathbb{J}'' \xrightarrow{\sim} \mathbb{J}'$,
	and $X \twoheadrightarrow Y$
	the map 
	$\iota^{\**}_{a,b} \mathcal{C}
	\twoheadrightarrow
	\iota^{\**}_{Fa,Fb} \mathcal{D}$.
\end{proof}

The remainder of this section addresses the postponed proof of
Lemma \ref{HOMINPULL LEM}.
We first make some remarks about 
the model structures
on $\V_{A/}$, $\V_{A//Y}$ in Remark \ref{UNDEROVER REM}.

\begin{remark}\label{UNDFGT REM}
	Since the forgetful functor $\V_{A//Y} \to \V$
	preserves all weak equivalences, 
	it preserves left and right homotopies \cite[\S 4.1,\S 4.12]{DS95} between maps.
\end{remark}

\begin{remark} \label{LEFTQUILUND REM}
	For any map $A \to A'$ the induced adjunction
	$A' \amalg_A (-) \colon \V_{A/} 
	\rightleftarrows 
	\V_{A'/} \colon \mathsf{fgt}$
	is Quillen.
	In particular,
	given a cofibration 
	$A \rightarrowtail B$
	and map $A' \to X$ with fibrant $X$, one has
	\[
	\Ho \V_{A/}\left(A \rightarrowtail B, A \to X\right)
	\simeq
	\Ho \V_{A'/}\left(A' \rightarrowtail A' \amalg_A B, A' \to X\right).
	\]
\end{remark}

\begin{lemma}\label{LIFTEQUIV LEM}
	Let $\V$ be a model category and
	consider the lifting problems below, where $A,B$ are cofibrant, 
	the common map $A \rightarrowtail B$ is a cofibration, and $X$ is fibrant.
	\begin{equation}
	\begin{tikzcd}[column sep = 45pt]
	A
	\ar[rightarrowtail]{d}
	\arrow{r}{f}
	&
	X
	&%
	A
	\ar[rightarrowtail]{d}
	\arrow{r}{g}
	&
	X
	\\
	B \ar[dashed]{ru}[swap]{F}
	&
	&%
	B \ar[dashed]{ru}[swap]{G}
	&
	\end{tikzcd}
	\end{equation}
	Then, if $f$ and $g$ are homotopic 
	(i.e. they coincide in $\Ho \V(A,X)$),
	a lift $F$ exists iff a lift $G$ exists.
\end{lemma}

\begin{proof}
	Let 
	$X \overset{\sim}{\rightarrowtail} PX 
	\twoheadrightarrow X \times X$
	be a choice of path object for $X$ \cite[\S 4.12]{DS95},
	and $A \xrightarrow{H} PX$
	be a right homotopy between $f,g$,
	i.e. writing $p_1,p_2 \colon PX \to X$ for the two projections,
	one has
	$p_1H=f$, $p_2H=g$.
	The result now follows from Lemma \ref{TRANSFLIFT LEMMA}(ii) 
	applied to the following
	\begin{equation}
	\begin{tikzcd}[column sep = 45pt]
	A
	\ar[rightarrowtail]{d}
	\arrow{r}{H}
	&
	PX
	\arrow[twoheadrightarrow]{d} \ar{r}{\sim}[swap]{p_1}
	&
	X \ar{d}
	&%
	A
	\ar[rightarrowtail]{d}
	\arrow{r}{H}
	&
	PX
	\arrow[twoheadrightarrow]{d} \ar{r}{\sim}[swap]{p_2}
	&
	X \ar{d}
	\\
	B
	\arrow{r}
	&
	\**\ar{r}
	&
	\**
	&%
	B
	\arrow{r}
	&
	\**\ar{r}
	&
	\**
	\end{tikzcd}
	\end{equation}
	that shows that lifts $B \to X$ in either diagram
	exist iff a lift $B \to PX$ exists.
\end{proof}

In the remainder of the section we write
$\mathbb{C}_{\bullet} \in \V^{\Delta}$
for a cosimplicial frame on $1_{\V}$
\cite[Def. 16.6.1]{Hir03}.
In particular, this means that 
$\mathbb{C}_0 = 1_{\V}$
and that the degeneracy maps
$\mathbb{C}_n \to \mathbb{C}_0$ 
are weak equivalences.
Moreover, $\mathbb{C}_{\bullet} \in \V^{\Delta}$ is Reedy cofibrant so that, writing
$\mathbb{C}_{K} = \colim_{[n] \to K} \mathbb{C}_n$
for $K \in \mathsf{sSet}$,
one has that 
$\mathbb{C}_{K} \to \mathbb{C}_{L}$
is a cofibration in $\V$
whenever $K\to L$ is a monomorphism in $\mathsf{sSet}$.
In addition, $\mathbb{C} = \mathbb{C}_1$
is then a good cylinder on $1_{\V}$, in the sense of \cite[Def. 4.2(i)]{DS95}.

To avoid the need to label arrows, 
we will write
$\mathbb{C}_{\{i,j\}} \to \mathbb{C}_{\{0,1,2\}}$
to denote the map
$\mathbb{C}_{1} \to \mathbb{C}_{2}$
induced by the inclusion $\{i,j\} \subset \{0,1,2\}$,
and similarly for $\mathbb{C}_{\{i\}} \to \mathbb{C}_{\{0,1,2\}}$.

The following is the key to proving Lemma \ref{HOMINPULL LEM}.

\begin{lemma}\label{HOMTPOFHOMTP LEM}
	Let $\mathcal{D} \in \mathsf{Cat}_{\{0,1\}}(\V)$
	be fibrant and suppose 
	$\alpha \colon 1_{\V} \to \mathcal{D}(0,1)$
	is a homotopy equivalence. 
	Moreover, let 
	\[\beta \colon 1_{\V} \to \mathcal{D}(1,0), \qquad
	\bar{\beta} \colon 1_{\V} \to \mathcal{D}(1,0),\qquad
	H \colon \mathbb{C} \to \mathcal{D}(0,0), \qquad
	\bar{H} \colon \mathbb{C} \to \mathcal{D}(0,0)
	\]
	be two left homotopy inverses $\beta,\bar{\beta}$ to $\alpha$
	together with exhibiting homotopies 
	$H,\bar{H}$
	between $id_0$ and $\beta \alpha$, $\bar{\beta}\alpha$.
	Then there exists a commutative diagram (with $\alpha^{\**}$ the precomposition with $\alpha$)
	\begin{equation}
	\begin{tikzcd}[column sep = 45pt]
	\mathbb{C}_{\{1,2\}}
	\ar{d}
	\arrow{r}{B}
	&
	\mathcal{D}(1,0)
	\arrow{d}{\alpha^{\**}}[swap]{\sim}
	\\
	\mathbb{C}_{\{0,1,2\}}
	\arrow{r}[swap]{\mathcal{H}}
	&
	\mathcal{D}(0,0)
	\end{tikzcd}
	\end{equation}
	that satisfies the following compatibilities with restrictions
	\[
	B|_{\mathbb{C}_{\{1\}}} = \beta, \qquad
	B|_{\mathbb{C}_{\{2\}}} = \bar{\beta}, \qquad
	\mathcal{H}|_{\mathbb{C}_{\{0,1\}}} = H, \qquad
	\mathcal{H}|_{\mathbb{C}_{\{0,2\}}} = \bar{H}.
	\]
\end{lemma}

We note that, informally, $B$ is an homotopy between $\beta$ and $\bar{\beta}$
while $\mathcal{H}$
is a compatible homotopy of homotopies 
between $H$ and $\bar{H}$
(except encoded by a ``triangle'' rather than a ``square'').

\begin{proof}
	We first build the dashed lift $\bar{\mathcal{H}}$
	as on the left below (that exists since we are lifting a trivial cofibration against a fibrant object).
	Using this $\bar{\mathcal{H}}$ we now consider the right diagram,
	\begin{equation}
	\begin{tikzcd}[column sep = 35pt]
	\mathbb{C}_{\{0,2\}} 
	\amalg_{\mathbb{C}_{\{0\}}}
	\mathbb{C}_{\{0,1\}} 
	\ar[rightarrowtail]{d}[swap]{\sim}
	\arrow{r}{(\bar{H},H)}
	&
	\mathcal{D}(0,0)
	&%
	\mathbb{C}_{\{1\}} \amalg \mathbb{C}_{\{2\}}
	\ar[rightarrowtail]{d}
	\arrow{r}{(\beta,\bar{\beta})}
	&
	\mathcal{D}(1,0)
	\ar{r}{\sim}[swap]{\alpha^{\**}}
	&
	\mathcal{D}(0,0) 
	\\
	\mathbb{C}_{\{0,1,2\}}
	\arrow[dashed]{ur}[swap]{\widetilde{\mathcal{H}}}
	&
	&%
	\mathbb{C}_{\{1,2\}}
	\ar[dashed]{ru}{B}
	\ar[bend right=2]{rru}[swap]{\bar{\mathcal{H}}|_{\mathbb{C}_{\{1,2\}}}}
	&
	&
	\end{tikzcd}
	\end{equation}
	where $B$ making the top left triangle commute exists by Lemma \ref{TRANSFLIFT LEMMA}(ii). 
	Moreover, note that while 
	$\alpha^{\**}B$ and $\bar{\mathcal{H}}_{\mathbb{C}_{\{1,2\}}}$
	need not match (cf. Remark \ref{NOTMATCH REM}),
	we nonetheless know that these are homotopic maps
	in the undercategory $\V_{\left(\mathbb{C}_{\{1\}} \amalg \mathbb{C}_{\{2\}}\right)/}$.
	Now consider the following lifting problems, 
	where we note that $\mathbb{C}_{\partial \Delta[2]}$ 
	can be thought of as the informal ``union''
	$\mathbb{C}_{\{1,2\}} \cup \mathbb{C}_{\{0,2\}} \cup \mathbb{C}_{\{0,1\}}$.
	\begin{equation}\label{GOTTALIFT EQ}
	\begin{tikzcd}[column sep = 65pt]
	\mathbb{C}_{\partial \Delta[2]}
	\ar[rightarrowtail]{d}
	\arrow{r}{(\bar{\mathcal{H}}|_{\mathbb{C}_{\{1,2\}}},\bar{H},H)}
	&
	\mathcal{D}(0,0)
	&%
	\mathbb{C}_{\partial \Delta[2]}
	\ar[rightarrowtail]{d}
	\arrow{r}{(\alpha^{\**}B,\bar{H},H)}
	&
	\mathcal{D}(0,0)
	\\
	\mathbb{C}_{\{0,1,2\}} \ar{ru}[swap]{\bar{\mathcal{H}}}
	&
	&%
	\mathbb{C}_{\{0,1,2\}} \ar[dashed]{ru}[swap]{\mathcal{H}}
	&
	\end{tikzcd}
	\end{equation}
	Since $\alpha^{\**}B$ and $\bar{\mathcal{H}}|_{\mathbb{C}_{\{1,2\}}}$
	are homotopic in 
	$\V_{\left(\mathbb{C}_{\{1\}} \amalg \mathbb{C}_{\{2\}}\right)/}$,
	Remark \ref{LEFTQUILUND REM}
	implies that the top maps in \eqref{GOTTALIFT EQ}
	are homotopic in 
	$\V_{\left(
		\mathbb{C}_{\{0,2\}} \amalg_{\mathbb{C}_{\{0\}}} \mathbb{C}_{\{0,1\}}
		\right)/}$ and thus,
	by Remark \ref{UNDFGT REM}, also homotopic in 
	$\V$.
	Lemma \ref{LIFTEQUIV LEM} applied to \eqref{GOTTALIFT EQ}
	now gives the desired lift $\mathcal{H}$.
\end{proof}

\begin{proof}[Proof of Lemma \ref{HOMINPULL LEM}]
	Note first that
	$\mathcal{C} \times_{\mathcal{D}}\bar{\mathcal{C}}$
	is a homotopy pullback
	\cite[Prop. A.2.4.4]{Lur09},
	so we may assume that
	$\mathcal{C},\bar{\mathcal{C}},\mathcal{D}$
	are fibrant and both maps $F,\bar{F}$ are fibrations.
	
	We need to show that $(\alpha,\bar{\alpha})$
	admits left and right inverses up to homotopy.
	By symmetry, we need only address the left inverse case.
	
	Let
	$\beta \colon 1_{\V} \to \mathcal{C}(1,0)$,
	$\bar{\beta} \colon 1_{\V} \to \bar{\mathcal{C}}(1,0)$
	be left homotopy inverses to $\alpha,\bar{\alpha}$,
	with 
	$H \colon \mathbb{C} \to \mathcal{C}(0,0)$,
	$\bar{H} \colon \mathbb{C} \to \bar{\mathcal{C}}(0,0)$
	the homotopies between 
	$id_0$ and $\beta\alpha$, $\bar{\beta}\bar{\alpha}$.
	We now note that, for $\beta$ and $\bar{\beta}$ to induce the desired 
	left homotopy inverse to 
	$(\alpha,\bar{\alpha})$,
	we would need to know not just that 
	$F\beta=\bar{F}\bar{\beta}$
	but also 
	$FH=\bar{F}\bar{H}$.
	The proof will follow by showing that one can modify 
	$\beta,H$ so as to achieve this.
	
	We now consider the diagram below, 
	where $B, \mathcal{H}$ in the bottom square
	are obtained by applying Lemma \ref{HOMTPOFHOMTP LEM}
	to the homotopy equivalence 
	$F\alpha = \bar{F} \bar{\alpha}$ in $\mathcal{D}$,
	its two homotopy left inverses
	$F\beta, \bar{F} \bar{\beta}$,
	and exhibiting homotopies
	$F H, \bar{F} \bar{H}$.
	\begin{equation}\label{BIGISHDIAG EQ}
	\begin{tikzcd}[row sep = 15pt,column sep = 40pt]
	\mathbb{C}_{\{1\}} 	\ar{rr}{\beta} \ar{rd} \ar{dd}&&
	\mathcal{C}(1,0) \ar[twoheadrightarrow]{dd} \ar{dr}{\alpha^{\**}}
	\ar[dashed,leftarrow,bend left =15]{ddll}
	\\
	&
	\mathbb{C}_{\{0,1\}} \arrow[crossing over]{rr}[near end]{H} &&
	\mathcal{C}(0,0)  
	\ar[twoheadrightarrow]{dd}
	\\
	\mathbb{C}_{\{1,2\}} \ar{rr}[swap,near end]{B} \ar{rd} &&
	\mathcal{D}(1,0) \ar{rd}{(F\alpha)^{\**}}
	\\
	&
	\mathbb{C}_{\{0,1,2\}}
	\ar[dashed,bend right =15,crossing over]{uurr}
	\ar{rr}[swap]{\mathcal{H}} \arrow[uu, leftarrow, crossing over] &&
	\mathcal{D}(0,0)
	\end{tikzcd}
	\end{equation}
	We now claim that the curved dashed arrows in 
	\eqref{BIGISHDIAG EQ}
	making the diagram commute exist
	(explicitly, this means the dashed arrows are lifts in the front and back squares, and that the slanted square with two dashed sides commutes).
	To see this, regarding the diagonal $\searrow$ arrows as objects in the arrow category $\V^{\bullet \to \bullet}$, 
	\eqref{BIGISHDIAG EQ}
	is reinterpreted as a square in $\V^{\bullet \to \bullet}$,
	with the desired dashed arrows being precisely a lift of said square.
	But the right vertical arrow in that square
	(i.e. the right side face of \eqref{BIGISHDIAG EQ}) is a projective fibration 
	while the left vertical arrow 
	(i.e. the left side face)
	is a projective trivial cofibration
	(this amounts to the claim that the maps
	$\mathbb{C}_{\{1\}} \overset{\sim}{\rightarrowtail} \mathbb{C}_{\{1,2\}}$ and
	$\mathbb{C}_{\{0,1\}} \amalg_{\mathbb{C}_{\{1\}}} \mathbb{C}_{\{1,2\}} \overset{\sim}{\rightarrowtail} \mathbb{C}_{\{0,1,2\}}$
	are trivial cofibrations),
	so the dashed lifts indeed exist.
	
	Writing 
	$B'\colon \mathbb{C}_{\{1,2\}} \to \mathcal{C}(1,0)$,
	$\mathcal{H}' \colon \mathbb{C}_{\{0,1,2\}} \to \mathcal{C}(0,0)$
	for the lifts,
	we now set
	$\beta' = B'|_{\mathbb{C}_{\{2\}}}$,
	which is a left inverse of $\alpha$
	as exhibited by $H' = \mathcal{H}'|_{\mathbb{C}_{\{0,2\}}}$.
	Moreover, by construction, we now have
	$F\beta' = B|_{\mathbb{C}_{\{2\}}} = \bar{F} \bar{\beta}$
	and 
	$FH' = \mathcal{H}|_{\mathbb{C}_{\{0,2\}}} = \bar{F} \bar{H}$,
	hence
	$(\beta',\bar{\beta})\colon \mathbbm{1} \to
	\mathcal{C} \times_{\mathcal{D}}\bar{\mathcal{C}}$
	now defines the desired left homotopy inverse to 
	$(\alpha,\bar{\alpha}) \colon \mathbbm{1} \to
	\mathcal{C} \times_{\mathcal{D}}\bar{\mathcal{C}}$,
	as exhibited by
	$(H',\bar{H}) \colon \mathbb{C} \to
	\mathcal{C} \times_{\mathcal{D}}\bar{\mathcal{C}}$.
\end{proof}

\providecommand{\bysame}{\leavevmode\hbox to3em{\hrulefill}\thinspace}
\providecommand{\MR}{\relax\ifhmode\unskip\space\fi MR }
\providecommand{\MRhref}[2]{%
  \href{http://www.ams.org/mathscinet-getitem?mr=#1}{#2}
}
\providecommand{\doi}[1]{%
  doi:\href{https://dx.doi.org/#1}{#1}}
\providecommand{\arxiv}[1]{%
  arXiv:\href{https://arxiv.org/abs/#1}{#1}}
\providecommand{\href}[2]{#2}

\makeatletter\@input{labels-AllColors.tex}\makeatother

\end{document}
